\documentclass[12pt]{article}
\usepackage[margin=1in]{geometry}
\usepackage[linesnumbered,ruled,vlined]{algorithm2e}

\usepackage{multirow}
\usepackage{amsthm,amsmath,amssymb}
\usepackage{float,multirow,xcolor}
\usepackage{booktabs,verbatim,natbib}
\usepackage{graphicx}
\usepackage[utf8]{inputenc}

\newcommand{\be}{\begin{equation}}
\newcommand{\ee}{\end{equation}}
\newcommand{\bes}{\begin{equation*}}
\newcommand{\ees}{\end{equation*}}
\newcommand{\beqn}{\begin{eqnarray}}
\newcommand{\eeqn}{\end{eqnarray}}
\newcommand{\beqns}{\begin{eqnarray*}}
\newcommand{\eeqns}{\end{eqnarray*}}

\newcommand{\iter}{\mathrm{iter}}

\newcommand{\diag}{\mbox{diag}}
\newcommand{\supp}{\mbox{supp}}

\newcommand{\Tr}{\mbox{Tr}}
\newcommand{\Sp}{\mbox{Sp}}

\newtheorem{theorem}{Theorem}
\newtheorem{lemma}{Lemma}
\newtheorem{corollary}{Corollary}
\newtheorem{proposition}{Proposition}
\newtheorem{remark}{Remark}

\newcommand{\argmin}{\mathrm{argmin}}

\newcommand{\sign}{\mathrm{sign}}
\newcommand{\calA}{\mathcal{A}}
\newcommand{\ba}{\mathbf{a}}
\newcommand{\bb}{\mathbf{b}}

\newcommand{\bt}{\mathbf{t}}
\newcommand{\bu}{\mathbf{u}}
\newcommand{\bv}{\mathbf{v}}

\newcommand{\bx}{\mathbf{x}}

\newcommand{\bA}{\mathbf{A}}
\newcommand{\bB}{\mathbf{B}}
\newcommand{\bC}{\mathbf{C}}
\newcommand{\bD}{\mathbf{D}}
\newcommand{\bE}{\mathbf{E}}

\newcommand{\bG}{\mathbf{G}}

\newcommand{\bI}{\mathbf{I}}

\newcommand{\bP}{\mathbf{P}}

\newcommand{\bT}{\mathbf{T}}
\newcommand{\bU}{\mathbf{U}}
\newcommand{\bV}{\mathbf{V}}
\newcommand{\bW}{\mathbf{W}}
\newcommand{\bX}{\mathbf{X}}
\newcommand{\bY}{\mathbf{Y}}
\newcommand{\bZ}{\mathbf{Z}}
\newcommand{\bDelta}{\mathbf{\Delta}}
\newcommand{\bSigma}{\mathbf{\Sigma}}

\newcommand{\calB}{{\mathcal{B}}}

\newcommand{\calS}{{\mathcal{S}}}
\newcommand{\Expect}{{\mathbb{E}}}

\newcommand{\psdleq}{\preceq}

\long\def\ignore#1{}

\newcommand{\reals}{\mathbb{R}}

\usepackage{multibib}
\newcites{New}{References}

\title{Theoretical Guarantees for Sparse Principal Component Analysis using the Elastic-Net\thanks{Teng Zhang and Haoyi Yang are co-first authors who contributed equally to this paper. Lingzhou Xue (Email: lzxue@psu.edu) is the corresponding author.}
}
\author{Teng Zhang$^\dag$, Haoyi Yang$^\ddag$, and Lingzhou Xue$^\ddag$ \\ $^\dag$Department of Mathematics, University of Central Florida \\ $^\ddag$Department of Statistics, The Pennsylvania State University}
\date{
First Version: Jan. 2022; This Version: Jan. 2023.
}

\begin{document}

\maketitle

\begin{abstract}
Sparse principal component analysis (SPCA) is widely used for dimensionality reduction and feature extraction in high-dimensional data analysis. Despite many methodological and theoretical developments in the past two decades, the theoretical guarantees of the popular SPCA algorithm proposed by \cite{doi:10.1198/106186006X113430} are still unknown. This paper aims to address this critical gap. We first revisit the SPCA algorithm of \cite{doi:10.1198/106186006X113430} using the elastic net and present our implementation. {We also study a computationally more efficient variant of the SPCA algorithm in \cite{doi:10.1198/106186006X113430} that can be considered as the limiting case of SPCA. We provide the guarantees of convergence to a stationary point for both algorithms and prove that, under a sparse spiked covariance model, both algorithms} can recover the principal subspace consistently under mild regularity conditions. We show that their estimation error bounds  match the best available bounds of existing works or the minimax rates up to some logarithmic factors. Moreover, we demonstrate the competitive numerical performance of both algorithms in numerical studies.
\end{abstract}


	\section{Introduction}
	Principal component analysis (PCA) is one of the most popular tools for dimensionality reduction and feature extraction of multivariate data analysis with $n$ observations and $p$ variables. With the rapid advances in data collection and analysis, contemporary
applications often involve  high-dimensional data sets where the dimension $p$ is comparable or much larger than the sample size $n$, and it has been shown that  PCA yields an inconsistent estimate of the principal directions for high-dimensional data sets \citep{doi:10.1198/jasa.2009.0121,BAIK20061382,10.1214/12-AOS1014,10.2307/25662226,10.1214/08-AOS618,10.2307/24307692,10.1214/009117905000000233}. In addition, a consistent estimation is generally impossible without additional assumptions on the structure of the eigenvector or eigenspace. In view of these negative results, to handle the challenges from high dimensionality, a natural approach is to combine the essence of PCA with the assumption that the phenomena of interest depend mostly on a few variables, i.e., the principal components are sparse.

There has been a considerable effort in the development of new methodology and theory for sparse PCA in the past two decades, either for recovering the leading eigenvector \citep{10.2307/1391037,10.5555/2976248.2976363,doi:10.1198/106186006X113430,doi:10.1137/050645506,10.5555/1390681.1442775,doi:10.1198/jasa.2009.0121,amini2009,10.5555/1756006.1756021,NIPS2013_81e5f81d,10.5555/2567709.2502610,SHEN2013317,6875223,krauthgamer2015,chen2020alternating,jankova2018debiased}, or for recovering multiple eigenvectors and eigenspace \citep{doi:10.1198/106186006X113430,Witten+Hastie+Tibishirani,cai2013,ma2013,NIPS2013_81e5f81d,NIPS2014_5406,lei2015}. 
To analyze algorithms for sparse PCA,  the ``spiked covariance model'' \citep{10.1214/aos/1009210544,doi:10.1198/jasa.2009.0121} is a commonly used model, and it has been used by \cite{cai2013,NIPS2014_5406,ma2013} and \cite{NIPS2013_81e5f81d}. 
In this model, the data matrix $\bX\in\reals^{n\times p}$ is generated by
\begin{equation}\label{eq:spikemodel}
\bX=\bU\diag(\beta_1,\cdots,\beta_r)\bV^T+\bE,
\end{equation}
where $\bU\in\reals^{n\times r}$ is the random effects matrix with entries \emph{i.i.d.} sampled from $N(0,1)$, $\beta_1^2,\cdots, \beta_r^2$ are the leading eigenvalues of the covariance matrix of the noiseless observations satisfying that $\beta_1\geq\beta_2\geq\cdots\geq\beta_r>0$,  $\bV\in\reals^{p\times r}$ is an orthogonal matrix with columns spanning the principal subspace of the observations,  and  $\bE\in\reals^{n\times p}$ represents the additive noise that is \emph{i.i.d.} $N(0,1)$ and  independent of $\bU$.  Equivalently, the rows of $\bX\in\reals^{n\times p}$ are  independently drawn from the multivariate normal distribution $N_p(0,\bSigma)$, where $\bSigma=\bV\diag(\beta_1^2+1,\cdots,\beta_r^2+1,1,\cdots,1)\bV^T\in\reals^{p\times p}$ is the covariance matrix. Following \cite{cai2013,10.1214/13-AOS1151,NIPS2014_5406} and \cite{lei2015}, we make the  ``sparse subspace'' assumption that the union of the support sets  of eigenvectors  $\calS=\cup_{1\leq q\leq r}\supp(\bv_i)$ is sparse, where $\bv_i$ is the $i$-th column of $\bV$. We assume that the size of the union of the support, denoted by $s=|\calS|$, is much smaller than $p$. This property is also called $\ell_0$ row sparsity in  \cite{cai2013}. {Compared with the spiked covariance model in  \cite{ma2013} that use noises from $N(0,\sigma^2)$, we use noises from $N(0,1)$. We remark that they are equivalent with a scaling of $\sigma$.}

Given a data matrix $\bX$ generated from the spiked covariance model, we aim to estimate the $r$-dimensional principal subspace $\Sp(\bV)$. As the principal subspace $\Sp(\bV)$ is uniquely identified with the associated projection matrix $\Pi_{\bV}\in\reals^{p\times p}$, we evaluate the estimation performance of the estimator $\hat{\bV}$ using the following loss function:
\begin{equation}\label{eq:error_measure}
L(\bV,\hat{\bV})=\|\Pi_{\bV}-\Pi_{\hat{\bV}}\|_F,
\end{equation}
where $\|\cdot\|_F$ denotes the Frobenius norm of a matrix. This is a standard measure  used in \cite{cai2013} and \cite{NIPS2013_81e5f81d} as well. $L(\bV,\hat{\bV})$ is zero if and only if the column spaces of $\hat{\bV}$ and $\bV$ are identical, and we refer the reader to Section 2.3 of \cite{cai2013} for a better understanding of this measure.

Now, we provide an overview of existing works on sparse PCA, which can be grouped into two topics: sparse PCA algorithms and the fundamental limits of sparse PCA algorithms. The papers in the first topic proposed sparse PCA algorithms, and the papers in the second topic established the limits such as minimax rates. For a comprehensive review of sparse PCA, we refer the readers to \cite{8412518}.
	
\emph{Sparse PCA algorithms.} Existing works on sparse PCA algorithms can be  grouped into several  categories: subset selection, convex relaxation via semidefinite programming (SDP), iterative thresholding, and Lasso-based methods.

{\bf Lasso-based methods}:   \cite{10.2307/1391037,doi:10.1198/106186006X113430} and \cite{Witten+Hastie+Tibishirani} proposed algorithms based on the lasso (or the elastic net)  and showed encouraging results in both simulation studies and real applications. \cite{jankova2018debiased} proposed a nonconvex, lasso-penalized M-estimator and proved that the estimation error is $O(\sqrt{s\log p/n})$ and nearly optimal when $\beta_1=O(1)$ and $r=1$.

Given a data matrix $\bX\in\reals^{n\times p}$ consisting of $n$ observations $\bx_1,\cdots,\bx_n\in\reals^p$,  \cite{doi:10.1198/106186006X113430} proposed a popular computationally efficient sparse PCA algorithm named the sparse principal component analysis (SPCA): To find the principal components $\bv_1,\cdots,\bv_r$, SPCA algorithm solves the following optimization problem:
	\begin{align}\label{eq:problem_matrix}
	    &(\hat{\bA},\hat{\bB})=\argmin_{\bA,\bB\in\reals^{p\times r},}f(\bA,\bB),\,\,\text{subject to $\bA^T\bA=\bI$},\end{align}
	    where\begin{align}f(\bA,\bB)=\sum_{i=1}^n\|\bx_i-\bA\bB^T\bx_i\|^2+\lambda_0\|\bB\|_F^2+\lambda_{1}\|\bB\|_1,
	\end{align}
where $\lambda_0,\lambda_1>0$ are regularization parameters.	The principal vectors can be estimated from the solution of \eqref{eq:problem_matrix}: let $\hat{\bB}=[\hat{\bb}_1,\cdots,\hat{\bb}_r]$, then $\bv_i$ are proportional to $\hat{\bb}_j$ for $j=1, \cdots, r$. 
	
Then, as shown in \cite{doi:10.1198/106186006X113430}, an alternating algorithm can be used to
solve \eqref{eq:problem_matrix}. However, the theoretical guarantees of SPCA and other Lasso-based methods are less explored in the literature, compared to other categories of sparse PCA algorithms such as \cite{ma2013,cai2013,NIPS2013_81e5f81d} and \cite{lei2015}. In particular, both convergence and statistical properties of
the SPCA algorithm \citep{doi:10.1198/106186006X113430} are still unknown. We aim to close this important theoretical gap in this paper.

	
{\bf Subset selection}: \cite{doi:10.1198/jasa.2009.0121} selected the subset of coordinates with the largest sample variances and performed standard PCA on the selected coordinates. They established the consistency of the proposed method under the high-dimensional regime $p>n$.  
\cite{cai2013} worked on the setting of multiple eigenvectors and introduced an adaptive and efficient procedure based on subset selection. They showed that the proposed method achieves the optimal rates of convergence for estimating the principal subspace $\Sp(\bV)$ over a large range of parameter spaces. \cite{krauthgamer2015}  proposed a covariance thresholding algorithm and showed that it performs well empirically. \cite{NIPS2014_5406} studied this estimator and proved the estimation error of each  principal component. 
	
{\bf Convex relaxation}: \cite{10.5555/1390681.1442775,doi:10.1137/050645506} used a modification of the classical variational representation and derived a semidefinite programming (SDP) approach based on convex relaxation. Assuming $\beta_i=O(1)$ for all $1\leq i\leq r$, \cite{amini2009} and \cite{krauthgamer2015} showed that the estimator is consistent if the sparsity level satisfies $s\leq O(\sqrt{n/\log p})$, and the estimator is inconsistent if $s\geq \Omega(\sqrt{n})$. \cite{NIPS2013_81e5f81d} and  \cite{lei2015} generalized the above method to the setting of multiple eigenvectors, and they established an estimation error of  $O(\frac{\sqrt{\beta_1^2+1}}{\beta_r^2}s\sqrt{\frac{\log p}{n}})$.
	
{\bf Iterative thresholding}: {\cite{doi:10.1198/106186006X113430} pointed out the popular  sparse PCA algorithm in \eqref{eq:problem_matrix} can be computationally expensive when $p\gg n$ since it requires solving Lasso problems iteratively. To improve the computational efficiency for the $p\gg n$ setting such as the gene expression arrays data set, they designed a variant of the algorithm that can be considered as a special case when $\lambda_0\rightarrow\infty$, and its implementation replaced the iterative step of solving Lasso problem with a simple iterative thresholding procedure. Similar to the original SPCA algorithm, the convergence and statistical properties of
this algorithm are also unknown. In this paper, we will close this important theoretical gap too.} 

\cite{ma2013} proposed another iterative thresholding approach and showed that when the principal components $\{\bv_i\}_{i=1}^r$ are sparse in the sense that $\max_{1\leq i\leq r}\|\bv_i\|_p\leq s$ for $0<p<2$, then the proposed approach recovers the principal subspace and leading eigenvectors consistently. We remark that their setting is slightly different as our model assumes that $\bv_i$ are sparse, which is equivalent to the case $p=0$. The other methods in this category generally lack strong theoretical guarantees. \cite{10.5555/1756006.1756021} and \cite{10.5555/2567709.2502610} proposed methods based on iterative power methods and iterative thresholding. \cite{10.5555/2567709.2502610} showed that the estimator has a strong sparse recovery result in estimating a single principal component. \cite{SHEN2013317} proposed a method and showed that with fixed sample size and increasing dimension, in a large set of sparsity assumptions, the proposed algorithm is still consistent while conventional PCA is inconsistent.
{\cite{NIPS2014_74563ba2} proposed a two-stage sparse PCA procedure that combines iterative thresholding with the convex relaxation method in \cite{lei2015} as initialization and simultaneously characterized the computational and statistical performance of this procedure.}

{\bf Other methods}: \cite{10.5555/2976248.2976363} proposed a method based on an alternative spectral formulation and variational eigenvalue
bounds and provided an effective greedy strategy using branch-and-bound search. \cite{6875223} analyzed an
approximate message passing (AMP) algorithm to estimate the
underlying signal, but the focus is slightly different as it assumes that $p/n$ is fixed and aims to achieve the optimal information-theoretically mean squared error rather than the correct sparsity.

\emph{Fundamental limits of sparse PCA algorithms.}  Under the spiked covariance model with $r=1$ and $\beta_1=O(1)$, \cite{amini2009} proved that if $s \geq \Omega(n/\log p)$, then no algorithm, efficient or not, can reliably recover the sparse eigenvector. 
	For the setting of multiple eigenvectors, Theorem 3 of \cite{cai2013} derives an optimal minimax rate of \[\|\Pi_{\hat{\bV}}-\Pi_{\bV}\|_F=O\left(\frac{\sqrt{\beta^2+1}}{\beta^2}\sqrt{\frac{(s-r)\log \frac{e(p-r)}{s-r}+r(s-r)}{n}}\right),\] assuming that $\kappa\beta\geq \beta_1\geq \beta_r\geq\beta$ for some $\kappa=O(1)$. 	Under another set of conditions of the spiked covariance model (see Conditions 1-3 of \cite{10.1214/13-AOS1151}) that assume the ``column-wise $\ell_q$ sparsity'', Theorem 3.2 of \cite{10.1214/13-AOS1151} establishes a minimax error in the order of \[\|\Pi_{\hat{\bV}}-\Pi_{\bV}\|_F=O\left(\frac{\sqrt{\beta_1^2+1}}{\beta_r^2}\sqrt{\frac{s(r+\log\frac{p-r}{s-r})}{n}}\right)\]

\emph{Our motivation and contributions.} As we mentioned, several papers have established the theoretical guarantee of sparse PCA algorithms by showing that the proposed algorithm is consistent in the high-dimensional regime \citep{cai2013,NIPS2014_5406,NIPS2013_81e5f81d,lei2015,ma2013}.  {However, as one of the earliest and the most influential works on sparse PCA, the SPCA algorithm proposed by   \cite{doi:10.1198/106186006X113430} still lacks a proper theoretical guarantee, even though it is simple to implement, requires few parameters, is widely adopted in real applications, and performs well in practice. Therefore, there is an urgent need to clarify the foundational issues of \cite{doi:10.1198/106186006X113430}.}  We aim to provide the essential theoretical guarantees of  \cite{doi:10.1198/106186006X113430} in this paper. The main contributions of our paper to the current literature can be summarized as follows. 

Firstly, this paper investigates the SPCA algorithm proposed by \cite{doi:10.1198/106186006X113430} and provides the algorithmic convergence guarantee to a stationary point under very mild conditions. To the best of our knowledge, the algorithmic convergence result of \cite{doi:10.1198/106186006X113430} is unknown before our work, and we fill this important gap. 

Secondly, we study the statistical properties of the SPCA algorithm \citep{doi:10.1198/106186006X113430}  rigorously under the spiked covariance model. Specifically, under the high-dimensional regime that $p\geq nr(1+\log r)\log p$, the estimation error of the principal subspace is in the order of  $\sqrt{sr\log p}\frac{\sqrt{\beta_1^2+1}}{\beta_r^2\sqrt{n}}$. This rate is on par or better than existing results (up to some logarithmic factors) in \cite{cai2013,ma2013,NIPS2013_81e5f81d,NIPS2014_5406} and \cite{lei2015}. Compared with the fundamental limit of any sparse PCA algorithms  \citep{cai2013,10.1214/13-AOS1151}, this rate is nearly optimal up to a logarithmic factor. To the best of our knowledge, our work closes another important theoretical gap that has existed for over a decade since the statistical properties of the SPCA algorithms \citep{doi:10.1198/106186006X113430} are still unknown.

{Thirdly, we prove that the limiting case of the SPCA algorithm as $\lambda_0\rightarrow\infty$ achieves the same theoretical properties but has a much smaller computational cost in each iteration.}  
We also run experiments with synthetic data to verify our theoretical analysis and  demonstrate the competitive performance of our implementation of the SPCA algorithm.

{Lastly, it is worth noting that our approach to studying the SPCA algorithms of \cite{doi:10.1198/106186006X113430} is novel and distinct from existing proofs presented in the literature, such as those in \cite{NIPS2014_74563ba2} and \cite{ma2013}. 
The idea and strategy of our proofs have the potential to be extended to study the theoretical properties of a broader class of methods in the literature, such as the sparse logistic PCA \citep{lee2010sparse} for binary data and
the sparse exponential-family PCA \citep{LU2016681}, which will be an important future work. Furthermore, our proof development has introduced novel theoretical tools, such as the perturbation of the unitary factor in polar decomposition, as described in Lemma~\ref{lemma:pertubation3}. These tools may have independent significance and utility for future research. }

We close this section with an outline for the rest of this paper. In Section~\ref{sec:alg}, we {revisit the original} SPCA algorithm  and its variant {as $\lambda_0\to\infty$}. The algorithmic convergence guarantees of both algorithms are also presented in Section~\ref{sec:alg}. In Section~\ref{sec:analysis}, we analyze their statistical properties and show that under the sparse spiked covariance model~\eqref{eq:error_measure}, both methods recover the principal subspace consistently in a large range of settings. Finally, the numerical simulations in Section~\ref{sec:numerical} verify our theoretical findings and show that both algorithms have promising numerical performances compared with existing methods that also enjoy theoretical guarantees.  The proofs are presented in Section 5. Section~\ref{sec:conclusion} includes a few concluding remarks.

\section{SPCA Algorithms and Convergence Guarantees}\label{sec:alg}

In this section, we present SPCA algorithms and prove their convergence guarantees. In Subsection~\ref{subsec:spca}, we revisit the original SPCA algorithm of \cite{doi:10.1198/106186006X113430} and then present our implementation. In Subsection~\ref{subsec:itps}, 
{we study  a variant of the SPCA algorithm as $\lambda_0\rightarrow\infty$ (see Section 4 of \cite{doi:10.1198/106186006X113430}) and then present our implementation. } The convergence guarantees of both  SPCA algorithms are provided in Subsection~\ref{subsec:convergence}.

\subsection{{The original SPCA algorithm}}\label{subsec:spca}

Let us first revisit the SPCA algorithm proposed by \cite{doi:10.1198/106186006X113430}. The SPCA algorithm uses an iterative approach to solve $\bA$ and $\bB$ in an alternating fashion from the optimization problem \eqref{eq:problem_matrix}. When $\bB$ is fixed, the update of $\bA$ can be rewritten as a reduced rank Procrustes rotation problem \citep{doi:10.1198/106186006X113430}. If $\bA$ is fixed,  the update of $\bB$ is convex and can be solved as $r$ independent elastic net problems. In summary, the iterative updates of the SPCA  algorithm are as follows:
	\begin{align}\label{eq:update}
	    \bA^{(k+1)}&={\bX^T\bX\bB^{(k)}}(\bB^{(k)\,T}\bX^T\bX\bX^T\bX\bB^{(k)})^{-\frac{1}2}\\
	    \bB^{(k+1)}&=\argmin_{\bB\in\reals^{p\times r}}\|\bX(\bB-\bA^{(k+1)})\|_F^2+\lambda_0\|\bB\|_F^2+\lambda_1\|\bB\|_1.\label{eq:updateB}
	\end{align}

Given the iterative updates ~\eqref{eq:update}-\eqref{eq:updateB}, the SPCA algorithm still requires an initial estimation of $\bB\in\reals^{p\times r}$, which should be a good guess of the principal subspace  $\bV\in\reals^{p\times r}$. The R-package {\em elasticnet} \citep{enet} implemented the SPCA algorithm of \cite{doi:10.1198/106186006X113430} and  used  the singular value decomposition (SVD) of $\bX$ for this initial estimation. More specifically, given the SVD of $\bX=\bU_{X}\bSigma_{X}\bV_{X}^T$, where $\bU_{X}\in\reals^{m\times k}$, $\bV_{X}\in\reals^{n\times k}$, $\bSigma_{X}\in\reals^{k\times k}$, \cite{enet} directly takes $\hat{\bV}$ as initialization. As an alternative, we apply the diagonal thresholding procedure~\citep{doi:10.1198/jasa.2009.0121}, a simple sparse PCA algorithm that first estimates the support set $\calS$ by
\[
\hat{\calS}=\{1\leq i\leq p: \sum_{j=1}^n X_{ij}^2>C_{thr}\},
\]
where $C_{thr}$ is a thresholding parameter. Then, the procedure uses the top $r$ right singular vectors of  $\bX_{\hat{\calS}}$ to estimate $\bv_1,\cdots,\bv_r$. For our initialization, we generate $\bB^{(0)}$ by combining the top $r$ right singular vectors of  $\bX_{\hat{\calS}}$. we recommend using a large enough $\lambda_0$, and we will show the reason in the theoretical part.

For technical reasons to be clarified in Section \ref{sec:analysis},  
we follow existing theoretical works on sparse PCA such as \cite{cai2013} and \cite{NIPS2014_5406} to split the data set into two subsets, in which the initialization $\bB^{(0)}$ is  generated from one subset, and the iterative updates are generated from the other subset until convergence. Combining this initialization with the iterative updates, we summarize our implementation in Algorithm \ref{alg:spca1}.

\begin{algorithm}[H]
\caption{{{The SPCA algorithm 
}
}}  
\label{alg:spca1}
\begin{flushleft} 
	{\bf Input:}  data matrix $\bX\in\reals^{n\times p}$, rank $r$, thresholding parameter $C_{thr}$, and $\lambda_0$ \& $\lambda_1$.\\
	{\bf Output:} An estimation of $\bV$.\\
	{\bf Steps:}\\
	{\bf 1:} Divide the data set $\bX$ into two subsets $\bX_{(1)}\in\reals^{n_1\times p}$ and $\bX_{(2)}\in\reals^{n_2\times p}$ with similar sizes, i.e., $n_1,n_2\approx n/2$.\\ 
	{\bf 2:} Apply the diagonal thresholding \citep{doi:10.1198/jasa.2009.0121} {or the Fantope  \citep{NIPS2013_81e5f81d}} on the data set $\bX_{(1)}$ to obtain an initial estimate $\bB^{(0)}\in\reals^{p\times r}$.\\
	{\bf 3:} Let $\bB^{(0)}$ be the initialization and apply the iterative updates 
	\eqref{eq:update}-\eqref{eq:updateB} on the data set $\bX_{(2)}$ to obtain the sequence $\{(\bA^{(k+1)},\bB^{(k+1)})\}$, $k=0,1,2,\ldots,$ until convergence.
	\\
	{\bf 4:} Output $\hat{\bV}=\lim_{\iter\rightarrow\infty}{\bB^{(\iter)}}$ as an  estimation of $\bV$.
	\end{flushleft} 
	\end{algorithm}

We would like to point out that, compared with \cite{doi:10.1198/106186006X113430} and the R-package {\em elasticnet} \citep{enet}, another major difference is the choice of regularization parameter $\lambda_0$ in the iterative updates ~\eqref{eq:update}-\eqref{eq:updateB}. \cite{doi:10.1198/106186006X113430} and the R-package {\em elasticnet} \citep{enet} suggested to use a small $\lambda_0$ and set $10^{-6}$ as the default choice. {The setting $\lambda_0\rightarrow\infty$ is only recommended by \cite{doi:10.1198/106186006X113430} for large data sets (e.g., gene expression arrays) to reduce the computational cost. 
} However, our implementation suggests to use a $\lambda_0$ that is large enough. Our theoretical analysis in Section \ref{sec:analysis} will provide a justification for such a choice of $\lambda_0$. The statistical properties in Section \ref{sec:analysis} imply that SPCA with a large $\lambda_0$ works under the spiked covariance model. In practice, we observe that our implementation works well with $\lambda_0=10^{6}$. Please see the simulation studies in Section \ref{sec:numerical} for more details. {It is worth noting that our implementation of the sparse principal component analysis (SPCA) algorithm omits the data splitting scheme in step 1 in Algorithm~\ref{alg:spca1}. This is because the splitting scheme is only introduced to facilitate certain technicalities in the proof, as discussed in Remark 3 in Subsection \ref{sec:main}. Similarly, our implementation of the Fantope-based initialization method utilizes the top eigenvectors of the Fantope solution as the initialization. This deviates slightly from the Fantope-based thresholding method analyzed theoretically in this work, as discussed before the statement of Lemma~\ref{prop:init}.} 

\subsection{{{The limiting case of the SPCA algorithm}}}\label{subsec:itps}

{
It is of note that the subproblem \eqref{eq:updateB} in the SPCA algorithm is a LASSO-type problem and requires an iterative solver such as the elastic net \cite{zou2005regularization}. 
\cite{doi:10.1198/106186006X113430} pointed out the potentially expensive computational cost of the original SPCA algorithm for data with thousands of variables in the analysis of gene expression arrays,} {and Section 4 of \cite{doi:10.1198/106186006X113430} showed that the limiting case of the SPCA algorithm when $\lambda_0\rightarrow\infty$ has a simpler iterative thresholding procedure and requires a smaller computational cost.}


{For the completeness of the paper, here we derive the iterative thresholding procedure in Section 4 of \cite{doi:10.1198/106186006X113430} explicitly.} Note that  \eqref{eq:updateB} can be written as
\[
\bB^{(k+1)}=\argmin_{\bB\in\reals^{p\times r}}\|\bX\bB\|_F^2-2\mathrm{tr}(\bA^{(k+1)T}\bX^T\bX\bB)+\lambda_0\|\bB\|_F^2+\lambda_1\|\bB\|_1,
\]
where $\|\bX\bB\|_F^2$ and $\lambda_0\|\bB\|_F^2$ are both quadratic terms about $\bX$. When $\lambda_0\rightarrow\infty$, $\|\bX\bB\|_F^2$ is dominated by $\lambda_0\|\bB\|_F^2$, and   \eqref{eq:updateB} can be  approximately by 
\[
\bB^{(k+1)}=\argmin_{\bB\in\reals^{p\times r}}-2\mathrm{tr}(\bA^{(k+1)T}\bX^T\bX\bB)+\lambda_0\|\bB\|_F^2+\lambda_1\|\bB\|_1,
\]
and 
each entry of the solution $\bB^{(k+1)}$, i.e., $[\bB^{(k+1)}]_{ij}$, is solved by
the soft-thresholding function
$\frac{1}{\lambda_0}S([\bX^T\bX\bA^{(k+1)}]_{ij},\lambda_1/2)$,
where  $S(x,a)=\sign(x)\max(|x|-a,0)$. In summary, we have the following formula at each update: \begin{align}\label{eq:update2} \bA^{(k+1)}&={\bX^T\bX\bB^{(k)}}(\bB^{(k)\,T}\bX^T\bX\bX^T\bX\bB^{(k)})^{-\frac{1}2}\\
[\bB^{(k+1)}]_{ij}&=S([\bX^T\bX\bA^{(k+1)}]_{ij},\lambda_1/2),\,\,\,\text{for all $1\leq i\leq p, 1\leq j\leq r$}.\label{eq:updateB2}
\end{align}

Compared with \eqref{eq:update}-\eqref{eq:updateB} in Subsection 2.1, the iterative updates \eqref{eq:update2}-\eqref{eq:updateB2} are easier to calculate since the formula is explicit and does not require solving an elastic-net problem. {For ease of presentation, we refer to this SPCA algorithm as $\lambda_0\rightarrow\infty$ in   \eqref{eq:update2}-\eqref{eq:updateB2} as the iterative thresholding principal subspace (ITPS) algorithm throughout the rest of this paper.} We remark that the ITPS algorithm differs from the various iterative thresholding procedures such as \citep{10.5555/1756006.1756021,10.5555/2567709.2502610,SHEN2013317,ma2013} and others. 

We write down the complete procedure for our implementation of the ITPS algorithm in Algorithm~\ref{alg:spca}, where we again follow \cite{cai2013} and \cite{NIPS2014_5406} to split the data set into two subsets, in which $\bB^{(0)}$ is  generated from one subset and the iterative updates are generated from the other subset until convergence.

\begin{algorithm}
\caption{{{The ITPS algorithm}}}  
\label{alg:spca}
\begin{flushleft} 
	{\bf Input:}  data matrix $\bX\in\reals^{n\times p}$, rank $r$, thresholding parameter $C_{thr}$, and $\lambda_1$.\\
	{\bf Output:} An estimation of $\bV$.\\
	{\bf Steps:}\\
	{\bf 1:} Divide the data set $\bX$ into two subsets $\bX_{(1)}\in\reals^{n_1\times p}$ and $\bX_{(2)}\in\reals^{n_2\times p}$ with similar sizes, i.e., $n_1,n_2\approx n/2$.\\ 
	{\bf 2:} Apply the diagonal thresholding  \citep{doi:10.1198/jasa.2009.0121} {or the Fantope  \citep{NIPS2013_81e5f81d}} on the data set $\bX_{(1)}$ to obtain an initial estimate $\bB^{(0)}\in\reals^{p\times r}$.\\
	{\bf 3:} Let $\bB^{(0)}$ be the initialization and apply the iterative update 
	\eqref{eq:update2}-\eqref{eq:updateB2} on the data set $\bX_{(2)}$ to obtain the sequence $\{(\bA^{(k+1)},\bB^{(k+1)})\}$, $k=0,1,2,\ldots,$  until convergence.
	\\
	{\bf 4:} Output $\hat{\bV}=\lim_{\iter\rightarrow\infty}{\bB^{(\iter)}}$ as an  estimation of $\bV$.
	\end{flushleft} 
	\end{algorithm}

\subsection{Algorithmic convergence guarantees} \label{subsec:convergence}

Recall that \cite{doi:10.1198/106186006X113430} computed the iterative updates ~\eqref{eq:update}-\eqref{eq:updateB} in an alternating minimization fashion when solving the nonconvex manifold optimization problem \eqref{eq:problem_matrix}. It is an important but challenging research topic to study the convergence of alternating minimization algorithms for nonconvex problems. By adding the proximal terms, \cite{attouch2010proximal} provided the convergence guarantees of proximal alternating minimization algorithms based on the Kurdyka--{\L}ojasiewicz inequality. 
However, the convergence guarantees of \cite{attouch2010proximal} are established when solving the unconstrained nonconvex problems and thus they do not apply to the  constrained nonconvex problem \eqref{eq:problem_matrix}. 
To the best of our knowledge, there is still no convergence guarantee of the SPCA algorithm proposed by \cite{doi:10.1198/106186006X113430} in the literature. In the sequel, we will provide a direct proof for the algorithmic convergence guarantees of both SPCA and ITPS algorithms without adding the proximal terms.


First, we have the following result on the algorithmic convergence of  the SPCA algorithm:

\begin{theorem}\label{thm:main3a}[Convergence of the SPCA algorithm] The sequence $\{(\bA^{(k)},\bB^{(k)})\}$ generated by the iterative updates ~\eqref{eq:update}-\eqref{eq:updateB} of the SPCA algorithm has at least one limiting point for any $\lambda_0>0$. Moreover, if there is a limiting point  $(\bar{\bA}_{\mathrm{SPCA},\lambda_0},\bar{\bB}_{\mathrm{SPCA},\lambda_0})$ satisfying that $\bar{\bB}_{\mathrm{SPCA},\lambda_0}^T\bX^T\bX\bX^T\bX\bar{\bB}_{\mathrm{SPCA},\lambda_0}$ is invertible (i.e., has rank $r$), then the SPCA algorithm converges to $(\bar{\bA}_{\mathrm{SPCA},\lambda_0},\bar{\bB}_{\mathrm{SPCA},\lambda_0})$, which a stationary point of the constrained  problem \eqref{eq:problem_matrix}.
\end{theorem}

Theorem~\ref{thm:main3a} proves the convergence of  $\{(\bA^{(k)},\bB^{(k)})\}$ generated by the iterative updates ~\eqref{eq:update}-\eqref{eq:updateB} of the SPCA algorithm to a stationary point of the constrained problem \eqref{eq:problem_matrix} under some very mild conditions. This result provides the essential convergence guarantee for the popular SPCA algorithm \citep{doi:10.1198/106186006X113430} and also our implementation in Algorithm~\ref{alg:spca}. Thus, Theorem~\ref{thm:main3a} fills this important gap that exists in over a decade.  
The proof of Theorem~\ref{thm:main3a} is presented in Section 5. It is worth pointing out that we do not use the Kurdyka-{\L}ojasiewicz inequality to prove Theorem ~\ref{thm:main3a}. The proof of Theorem~\ref{thm:main3a} mainly follows from the observation that the objective function is nonincreasing with the iterative updates that minimize $\bA$ and $\bB$ alternatively, and therefore the objective value must converge. It follows that the sequence $\{(\bA^{(k)},\bB^{(k)})\}$ does not diverge and has at least one limiting point. In addition, the assumption on the invertibility of  $\bar{\bB}_{\mathrm{SPCA},\lambda_0}^T\bX^T\bX\bX^T\bX\bar{\bB}_{\mathrm{SPCA},\lambda_0}$ guarantees that the algorithm is well-defined at the limiting point (note that the update of $\bA$ involves a matrix inversion). This condition is usually satisfied in practice unless when $\lambda_1$ is too large and $\bB^{(k+1)}$ becomes almost zero after solving \eqref{eq:updateB} or applying \eqref{eq:updateB2}. 

Next, given the convergence result in Theorem~\ref{thm:main3a}, we study the algorithmic convergence of the iterative updates \eqref{eq:update2}-\eqref{eq:updateB2} of the ITPS algorithm, which can be considered as the limiting case of the SPCA algorithm.  Although the solution \eqref{eq:updateB} goes to zero as $\lambda_0\rightarrow\infty$, this scaling issue can be resolved by using 
\begin{align*}
\tilde{f}(\bA,\bB)= & \ \lim_{\lambda_0\rightarrow\infty}\Big(\lambda_0f(\bA,\bB/\lambda_0)-\lambda_0\sum_{i=1}^n\|\bx_i\|^2\Big)\\
=& \ -2\sum_{i=1}^n\bx_i^T\bA\bB^T\bx_i+\|\bB\|_F^2+\lambda_1\|\bB\|_1
\end{align*}
and applying an alternating minimization algorithm to solve \begin{equation}\argmin_{\bA,\bB:\bA^T\bA=\bI}\tilde{f}(\bA,\bB).\label{eq:problem_ITPS}
\end{equation}
Before proceeding, we define $\mathcal{N}_\epsilon(\bar{\bB})=\{\bB: \|\bB-\bar{\bB}\|_F\leq \epsilon\}$ as the $\epsilon$-neighborhood of $\bar{\bB}$, and define the function $$g(\bB):=f(\bX^T\bX{\bB}({\bB}^T\bX^T\bX\bX^T\bX{\bB})^{-\frac{1}2},\bB)$$ when ${\bB}^T\bX^T\bX\bX^T\bX{\bB}$ is invertible (i.e., has rank $r$). 

\begin{theorem}\label{thm:main3b}[Convergence of the ITPS algorithm]
\begin{itemize}
    \item [(a)] The sequence $\{(\bA^{(k)},\bB^{(k)})\}$ generated by the iterative updates \eqref{eq:update2}-\eqref{eq:updateB2} of the ITPS algorithm has at least one limiting point. If there is a limiting point  $(\bar{\bA}_{\mathrm{ITPS}},\bar{\bB}_{\mathrm{ITPS}})$ satisfying that $\bar{\bB}_{\mathrm{ITPS}}^T\bX^T\bX\bX^T\bX\bar{\bB}_{\mathrm{ITPS}}$ is invertible (i.e., has rank $r$), then the ITPS algorithm converges to $(\bar{\bA}_{\mathrm{ITPS}},\bar{\bB}_{\mathrm{ITPS}})$, which is a stationary point of the constrained problem \eqref{eq:problem_ITPS}.
    \item [(b)] Under the same conditions of (a),  there exists $\epsilon_0>0$ such that  ${\bB}^T\bX^T\bX\bX^T\bX{\bB}$ is invertible (i.e., has rank $r$) in $\mathcal{N}_{\epsilon_0}(\bar{\bB}_{\mathrm{ITPS}})$. Moreover,  when $\lambda_0$ is large enough, there exists a unique local minimizer of the function $g(\bB)$, denoted by $\bar{\bB}_{\lambda_0}$, {such that $\lambda_0\bar{\bB}_{\lambda_0}\in \mathcal{N}_{\epsilon_0}(\bar{\bB}_{\mathrm{ITPS}})$.} Let $\bar{\bA}_{\lambda_0}=\bX^T\bX\bar{\bB}_{\lambda_0}(\bar{\bB}_{\lambda_0}^T\bX^T\bX\bX^T\bX\bar{\bB}_{\lambda_0})^{-\frac{1}2}$. Then, $(\bar{\bA}_{\lambda_0},\bar{\bB}_{\lambda_0})$ is a stationary point of the constrained problem \eqref{eq:problem_matrix}, and we have
      $$\lim_{\lambda_0 \rightarrow\infty}(\bar{\bA}_{\lambda_0},{\lambda_0\bar{\bB}_{\lambda_0}})=(\bar{\bA}_{\mathrm{ITPS}},\bar{\bB}_{\mathrm{ITPS}}).$$
\end{itemize}
\end{theorem}

Theorem~\ref{thm:main3b} provides the theoretical justification that the ITPS algorithm can be considered as the limiting case of the SPCA algorithm when $\lambda_0\rightarrow\infty$. Part (a) of Theorem~\ref{thm:main3b} studies the convergence of  $\{(\bA^{(k)},\bB^{(k)})\}$ generated by the iterative updates \eqref{eq:update2}-\eqref{eq:updateB2} of the ITPS algorithm to a stationary point of the constrained problem \eqref{eq:problem_ITPS} under some mild conditions, and its proof is similar to that of Theorem~\ref{thm:main3a}. 
Part (b) of Theorem~\ref{thm:main3b} connects the stationary point of the constrained problem \eqref{eq:problem_ITPS}, i.e. $(\bar{\bA}_{\mathrm{ITPS}},\bar{\bB}_{\mathrm{ITPS}})$,  to a sequence of stationary points of the original constrained problem \eqref{eq:problem_matrix} for SPCA when $\lambda_0\rightarrow\infty$. To show this connection, we prove the strong convexity of $g(\bB)$ in the $\epsilon$-neighborhood of $\bar{\bB}_{\mathrm{ITPS}}$ when $\lambda_0$ is large enough and then carefully construct the sequence $\{(\bar{\bA}_{\lambda_0},{\bar{\bB}_{\lambda_0}})\}$ such that they are stationary points of the original constrained problem \eqref{eq:problem_matrix} and $\lim_{\lambda_0 \rightarrow\infty}(\bar{\bA}_{\lambda_0},{\lambda_0\bar{\bB}_{\lambda_0}})=(\bar{\bA}_{\mathrm{ITPS}},\bar{\bB}_{\mathrm{ITPS}}).$ The complete proof of Theorem~\ref{thm:main3b} is presented in Section 5.

{We remark that while  \cite{doi:10.1198/106186006X113430} already showed that the minimizer of the ITPS objective function  converges to the minimizer of the SPCA objective function (see Theorem 5 of \cite{doi:10.1198/106186006X113430}), it does not prove the convergence of either SPCA or ITPS algorithms. Moreover, we will make several additional observations about the convergence guarantee of the ITPS algorithm as follows. First, we show that the ITPS algorithm converges to meaningful results in Theorem~\ref{thm:main3b}(a). Second, we show that the stationary point of the ITPS algorithm converges to the stationary point of the SPCA algorithm in Theorem~\ref{thm:main3b}(b). Finally, the choice of $\lambda_0\rightarrow\infty$ can also be justified by our theoretical analysis in Section \ref{sec:analysis}. Our analysis shows that while the ITPS algorithm was originally proposed in \cite{doi:10.1198/106186006X113430} to handle the $p\gg n$ setting numerically, it actually performs as well as the original SPCA algorithm under a wide range of scenarios.}

\section{Statistical Properties}\label{sec:analysis}
	
This section is devoted to studying the statistical properties of both SPCA and ITPS algorithms  under the spiked covariance model \eqref{eq:error_measure}. We present the main results in Subsection~\ref{sec:main}. After some preliminary assumptions, we first establish the convergence rates for subspace estimation error of SPCA and ITPS algorithms with a good initialization in Theorem~\ref{thm:main}. Then we introduce the subspace estimation error of Algorithm~\ref{alg:spca1} and Algorithm~\ref{alg:spca}, i.e., SPCA and ITPS algorithms with the diagonal thresholding as initialization, in Theorem~\ref{thm:main2}. We discuss the main results and compare them with existing works in Subsection~\ref{subsec:discussion}. Moreover, we include a sketch of the proof in Subsection~\ref{sec:mainproof}.

Now we define the notations to be used in this section. Assume that  $\calB$ is a subset of $ \{1,\cdots,n\}$, then $\calB^c$ is the complementary set. For $\bx\in\reals^n$ and $\bx_i\in\reals^n$, we use $\bx_{\calB}$ and $\bx_{i,\calB}$  to denote the subvectors of $\bx$ and $\bx_i$ indexed by $\calB$. For $\bX\in\reals^{n\times p}$, we use  $\bX_{\calB}$ to represent the submatrix of $\bX$ with rows indexed by $\calB$.   For any matrix $\bX\in\reals^{m\times n}$ with rank $k$, we let  $\|\bX\|_1=\sum_{i,j}| X_{ij}|$ be its  $\ell_1$ norm, $\|\bX\|$ be its spectral norm, $\|\bX\|_F$ be its Frobenius norm, and $\|\bX\|_{\infty}$ be its elementwise matrix norm: $\|\bX\|_{\infty}=\max_{i,j}| X_{ij}|$.  In addition, we let $\sigma_i(\bX)$ be the $i$-th largest singular value of $\bX$. Given the SVD that $\bX=\bU_{X}\bSigma_{X}\bV_{X}^T$ with $\bU_{X}\in\reals^{m\times k}$, $\bV_{X}\in\reals^{n\times k}$, and $\bSigma_{X}\in\reals^{k\times k}$, we use $\Pi_{X}\in\reals^{m\times m}$ to denote the projector to the column space of $\bX$, i.e., $\Pi_{X}=\bU_{X}\bU_{X}^T$, and $\Pi_{\bX,r}\in\reals^{m\times m}$ be the projector to the span of the first $r$ left singular vectors of $\bX$, i.e., the first $r$ columns of $\bU_{X}$. Throughout the rest of this paper, we use $C$, $C_0$, $C_1$, $C_2$, $c$, and $c_0$ to represent constants that do not depend on $n$, $p$, $s$,  and $r$. The values of these constants can change from
line to line. 

\subsection{Main results}\label{sec:main}

The main results in this subsection aim to provide the essential statistical properties for the popular SPCA algorithm \citep{doi:10.1198/106186006X113430} and also for the ITPS algorithm.

We first introduce $\kappa$, a key quantity in the estimation of the principal subspace,  defined by
\[
\kappa:=\kappa(n,s,\beta)=\frac{\sqrt{(\beta_1^2+1)s}}{\beta_r^2 \sqrt{n}}.
\]

Recall that the principal subspace $\Sp(\bV)$ is uniquely identified with the associated projection matrix $\Pi_{\bV}\in\reals^{p\times p}$, and $L(\bV,\hat{\bV})=\|\Pi_{\bV}-\Pi_{\hat{\bV}}\|_F$ measures the error of the estimator $\tilde{\bV}$. Let $\Pi_{\bX_{\calS}}$ be the projection matrix of {the oracle estimator} that knows the sparsity pattern $\calS$  in advance. We remark that $\kappa$ is an upper bound of the estimation error  $\|\Pi_{\bV}-\Pi_{\bX_{\calS}}\|_F$, as shown in the proof of main results (see  \eqref{eq:eigenspacebound1} in Section 5). This also suggests that our problem is difficult when $\kappa=O(1)$: under this setting the oracle estimator does not outperform a random estimator in terms of the order of the operator norm of the estimation error, because for any estimator $\tilde{\bV}$,  $\|\Pi_{\bV}-\Pi_{\tilde{\bV}}\|\leq\|\Pi_{\bV}\|+\|\Pi_{\tilde{\bV}}\|={2}$. In addition, the component $\frac{\sqrt{(\beta_1^2+1)}}{\beta_r^2}$ is called the ``effective noise variance'' in \cite{10.1214/13-AOS1151} and the ``parametric term'' in \cite{ma2013}. This component is also used in the theoretical analysis of \cite{NIPS2013_81e5f81d}, where the upper bound of the estimation error is $O(\frac{\sqrt{\beta_1^2+1}}{\beta_r^2}s\sqrt{\frac{\log p}{n}})$.

We now state the main assumptions for our theoretical results:

\noindent
\textbf{(C0).} [Condition on dimensionality and the size of the support set] {$n\ge c_0s$ for some $c_0>0$, $p\ge 5 s$, and $\log p = o(n)$.
}

\noindent
\textbf{(C1).} [Condition on the strength of the signal] $\kappa\leq \kappa_*$ for some $\kappa_*>0$. 

\noindent
\textbf{(C2).} [Condition on initialization]  The initialization $\bB^{(0)}$ has the correct sparsity pattern in the sense that $\bB^{(0)}_{\calS^c}=0$, and $\|\Pi_{\bB^{(0)}}-\Pi_{\bV}\|_F\leq M_0/\sqrt{s}$ for some $M_0>0$. In addition, $\bB^{(0)}$ is independent of the matrix $\bX$ used in the iterative update.

 \noindent
\textbf{(C3).} [Condition on regularization parameters]
\begin{enumerate}
 \item [(a)] There exist two constants $M_1,M_2>0$ such that 
 $$
 M_1\Big(\sqrt{\log p(n(\beta_1^2+1)+p)} + n\log^{\frac{3}{2}} p\frac{\sqrt{pr}+r(1+\log r)\sqrt{\beta_1^2+1}}{p}\Big)
 \le \lambda_1
 $$
 and 
 $$\lambda_1\leq M_2\frac{\min(\beta_r^2,1)\sqrt{(\beta_r^2+1)n((\beta_r^2+1)n+p)}}{\sqrt{sr}}.$$
 
\item [(b)] 	
	(for the SPCA algorithm only) $\lambda_0>4\Big(n(\beta_r^2+1)s+p\Big)$
	\end{enumerate}

In condition \textbf{(C0)}, we assume that  the dimension $p$ grows at a sub-exponential rate of the sample size $n$, and the size $s$ of the support set is no larger than the order of $n$. In condition \textbf{(C1)}, we  assume that $\kappa$ is not large, which is necessary: as we argued before, the oracle estimator does not outperform a random estimator if $\kappa$ is larger than $2$. 
  
In order to investigate the iterative updates of SPCA and ITPS, we also assume that the initialization $\bB^{(0)}$ is ``nice'' in condition \textbf{(C2)}. We will establish the theoretical justification of the initialization by the diagonal thresholding to satisfy \textbf{(C2)}. 

Next, we explain the condition \textbf{(C3)} on regularization parameters. The condition \textbf{(C3)}(a) gives the upper and lower bound on the parameter $\lambda_1$ for both SPCA and ITPS algorithms, and \textbf{(C3)}(b) gives the lower bound on the parameter $\lambda_0$ for the SPCA algorithm only. 

In condition \textbf{(C3)}(a), the purpose of the lower bound of $\lambda_1$ is to make sure that there is sufficient thresholding or $\ell_1$ penalization such that the correct sparsity is achieved. The purpose of the upper bound of $\lambda_1$ is to make sure that 
the update formula in \eqref{eq:updateB} or \eqref{eq:updateB2} would not return a zero matrix. 

Also, we would like to point out that the lower bound and the upper bound of $\lambda_1$ in condition \textbf{(C3)}(a) implicitly assume that the lower bound is no larger than the upper bound:
\begin{align*}
  &  M_1\Big(\sqrt{\log p(n(\beta_1^2+1)+p)} + n\log^{\frac{3}{2}} p\frac{\sqrt{pr}+r(1+\log r)\sqrt{\beta_1^2+1}}{p}\Big)\\
\le & M_2\frac{\min(\beta_r^2,1)\sqrt{(\beta_r^2+1)n((\beta_r^2+1)n+p)}}{\sqrt{sr}}.
\end{align*}
which holds when $\beta_r,s,r$ are fixed and $n,p\rightarrow\infty$.

We remark that the condition \textbf{(C3)}(b) can be removed at the expense of a slightly weaker bound in the estimation error (see the discussion about the choice of parameters in Subsection~\ref{subsec:discussion}). The lower bound of $\lambda_0$ is due to technicalities in the proof.
	
	The sequence $\{(\bA^{(k)},\bB^{(k)})\}$ is generated by the iterative updates \eqref{eq:update}-\eqref{eq:updateB} of the SPCA algorithm or \eqref{eq:update2}-\eqref{eq:updateB2} of the ITPS algorithm. Given the algorithmic convergence guarantee of  $\{(\bA^{(k)},\bB^{(k)})\}$ (see Theorems \ref{thm:main3a} and \ref{thm:main3b}), we know that 	$\lim_{\iter\rightarrow\infty}{\bB^{(\iter)}}$ exists. 	The following theorem shows that under the conditions \textbf{(C0)--(C3)}, $\lim_{\iter\rightarrow\infty}{\bB^{(\iter)}}$ can recover
	the principal subspace from the data matrix $\bX$ with high probability.

	\begin{theorem}\label{thm:main}[Estimation errors of SPCA and ITPS with a good initialization]
Under the conditions \textbf{(C0)--(C3)}, for  the sequence $\bB^{(\iter)}$ generated by \eqref{eq:update}-\eqref{eq:updateB} of the SPCA algorithm or \eqref{eq:update2}-\eqref{eq:updateB2} of the ITPS algorithm, there exist constants $C_0, C_1>0$ 
such that we have
	\begin{equation}\label{eq:error_main}
\lim_{\iter\rightarrow\infty}\Big\|{\Pi_{\bB^{(\iter)}}}-\Pi_{\bV}\Big\|_F\leq C_1\kappa\sqrt{r}\Big(\frac{\lambda_1}{\sqrt{n(\beta_1^2+1)+p}}\Big)+C_1\kappa\sqrt{\frac{t}{s}},
	\end{equation}
	where $t$ can be chosen arbitrarily, under an event $E_1$ whose probability $\Pr(E_1)$ is at least 
 \begin{equation} 1-C_0/p-C_0pe^{-n/C_0}-C_0pe^{-p/C_0}-e^{-s/2\kappa^2}-6e^{-\frac{\min(\sqrt{ps},\sqrt{ns})^2}{C_0}}-4e^{-\frac{s}{C_0\kappa^2}}-2e^{-t}.\label{eq:prob_E}
	\end{equation}  
	\end{theorem}

Theorem~\ref{thm:main} fills an important theoretical gap that has persisted for over a decade since there is no statistical property established for the popular SPCA algorithm \citep{doi:10.1198/106186006X113430} yet to the best of our knowledge. The proof of Theorem~\ref{thm:main} is deferred to the Section 5 after introducing some technical lemmas, and the complete proof of these lemmas is presented in Section 5. 

\begin{remark}
 Theorem~\ref{thm:main} also holds when the conditions \textbf{(C2)} and \textbf{(C3)} are replaced with the following less restrictive conditions. 

\noindent
\textbf{(C2').} [Condition on initialization]  The initialization $\bB^{(0)}$ has the correct sparsity pattern in the sense that $\bB^{(0)}_{\calS^c}=0$, and $\|\Pi_{\bB^{(0)}}-\Pi_{\bV}\|_F\leq c\min({\beta_r},1)$. In addition, $\bB^{(0)}$ is independent of the matrix $\bX$ used in the iterative update.

\noindent
\textbf{(C3').} [Initialization-dependent condition on regularization parameters] 
\begin{enumerate}
\item  [(a)] There exist constants $C,c>0$ such that
\begin{align*}
    \lambda_1\geq C& \sqrt{\log p(n(\beta_1^2+1)+p)}(1+a\sqrt{s}) \\
    &+ C\sqrt{n\log^{3}p} \frac{\sqrt{pn}(\sqrt{r}+a\sqrt{s})+n\sqrt{s}a^2+r\sqrt{n(\beta_1^2+1)}}{p},
\end{align*}
where $a= \|\Pi_{\bB^{(0)}}-\Pi_{\bV}\|_F$, and  $$\lambda_1\leq c\frac{\min(\beta_r^2,1)\sqrt{(\beta_r^2+1)n((\beta_r^2+1)n+p)}}{\sqrt{sr}}.$$
\item  [(b)]	
	(For SPCA algorithm only) $\lambda_0>C\Big(n(\beta_r^2+1)s+p\Big)$
	\end{enumerate}
However, \textbf{(C3')} requires $\lambda_1$ to be dependent on the initialization error, which may be difficult to be examined in practice. Thus, we use the more restrictive conditions \textbf{(C2)} and \textbf{(C3)}.
\end{remark}

The following corollary gives the sufficient condition such that $\Pr(E_1)\to 1$ in Theorem~\ref{thm:main}.

\begin{corollary}\label{coro1}
Under the same conditions of Theorem~\ref{thm:main}, if we assume that $s/\kappa^2\rightarrow\infty$ and $t\rightarrow\infty$ as $n,p\rightarrow\infty$, then $\Pr(E_1)$ in \eqref{eq:prob_E} converges to $1$.  If we in addition assume that $\kappa\sqrt{r}{\lambda_1}/{\sqrt{n(\beta_1^2+1)+p}}\rightarrow 0$  and $\kappa \sqrt{{t}/{s}} \rightarrow 0$, then we have 
$$
\lim_{\iter\rightarrow\infty}\Big\|{\Pi_{\bB^{(\iter)}}}-\Pi_{\bV}\Big\|_F=o_p(1).
$$
\end{corollary}

To obtain the consistency in Corollary~\ref{coro1}, we need to choose $t$  properly {such that $t\rightarrow\infty$ and $\frac{t}{s/\kappa^2}\rightarrow 0$.} Then $\Pr(E_1)$ in \eqref{eq:prob_E} converges to $1$ and the estimation error bound in \eqref{eq:error_main} converges to $0$ as $n,p\rightarrow\infty$. 

\begin{remark}\label{remark1}
Note that $s/\kappa^2\rightarrow\infty$ holds when $\kappa=o(1)$ and $s$ is fixed, or $\kappa=O(1)$ (which follows from \textbf{(C1)}) and $s\to\infty$. The condition that $\kappa=o(1)$ when $s$ is fixed, or $s\to\infty$  in Corollary 1 is satisfied under the conditions SP and $AD(m,\kappa)$ of \cite{ma2013}. Specifically, we have $\kappa=o(\sqrt{s})$ when  $\beta_1^2+1=o(n\beta_r^4)$, and it immediately implies that $\kappa=o(1)$ when $s$ is fixed, or $s\to\infty$. According to the condition SP, we have $n\beta_r^i\to\infty(i=2,4)$, then we have $1=o(n\beta_r^4)$. We also have $\beta_1^2=o(n\beta_r^4)$ according to the condition $AD(m,\kappa)$. Thus, $\beta_1^2+1=o(n\beta_r^4)$ and thus $\kappa=o(\sqrt{s})$ under the conditions SP and $AD(m,\kappa)$ of \cite{ma2013}.
\end{remark}


Theorem~\ref{thm:main} requires a good initialization satisfying the condition \textbf{(C2)}. In what follows, we aim to show how one can obtain the desired initialization for SPCA and ITPS. 

We introduce a stronger condition on the strength of the signal as follows:

\noindent
\textbf{(C1a).} [A stronger condition on the strength of the signal for diagonal thresholding initialization] 
There exists $M_0'>0$ such that
\begin{equation}\label{signal}
\frac{\sqrt{s(\beta_1^2+1)}+\sqrt{s}\beta_r(n\log p)^{\frac{1}{4}}}{\beta_r^2\sqrt{n}}\leq M_0'/{\sqrt{s}}.
\end{equation}

{\noindent
\textbf{(C1b).} [A stronger condition on the strength of the signal for the Fantope-based thresholding initialization] 
There exists $C>0$ such that
\begin{equation}\label{signal2}
\sqrt{n}\geq C\log^2 p\Big(\frac{\beta_1^2+1}{\beta_r^2}s(\frac{\beta_1^2+1}{\beta_r^4}s\sqrt{\log p}+\sqrt{(\beta_1^2+1)s})\Big)
\end{equation}
}

{\textbf{(C1a)} and \textbf{(C1b)} serve a similar purpose as \textbf{(C1)}, but they make a stronger assumption on the strength of the signal. The left-hand side of \eqref{signal} is larger than $\kappa$, and its right-hand side is smaller than $O(1)$. These stronger conditions are needed since we would like to make sure that the diagonal thresholding procedure and the Fantope procedure return a sufficiently good initialization. \textbf{(C1a)}  can be simplified to $s\leq C(\frac{n}{\log p})^{\frac{1}{4}}$ when $\beta_1, \cdots, \beta_r=O(1)$; and under the same assumption, \textbf{(C1b)} can be simplified to $s\leq C(\frac{n}{\log^3 p})^{\frac{1}{4}}$.

The following lemma establishes the theoretical guarantee of two initialization methods: diagonal thresholding and Fantope-based thresholding. While it is natural to use the top eigenvectors of the Fantope solution \citep{NIPS2013_81e5f81d,lei2015} as initialization, its theoretical analysis is not straightforward due to the absence of guaranteed sparsity. To overcome this challenge, we examine a Fantope-based thresholding method as described in part (b) of the lemma. This procedure requires a separate  dataset $\bX'$ that is independent of the primary dataset, $\bX$. This assumption of independence can be satisfied by implementing the data splitting scheme similar to Step 1 in Algorithm~\ref{alg:spca}: divide $\bX_{(1)}$ into two subsets, use the first subset as $\bX'$, and use the second subset as $\bX$.}

\begin{lemma}\label{prop:init}(a) [Estimation error of the diagonal thresholding initialization]
Under the condition \textbf{(C0)}, if \begin{equation}\label{eq:betar1}\beta_r>2\sqrt{C_2s}(\frac{\log p}{n})^{\frac{1}{4}}\end{equation} and $\bB^{(0)}$ is generated from the diagonal thresholding with $C_{thr}=n+4\sqrt{pn}$, there exists a constant $C_2>0$  such that 
$\bB^{(0)}$ satisfies that $\bB^{(0)}_{\calS^c}=0$ and 
\begin{equation}\label{eq:lemma11}
\|\Pi_{\bB^{(0)}}-\Pi_{\bV}\|_F\leq \frac{C_1\sqrt{(\beta_1^2+1)}(\sqrt{s}+\sqrt{s}/\kappa)+C_2\beta_r\sqrt{s}(n\log p)^{\frac{1}{4}}}{\beta_r^2\sqrt{n}}\end{equation}
under an event $E_2$ whose probability at least
\begin{equation}\label{eq:prob_E2}
\Pr(E_2)\geq 1-2/p-\exp(-s/2\kappa^2).
\end{equation}
If we also assume the condition \textbf{(C1a)}, {then \eqref{eq:betar1} holds and} there exists some $M_0>0$ such that 
\begin{equation}\label{eq:lemma12}
\|\Pi_{\bB^{(0)}}-\Pi_{\bV}\|_F\leq M_0/\sqrt{s}
\end{equation} holds under the same event $E_2$.

{(b) [Estimation error of the Fantope-based thresholding initialization] Let $\widehat{\bV}\in\reals^{p\times r}$ consist of the top $r$ eigenvectors of $\widehat{\bZ}$, the Fantope solution of $\bX'$, a data matrix consists of $n$ samples and independent of $\bX$, and 
let 
\begin{equation}\label{eq:fantope}\widehat{\calS}=\Big\{1\leq i\leq p: \|[\bX^T\bX\widehat{\bV}]_{i}\|\geq  C\log^2 p\Big(\frac{\beta_1^2+1}{\beta_r^4}s\sqrt{n\log p}+ \sqrt{(\beta_1^2+1)sn}\Big)\Big\},\end{equation}
 and $\bB^{(0)}$ be the top $r$ singular vectors of $[\bX]_{\widehat{\calS}}$.  If \begin{equation}\label{eq:betar2}\beta_r\geq C\sqrt{s}\beta_1\log p\Big((\beta_1^2+1)(\frac{\beta_1^2+1}{\beta_r^4}s\sqrt{\log p/n}+\sqrt{(\beta_1^2+1)s/n})\Big),\end{equation} then under an event $E_3$ whose probability is at least $1-2/p$, we have $\widehat{\calS}\subseteq \calS$, and 
\[
\|\Pi_{\bB^{(0)}}-\Pi_{\bV}\|_F\leq C\sqrt{s}\log^2 p\frac{\beta_1^2+1}{\beta_r^2}\Big(\frac{\beta_1^2+1}{\beta_r^4}s\sqrt{\log p/n}+\sqrt{(\beta_1^2+1)s/n}\Big).
\]
If we also assume the condition \textbf{(C1b)}, then \eqref{eq:betar2} holds and there exists some $M_0>0$ such that 
\begin{equation}\label{eq:lemma12b}
\|\Pi_{\bB^{(0)}}-\Pi_{\bV}\|_F\leq M_0/\sqrt{s}
\end{equation} holds under the same event $E_3$.
}
\end{lemma}

Lemma~\ref{prop:init}(a) shows that  $\bB^{(0)}$ generated from the diagonal thresholding has an estimation error rate of $O(n^{-\frac{1}{4}})$ and satisfies  \textbf{(C2)} under the  condition \textbf{(C1a)}. {Lemma~\ref{prop:init}(b) shows that a Fantope solution-based thresholding procedure has an estimation error rate of $O(n^{-\frac{1}{2}})$, improved from part (a); but its requirement \textbf{(C1b)} is a little bit more restrictive than  \textbf{(C1a)}, due to some technicalities in the proof and the fact that the rate in \cite[Corollary 3.3]{NIPS2013_81e5f81d} is larger than the optimal minimax rate by a factor of $\beta_1/\beta_r \cdot \sqrt{s/r}$. In our numerical studies, we observed that two initial estimates performed similarly, and  the Fantope can be computationally expensive as it requires solving the SDP of size $p \times p$.}

The probability $\Pr(E_2)$ in Lemma~\ref{prop:init} tends to $1$ if $s/\kappa^2\rightarrow\infty$ and $t\rightarrow\infty$ as $n,p\rightarrow\infty$. As we discussed in Corollary \ref{coro1} and Remark \ref{remark1}, $s/\kappa^2\rightarrow\infty$ holds when $\kappa=o(1)$ and $s$ is fixed, or $\kappa=O(1)$ and $s\to\infty$, which holds under the conditions SP and $AD(m,\kappa)$ of \cite{ma2013}.

As shown in Lemma~\ref{prop:init}, without requiring the  condition \textbf{(C1a)}, $\bB^{(0)}$ generated from the diagonal thresholding has the correct sparsity pattern that $\bB^{(0)}_{\calS^c}=0$ with high probability. Given the additional condition \textbf{(C1a)} on the signal strength, we obtain the desired upper bound for $\|\Pi_{\bB^{(0)}}-\Pi_{\bV}\|_F$, and then $\bB^{(0)}$ generated from the diagonal thresholding satisfies the condition \textbf{(C2)} with high probability. We will provide numerical results in Section~\ref{sec:numerical} to support the theoretical property of the diagonal thresholding in this lemma. Detailed proof of Lemma~\ref{prop:init} is presented in Section 5.

Now, combining Theorem~\ref{thm:main}  and Lemma~\ref{prop:init}, we prove Theorem~\ref{thm:main2} about the estimation error of SPCA and ITPS algorithms as described in Algorithm~\ref{alg:spca1} or Algorithm~\ref{alg:spca}.

	\begin{theorem}\label{thm:main2}
(a) [Estimation errors of SPCA and ITPS  with the diagonal thresholding initialization] 
Under the conditions \textbf{(C0)}, \textbf{(C1a)} and \textbf{(C3)}, for  the sequence $\bB^{(\iter)}$ generated by \eqref{eq:update}-\eqref{eq:updateB} of the SPCA algorithm or \eqref{eq:update2}-\eqref{eq:updateB2} of the ITPS algorithm, if $\bB^{(0)}$ is generated from the diagonal thresholding with $C_{thr}=n+4\sqrt{pn}$,  there exist constants $C_0, C_1>0$ 
such that \eqref{eq:error_main} holds with an arbitrarily chosen $t$  under the event $E_1 \cap E_2$ whose probability $\Pr(E_1\cap E_2)$ is at least
\begin{equation}\label{eq:E12}
1-(C_0+2)/p-C_0pe^{-n/C_0}-C_0pe^{-p/C_0}-2e^{-s/2\kappa^2}-6e^{-\frac{\min(\sqrt{ps},\sqrt{ns})^2}{C_0}}-4e^{-\frac{s}{C_0\kappa^2}}-2e^{-t}.
\end{equation}
{(b) [Estimation errors of SPCA and ITPS  with the Fantope-based thresholding initialization] Under the conditions \textbf{(C0)}, \textbf{(C1b)}, and \textbf{(C3)}, then for $\bB^{(0)}$ generated from the procedure in Lemma~\ref{prop:init}(b) and  the sequence $\bB^{(\iter)}$ generated by \eqref{eq:update}-\eqref{eq:updateB} of the SPCA algorithm or \eqref{eq:update2}-\eqref{eq:updateB2} of the ITPS algorithm, there exist constants $C_0, C_1>0$ 
such that \eqref{eq:error_main} holds with an arbitrarily chosen $t$  under the event $E_1 \cap E_3$ whose probability $\Pr(E_1\cap E_3)$ is at least \eqref{eq:E12}.} 
\end{theorem}

The complete proof of Theorem~\ref{thm:main2} is presented in Section 5. Theorem~\ref{thm:main2} implies the consistency of the estimator generated by the SPCA algorithm or the ITPS algorithm. 
Similar to the discussions after Theorem~\ref{thm:main} and Lemma~\ref{prop:init}, the probability in Theorem~\ref{thm:main2} will tend to $1$  when $s/\kappa^2,t\to\infty$ as $n,p\rightarrow\infty$.

	\begin{remark}\label{remark3}
The data splitting schemes in Algorithm~\ref{alg:spca1} and Algorithm~\ref{alg:spca} are necessary to obtain the estimation errors of SPCA and ITPS in Theorem~\ref{thm:main2}, as the condition \textbf{(C2)} requires that $\bB^{(0)}$ is independent of the data matrix used in the iterative update. 
\end{remark}

The following corollary combines Theorem~\ref{thm:main2}(a) with a specific choice of $\lambda_1$ as well as the regime discussed in Corollary~\ref{coro1} such that $\Pr(E_1\cap E_2)\to 1$ as $n,p\rightarrow\infty$.

\begin{corollary}\label{cor:main2a}
Under the same conditions of Theorem~\ref{thm:main2}(a), if we assume that $s/\kappa^2\rightarrow\infty$ and $t\rightarrow\infty$ as $n,p\rightarrow\infty$, when we choose   $$\lambda_1= M_1'\Big(\sqrt{\log p(n(\beta_1^2+1)+p)} + n\log^{1.5} p\frac{\sqrt{pr}+r(1+\log r)\sqrt{\beta_1^2+1}}{p}\Big)$$ for any $M_1'\geq M_1$, 
then we have that {for SPCA and ITPS  with the diagonal thresholding initialization,} 
	\begin{align*}
	&\lim_{\iter\rightarrow\infty}	\Big\|{\Pi_{\bB^{(\iter)}}}-\Pi_{\bV}\Big\|_F\\\
	\leq & \  \ C_1M_1'\kappa\sqrt{r\log p}\Big(1+n\log p\frac{\sqrt{pr}+r(1+\log r)\sqrt{\beta_1^2+1}}{p\sqrt{n(\beta_1^2+1)+p}}\Big)+C_1\kappa \sqrt{\frac{t}{s}}.\nonumber\end{align*}
with probability tending to $1$  as $n,p\to\infty$.
\end{corollary}	

Moreover, we present Corollary~\ref{cor:main2b} about the simplified upper bound of the estimation error under the  high-dimensional regime that $p>O(n\sqrt{r}\log p)$.

\begin{corollary}\label{cor:main2b}
Under the same conditions of Corollary~\ref{cor:main2a}, if we further assume that 
\begin{equation}\label{eq:large_p}\text{$p/\log p\geq n\sqrt{r}+\sqrt{n}(1+\log r)r$
}\end{equation}
and choose $\lambda_1= 3M_1\sqrt{\log p(n(\beta_1^2+1)+p)}$, then {for SPCA and ITPS  with the diagonal thresholding initialization,} \begin{align}\label{eq:estimation_error3}
	\lim_{\iter\rightarrow\infty}		\Big\|{\Pi_{\bB^{(\iter)}}}-\Pi_{\bV}\Big\|_F\leq 9C_1M_1\kappa\sqrt{r\log p}+C_1\kappa \sqrt{\frac{t}{s}}.\end{align}
holds with probability tending to $1$  as $n,p\to\infty$. If we in addition assume that $\kappa\sqrt{r\log p}\rightarrow 0$  and $\kappa \sqrt{{t}/{s}} \rightarrow 0$, then we have 
$$
\lim_{\iter\rightarrow\infty}\Big\|{\Pi_{\bB^{(\iter)}}}-\Pi_{\bV}\Big\|_F=o_p(1).
$$
\end{corollary}

The proof of Corollary~\ref{cor:main2b} follows from the observation that when \eqref{eq:large_p} holds, we have $$\sqrt{\log p(n(\beta_1^2+1)+p)} \geq  2n\log^{1.5} p\frac{\sqrt{pr}+r\sqrt{\beta_1^2+1}}{p}$$
Then, we may let $M_1'=3M_1$ so that  $\lambda_1= 3M_1\sqrt{\log p(n(\beta_1^2+1)+p)}$ in Corollary~\ref{cor:main2a}.

\subsection{Discussions about main results}\label{subsec:discussion}
	
	This subsection is devoted to  discussions about the main results presented in the previous subsection, including the choice of parameters, comparison with existing works, and comparison with minimax rates.
	
\textbf{Choice of parameters.} {Throughout the analysis, we assume a large $\lambda_0$ via a lower bound of at least $4p$ in condition \textbf{(C3)}(b). In fact, as discussed in Section~\ref{subsec:itps}, the ITPS algorithm can be considered as the limiting case of the SPCA algorithm when $\lambda_0\to \infty$. Our preference of the large $\lambda_0$ follows from the technicalities in the theoretical analysis, and empirically it also performs better than a small $\lambda_0$ as observed in Section~\ref{sec:numerical}.}

For the choice of $\lambda_1$, it should not be too small. The condition  \textbf{(C3)} assumes a lower bound of $\lambda_1$, which is used to make sure that the ``thresholding'' is sufficiently large and the output of the algorithm has the correct sparse pattern. We would like to point out that, if the lower bound is satisfied, then the estimation error increases linearly as a function of $\lambda_1$, as shown in the estimation error \eqref{eq:error_main} of Theorem~\ref{thm:main}. This dependence can be understood as the bias introduced by the $\ell_1$ penalty in the objective function \eqref{eq:problem_matrix}. 

Similarly, a reasonable  choice of  $\lambda_1$ should not be too large either. The condition  \textbf{(C3)} assumes an upper bound of $\lambda_1$, which is used to make sure that the ``thresholding'' should not be so large such that the update of $\bB$ in \eqref{eq:updateB} returns a zero matrix. 

Due to the technicalities in the proof of Theorem~\ref{thm:main}, we can remove the assumptions on the upper bound of $\lambda_1$ in the condition \textbf{(C3)} and obtain a weaker result under the spectral norm instead. Specifically, without requiring the upper bounds of $\lambda_1$ in \textbf{(C3)}, we can obtain the same estimation error bounds in Section~\ref{sec:main} for  $\|{\Pi_{\bB^{(\iter)}}}-\Pi_{\bV}\|$  instead of $\|{\Pi_{\bB^{(\iter)}}}-\Pi_{\bV}\|_F$.

\textbf{Comparison with existing works.} We compare our result in Corollary~\ref{cor:main2b} with existing theoretical analysis on sparse principal subspace estimation in the literature. We remark that while there are many existing works, many of them are not directly comparable. For example, \cite{ma2013} studied the case where the principal components $\{\bv_i\}_{i=1}^r$ are sparse in the sense that $\max_{1\leq i\leq r}\|\bv_i\|_p\leq s$ for $0<p<2$ (while we use the $\ell_0$ norm); and  \cite{jankova2018debiased} only studied the case $r=1$. Here we list three comparable references:
\begin{itemize}
	
\item \cite{cai2013} proposed an adaptive procedure, and Theorem 7 in their work can be compared with Theorem~\ref{thm:main}. Specifically, Theorem 7 assumes a good initialization with the correct sparsity pattern and $\|\Pi_{\bB^{(0)}}-\Pi_{\bV}\|\leq \frac{1}2$ and  $\{\beta_i\}_{i=1}^r$ are in the same order in the sense that there exists $\kappa\geq 1$ with $\kappa\beta\geq \beta_1\geq\cdots\geq \beta_r\geq \beta$. Theorem 7 of \cite{cai2013} then shows that the expectation of the estimation error under the squared Frobenius norm is bounded by 
	\begin{equation}\label{eq:cai}
	\Expect \|\Pi_{\hat\bV}-\Pi_{\bV}\|_F^2\leq C\min\left(r,p-r,\frac{\sqrt{(\beta^2+1)s}}{\beta^2\sqrt{n}}\sqrt{r+\log\frac{ep}{s}}\right). 
	\end{equation}

\item The covariance thresholding method \citep{NIPS2014_5406} has the following estimation error bound under the spectral norm:  $$\|\Pi_{\bB}-\Pi_{\bV}\|\leq O\Big(\frac{\beta_1^2+1}{\beta_r^2}s\sqrt{\frac{r}{n}\log\frac{p}{s^2}}\Big).$$

\item The SDP method \citep{NIPS2013_81e5f81d}  has the following estimation error bound under the Frobenius norm:
$$\|\Pi_{\bB}-\Pi_{\bV}\|_F\leq O\Big(\frac{\beta_1^2+1}{\beta_r^2}s\sqrt{\frac{\log p}{n}}\Big).$$ 
\end{itemize}

Compared with \cite{cai2013}, the estimation error bound obtained in Corollary~\ref{cor:main2b}  achieves \eqref{eq:cai} up to a factor of $\log p$. However, we remark that our result does not depend on the assumption that $\{\beta_i\}_{i=1}^r$ are in the same order. 
Compared with \cite{NIPS2013_81e5f81d}, Corollary~\ref{cor:main2b} improves the rate by a factor of $\sqrt{\frac{r}{{s(\beta_1^2+1)}}}$ (with logarithmic factors ignored). We remark that the result of \cite{NIPS2013_81e5f81d} is stronger than that  of  \cite{NIPS2014_5406}, so our result is also better than that of \cite{NIPS2014_5406}.
	
\textbf{Comparison with the minimax rates.} The minimax rates of sparse PCA have been studied in several works.  Theorem 3.2 of \cite{10.1214/13-AOS1151} investigates the minimax rate over $\bV$ from the set $\mathcal{P}_0(\sigma^2, s_0)=\{s\leq s_0, \frac{\beta_1^2+1}{\beta_r^4}\leq \sigma^2\}$, then 
$$
\inf_{\hat{\bV}}\sup_{\bV\in \mathcal{P}_0(\sigma^2, s_0)}\Expect \|\Pi_{\hat{\bV}}-\Pi_{{\bV}}\|_F\geq C
\sigma\sqrt{\frac{(s-r)(r+\log\frac{p-r}{s-r})}{n}}.$$
Theorem 3 of \cite{cai2013} investigates the minimax rate over $\bV$ from the set $\mathcal{\Theta}=\{1\leq r\leq s\leq p, \beta\leq \beta_r\leq {\beta_1}\leq \kappa\beta\}$, where $\kappa>1$ is a fixed constant and obtains a minimax rate of
$$
\inf_{\hat{\bV}}\sup_{\bV\in \mathcal{\Theta}}\Expect \|\Pi_{\hat{\bV}}-\Pi_{{\bV}}\|_F\geq C
\left(\frac{\sqrt{\beta^2+1}}{\beta^2}\sqrt{\frac{(s-r)\log \frac{e(p-r)}{s-r}+r(s-r)}{n}}\right).$$
Under their regimes, our estimation error bound in Corollary~\ref{cor:main2b} matches these minimax rates up to some logarithmic factors.

{
\textbf{Comparison with  \cite{ma2013}.} The work by \cite{ma2013} may be regarded as the first strong theoretical guarantee of an iterative thresholding  approach under the spiked covariance model. It is interesting to compare the considered ITPS algorithm with  \cite{ma2013} as they bear some resemblance. The iterative thresholding proposed by \cite{ma2013} comprises three key steps in each iteration: first, performing a multiplication with $\bX^T\bX$; second, finding an orthogonal component; and third, applying a thresholding operation. In contrast, ITPS incorporates one additional multiplication step, multiplying $\bX^T\bX$ and finding an orthogonal component in \eqref{eq:update2}, followed by another multiplication with $\bX^T\bX$ and thresholding in \eqref{eq:updateB2}. Notably, ITPS utilizes an explicit formula in \eqref{eq:update2} to find the orthogonal component, as opposed to the use of QR decomposition in \cite{ma2013}. Furthermore, ITPS can be regarded as an objective function optimizer for  \eqref{eq:problem_ITPS}, in contrast to \cite{ma2013}, which lacks this attribute and rigorous convergence guarantees.

The theory in \cite{ma2013} assumes that the principal components $\{\bv_i\}_{i=1}^r$ are sparse in the sense that $\max_{1\leq i\leq r}\|\bv_i\|_q^q\leq s$ for $0<q<2$. In comparison, our work assumes sparsity in terms of the size of the support size, which can be considered as the setting $q=0$. \cite{ma2013} showed that for some specific $K=O(\log n)$, the $K$-th iterative estimation $\tilde{\bB}^{(K)}$ satisfies that
\[
\|\Pi_{\tilde{\bB}^{(K)}}-\Pi_{\bV}\|\leq C\sqrt{r s \Big[\frac{\log p(\beta_r^2+1)}{n\beta_r^4}\Big]^{1-q/2}+\frac{(\beta_1^2+1)}{\beta_r^4}\frac{\log p}{n}}.
\]
When $q\rightarrow 0$, this bound is in the same order as the bound in Corollary~\ref{cor:main2b}, and the bound in  Corollary~\ref{cor:main2b}  is slightly tighter since it holds for $\|\Pi_{\tilde{\bB}^{(K)}}-\Pi_{\bV}\|_F$. 

The proofs in our paper and \cite{ma2013} share a common starting point: an initial  estimation free of false positives in support detection. However, the proof in \cite{ma2013} is based on a bound on the disparity between their algorithm and an oracle algorithm, which is only small for a finite number of iterations. In contrast, our proof hinges on demonstrating that our algorithm, in each iteration, does not generate any false positives in support detection. As a result, our theoretical guarantee applies to the limit of the iterative algorithm.
}

\section{Numerical Studies}\label{sec:numerical}
In this section, we compared the performance of SPCA and ITPS algorithms with several existing methods. Specifically, we considered the original SPCA algorithm of  \cite{doi:10.1198/106186006X113430} denoted by SPCA-ZHT, our implementation of the SPCA algorithm using the diagonal thresholding initialization as denoted by SPCA-ZYX, our implementation of the ITPS algorithm using the diagonal thresholding initialization, the iterative thresholding PCA of \cite{ma2013} denoted by ITSPCA, the augmented sparse PCA of  \cite{2012arXiv1202.1242P} denoted by AUGSPCA, and the diagonal thresholding sparse  PCA of \cite{ma2013} denoted by DTSPCA. We also included two variants of ITSPCA and AUGSPCA, denoted by ITSPCA$_{post}$ and AUGSPCA$_{post}$, which use a post-thresholding step to improve the estimation of the support set. More specifically, for the estimated top $r$ singular vectors $\hat{\bV}\in\reals^{p\times r}$ by ITSPCA and AUGSPCA, if the $i$-th row has a norm not larger than $0.01$, then we replaced the row with zero. {In addition, we examined the performance of the Fantpoe \citep{NIPS2013_81e5f81d,lei2015}, the SOAP algorithm  \citep{NIPS2014_74563ba2}, and our implementation of the SPCA/ITPS algorithm utilizing Fantope-based thresholding initialization. The results are provided in a separate supplementary file due to space constraints.}


The implementation details of these algorithms are described as follows. SPCA-ZHT is solved by the R-package {\em elasticnet} \citep{enet} using the QR decomposition of  $\textbf{X}$ as the initialization and setting $\lambda_0=10^{-6}$ as suggested by \cite{doi:10.1198/106186006X113430}. SPCA-ZYX uses DTSPCA as initialization, and we set $\lambda_0=500000$. We set the same $\lambda_1=(\log p)\lVert X\rVert_2^2$ based on the squared spectral norm for SPCA-ZHT, SPCA-ZYX, and ITS. We generated the initialization of SPCA-ZYX and ITPS using DTSPCA with the thresholding parameter by $C_{thr}=n+\sqrt{pn}$, as discussed in Lemma~\ref{prop:init}. For  the ITSPCA algorithm, we followed \cite{ma2013} to use $\alpha=3$ and $\gamma=1.5$. For the AUGSPCA algorithm, we used $\beta=2$ for the adjustable constant in the iterative thresholding steps, which is the default choice in their implementation. We used the default stopping criterion for SPCA-ZHT, DTSPCA, ITSPCA, and AUGSPCA. For both ITPS and SPCA-ZYX, following a similar reason to the one used in ITSPCA \citep{ma2013}, we used the stopping criterion that $L(\Pi_{\hat{\textbf{B}_{k}}}-\Pi_{\hat{\textbf{B}_{k+1}}}) \leq\frac{1}{np}$, which is much smaller than the theoretical estimation error bound.

In the simulation study, we generated the data matrix $\bX$ from the spiked covariance model~\eqref{eq:spikemodel}. The entries in both the random effects matrix $\bU\in\reals^{n\times r}$ and error matrix $\bE\in\reals^{n\times p}$ are \emph{i.i.d.} sampled from $N(0,1)$. The support set $\calS$ is a random set of size $s$ from $(1,\cdots,p)$, and $\bV$ is generated such that $\bV_{\calS}$ is a random orthogonal matrix of size $s\times r$ and $\bV_{\calS^c}$ is a zero matrix. We considered three settings of $n$ and $p$ to explore the high-dimensional scenario:  $(n,p)=(256,512),(512,1024),(1024,2048)$, and we let the rank $r$ of the data matrix be $2$ or $4$. We also considered two different specifications of eigenvalues $\beta_1,\ldots,\beta_r$: setting $\beta_i=3$ as the same value for each eigenvector $i=1,\cdots,r$, and setting $\beta_i$'s as $(\beta_1,\beta_2)=(3,4)$ when $r=2$ and  $(\beta_1,\beta_2,\beta_3,\beta_4)=(3,4,5,6)$ when $r=4$.

{\small
\begin{table}[H]
\caption{Comparison of different sparse PCA methods when the eigenvalues are of the same size.}\label{tab1}
\begin{center}
\begin{tabular}{l|ccccccccc}
\toprule
\multirow{2}{*}{Method} & \multicolumn{3}{c}{$n=256,p=512 $} & \multicolumn{3}{c}{$n=512,p=1024 $} & \multicolumn{3}{c}{$n=1024,p=2048$}     \\ 
  \cmidrule(r){2-4} \cmidrule(r){5-7} \cmidrule(r){8-10} 
{} & TPR              & FPR    & ${L}(\textbf{V},\hat{\textbf{V}})$       & TPR             & FPR    & ${L}(\textbf{V},\hat{\textbf{V}})$        & TPR             & FPR  & $ {L}(\textbf{V},\hat{\textbf{V}})$  \\  
\midrule
{}&\multicolumn{9}{c}{$(\beta_1,\beta_2)=(3,3)$}\\
ITSPCA            & 1.000 & 0.891 & 0.359 & 1.000 & 0.879 & 0.281 & 1.000 & 0.855 & 0.217 \\
ITSPCA$_{\mathrm{post}}$          & 0.999 & 0.065 & 0.350  & 1.000 & 0.028 & 0.260  & 1.000 & 0.009 & 0.186 \\
AUGSPCA  & 1.000 & 0.788 & 0.404 & 1.000 & 0.764 & 0.310  & 1.000 & 0.699 & 0.234 \\
AUGSPCA$_{\mathrm{post}}$           & 0.999 & 0.064 & 0.397 & 1.000 & 0.032 & 0.292 & 1.000 & 0.010  & 0.205 \\
DTSPCA & 0.477 & 0.000   & 0.891 & 0.469 & 0.000   & 0.892 & 0.478 & 0.000   & 0.901 \\
SPCA-ZHT     & 0.910 & 0.004 & 0.519 & 0.956 & 0.002 & 0.376 & 0.977 & 0.002 & 0.288\\
SPCA-ZYX & 0.955 & 0.001 & 0.336 & 0.976 & 0.000 & 0.257 & 0.984 & 0.000 & 0.199 \\
ITPS     & 0.955 & 0.001 & 0.335 & 0.976 & 0.000 & 0.255 & 0.985 & 0.000 & 0.197

\\\hline
{}&\multicolumn{9}{c}{$(\beta_1,\beta_2,\beta_3,\beta_4)=(3,3,3,3)$}\\

ITSPCA            & 1.000 & 0.962 & 0.495 & 1.000 & 0.934 & 0.389 & 1.000& 0.918 & 0.292 \\
ITSPCA$_{post}$            & 1.000 & 0.064 & 0.482 & 1.000 & 0.027 & 0.358 & 1.000& 0.008 & 0.245 \\
AUGSPCA  & 1.000 & 0.840  & 0.515 & 1.000 & 0.796 & 0.407 & 1.000& 0.768 & 0.304 \\
AUGSPCA$_{post}$        & 1.000 & 0.066 & 0.504 & 1.000 & 0.029 & 0.378 & 1.000& 0.009 & 0.258 \\
DTSPCA & 0.841 & 0.000   & 0.616 & 0.839 & 0.000   & 0.603 & 0.85 & 0.000   & 0.569 \\
SPCA-ZHT     & 0.998 & 0.005 & 0.735 & 0.998 & 0.003 & 0.547 & 0.998 & 0.003 & 0.401\\
SPCA-ZYX & 1.000 & 0.001 & 0.475 & 1.000 & 0.000 & 0.370 & 1.000 & 0.000 & 0.281 \\
ITPS     & 1.000 & 0.001 & 0.473 & 1.000 & 0.000 & 0.366 & 1.000 & 0.000 & 0.277\\

\bottomrule

\end{tabular}
\end{center}
\end{table}
}

To compare the numerical performance of these methods, we use the following three metrics to measure the difference between the estimator $\hat{\bV}$ associated with the support set $\hat{\calS}$ and the underlying $\bV$ associated with the support set ${\calS}$:
\begin{itemize}
    \item the true positive rate: $\mathrm{TPR}={\mathrm{card}(\calS\cap\hat{\calS})}/{\mathrm{card}(\calS)},$
    \item  the false positive rate: $\mathrm{FPR}={\mathrm{card}(\calS\cap\hat{\calS}^c)}/{\mathrm{card}(\calS)}$,
    \item and the loss function: $L(\bV,\hat{\bV})=\|\Pi_{\bV}-\Pi_{\hat{\bV}}\|_F$ as in \eqref{eq:error_measure}.
\end{itemize}
where the notation $\mathrm{card}(\cdot)$ denotes the cardinality of a set, and the support set $\calS$ or $\hat{\calS}$ is obtained by finding the support set of the leading $r$ singular vectors $\bV$ or $\hat{\bV}\in\reals^{p\times r}$. Here, $\mathrm{TPR}$ and $\mathrm{FPR}$ are used to measure the accuracy of the estimated sparsity pattern, and $L(\bV,\hat{\bV})$ is used to measure the estimation performance. These metrics were obtained over $100$ independent repetitions.  

The simulation results were summarized in Tables~\ref{tab1} and \ref{tab2}. Table~\ref{tab1} compares the performance of different SPCA methods when the eigenvalues $\beta_1,\ldots,\beta_r$ are of the same size. Table~\ref{tab2} compares the performance when the eigenvalues $\beta_1,\ldots,\beta_r$ are different.  

{\small
\begin{table}[H]
\caption{Comparison of different sparse PCA methods when the eigenvalues are different.}\label{tab2}
\begin{center}
\begin{tabular}{l|ccccccccc}
\toprule
\multirow{2}{*}{Method} & \multicolumn{3}{c}{$n=256,p=512 $} & \multicolumn{3}{c}{$n=512,p=1024 $} & \multicolumn{3}{c}{$n=1024,p=2048$}     \\ 
  \cmidrule(r){2-4} \cmidrule(r){5-7} \cmidrule(r){8-10} 
{} & TPR              & FPR    & ${L}(\textbf{V},\hat{\textbf{V}})$       & TPR             & FPR    & ${L}(\textbf{V},\hat{\textbf{V}})$        & TPR             & FPR  & $ {L}(\textbf{V},\hat{\textbf{V}})$  \\  
\midrule
{}&\multicolumn{9}{c}{$(\beta_1,\beta_2)=(3,4)$}\\
ITSPCA           & 1.000 & 0.842 & 0.510  & 1.000& 0.826 & 0.394 & 1.000 & 0.753 & 0.298 \\
ITSPCA$_{post}$           & 1.000 & 0.075 & 0.504 & 1.000& 0.035 & 0.382 & 1.000 & 0.014 & 0.275 \\
AUGSPCA  & 1.000 & 0.678 & 0.582 & 1.000& 0.582 & 0.448 & 1.000 & 0.567 & 0.323 \\
AUGSPCA$_{post}$           & 0.999 & 0.064 & 0.578 & 1.000& 0.036 & 0.440  & 1.000 & 0.015 & 0.303 \\
DTSPCA & 0.457 & 0.000   & 0.905 & 0.480 & 0.000   & 0.897 & 0.468 & 0.000   & 0.900   \\
SPCA-ZHT     & 0.881 & 0.000   & 0.585 & 0.933 & 0.000   & 0.462 & 0.954 & 0.000   & 0.340\\
SPCA-ZYX & 0.972 & 0.001 & 0.328 & 0.980 & 0.000 & 0.242 & 0.985 & 0.000 & 0.193 \\
ITPS     & 0.972 & 0.001 & 0.327 & 0.980 & 0.000 & 0.240 & 0.985 & 0.000 & 0.190

\\

\hline
{}&\multicolumn{9}{c}{$(\beta_1,\beta_2,\beta_3,\beta_4)=(3,4,5,6)$}\\

ITSPCA             & 1.000 & 0.933 & 0.682 & 1.000 & 0.905 & 0.551 & 1.000 & 0.864 & 0.415 \\
ITSPCA$_{post}$        & 1.000 & 0.071 & 0.674 & 1.000 & 0.037 & 0.534 & 1.000 & 0.014 & 0.381 \\
AUGSPCA      & 1.000 & 0.738 & 0.731 & 1.000 & 0.649 & 0.576 & 1.000 & 0.631 & 0.428 \\
AUGSPCA$_{post}$     & 1.000 & 0.065 & 0.724 & 1.000 & 0.038 & 0.561 & 1.000 & 0.015 & 0.394 \\
DTSPCA   & 0.845 & 0.000   & 0.634 & 0.846 & 0.000   & 0.627 & 0.862 & 0.000   & 0.596 \\
SPCA-ZHT     & 0.982 & 0.000   & 0.901 & 0.994 & 0.000   & 0.690 & 0.997 & 0.000   & 0.518\\
SPCA-ZYX & 1.000 & 0.000 & 0.469 & 1.000 & 0.000 & 0.358 & 0.999 & 0.000 & 0.283 \\
ITPS     & 1.000 & 0.000 & 0.466 & 1.000 & 0.000 & 0.354 & 1.000 & 0.000 & 0.274 \\
\bottomrule
\end{tabular}
\end{center}
\end{table}
}

As shown in Tables~\ref{tab1} and \ref{tab2}, SPCA-ZYX and ITPS  improve the estimation performance of existing methods, including SPCA-ZHT and their initialization DTSPCA. SPCA-ZYX and ITPS have the smallest estimation error in $L(\bV,\hat{\bV})$ under all the different settings. This observation provides numerical support to the theoretical results that SPCA-ZYX and ITPS can achieve the best available estimation error bounds or the minimax rates up to some logarithmic factors. 

SPCA-ZYX and ITPS also have appealing variable selection performance to recover the true support set in the simulation study. Compared with ITSPCA, AUGSPCA, and their post-thresholding counterparts, SPCA-ZYX and ITPS SPCA-ZYX and ITPS  have comparable true positive rates and lower true positive rates under most settings. The post-thresholding significantly improves the variable selection performance of both  ITSPCA and AUGSPCA. Compared with SPCA-ZHT, SPCA-ZYX and ITPS  have  higher true positive rates and slightly lower false positive rates. Compared with DTSPCA, SPCA-ZYX, and ITPS  have slightly higher true positive rates and significantly lower  false positive rates. 

We should note that DTSPCA always finds the correct sparsity pattern and achieves the perfect false positive rates under different settings, providing numeric evidence to the theoretical results in Lemma~\ref{prop:init}. 

Last but not least, the numerical performance of ITPS and SPCA-ZYX in terms of both accuracy of support set recovery and estimation error is very close under all different settings. However, the average computing time of SPCA-ZHT and SPCA-ZYX is significantly longer than that of ITPS, which is consistent with our motivation. ITPS avoids solving the LASSO-type subproblem at each iteration, which SPCA-ZHT and SPCA-ZYX require. For example, even under the relatively easier setting $(n,p,r)=(256,512,2)$ in Table~\ref{tab1}, the average computing time of SPCA-ZYX is already ten times longer than that of ITPS.

\section{Proofs of Theorems and Lemmas}\label{appA}

\subsection{Proof of Theorem ~\ref{thm:main3a}}
\begin{proof}
Since the SPCA algorithm is an alternating minimization algorithm, the functional value is nonincreasing, i.e., 
\[
f(\bA^{(k)},\bB^{(k)})\geq f(\bA^{(k+1)},\bB^{(k)}) \geq f(\bA^{(k+1)},\bB^{(k+1)})\geq \cdots
\]
and the functional value converges since it is non-negative. As a result, $(\bA^{(k)},\bB^{(k)})$ does not diverge and the sequence must have at least one limiting point.

Let $(\bar{\bA},\bar{\bB})$ be one limit point of the sequence $(\bA^{(k)},\bB^{(k)})$, and let us consider the function $g:\reals^{p\times r}\rightarrow \reals$ defined by
\[
g(\bB)=\|\bX(\bB-\bX^T\bX\bB(\bB^T\bX^T\bX\bX^T\bX\bB)^{-\frac{1}2})\|_F^2+\lambda_0\|\bB\|_F^2+\lambda_1\|\bB\|_1.
\]
Since $\bX^T\bX\bB(\bB^T\bX^T\bX\bX^T\bX\bB)^{-\frac{1}2}$ can be approximated by a linear function of $\bB$,  $g(\bB)$ is strongly convex in a neighborhood of $\bar{\bB}$, that is, there exists $\epsilon_1>0$ and $\alpha_1,\alpha_2,\alpha_3>0$ such that for all $\|\bB-\bar{\bB}\|_F\leq \epsilon_1$, \begin{equation}\alpha_1 \|\bB-\bar{\bB}\|_F^2\leq g(\bar{\bB})-g(\bB)\leq \alpha_2 \|\bB-\bar{\bB}\|_F^2+\alpha_3\|\bB-\bar{\bB}\|_1\label{eq:gbB}\end{equation}for some $\alpha_1,\alpha_2,\alpha_3>0$. The lower bound is due to the fact that there is a positive quadratic term $\lambda_0\|\bB\|_F^2$ in $g(\bB)$, and the additional $\ell_1$ term in the upper bound is due to the term $\lambda_1\|\bB\|_1$ in $g(\bB)$.

In addition, if we let $T: \reals^{p\times r}\rightarrow \reals^{p\times r}$ defined such that it is the updated $\bB$ after one iteration, i.e., $T(\bB^{(k)})=\bB^{(k+1)}$, then clearly $T$ is a continuous mapping at $\bar{\bB}$ as  $\bar{\bB}^T\bX^T\bX\bX^T\bX\bar{\bB}$ is invertible. 

Combining the two observations above, there exists a $\epsilon_2>0$, such that for all $\bB\in\mathcal{U}=\{\bB:g(\bB)-g(\bar{\bB})<\epsilon_2, \|\bB-\bar{\bB}\|_F\leq \epsilon_1\}$, we have $\|\bB-\bar{\bB}\|_F\leq \epsilon_3$ (by the lower bound in \eqref{eq:gbB}) and $\|T(\bB)-\bar{\bB}\|_F\leq \epsilon_4$ (by the continuity of $T$). In addition, we can choose $\epsilon_2$ to be sufficiently small such that $\epsilon_4\leq\epsilon_1$. By the monotonicity of the sequence $g(\bB^{(k)})$, if $\bB^{(k_0)}\in\mathcal{U}$, then $\bB^{(k)}\in\mathcal{U}$ for all $k\geq k_0$. In addition, the upper bound in \eqref{eq:gbB} implies that $\mathcal{U}$ contains a neighborhood of $\bar{\bB}$, so by the definition of limiting point, there must exist $k_0>0$ such that $\bB^{(k_0)}\in\mathcal{U}$. Note that \eqref{eq:gbB} also implies that $\bar{\bB}$ is the unique stationary point of $g$ in $\mathcal{U}$, and any limiting point of $\bB^{(k)}$ must be a fixed point of $T$ and a stationary point of $g$, we have that the sequence $\bB^{(k)}$ converges to $\bar{\bB}$.

By the definition of the alternating minimization algorithm, $\bar{\bA}=\argmin_{\bA^T\bA=\bI}f(\bA,\bar{\bB})$ and $\bar{\bB}=\argmin_{\bB}f(\bar{\bA},\bB)$, so $(\bar{\bA},\bar{\bB})$ is a stationary point of the constrained optimization problem \eqref{eq:problem_matrix}. Now, the proof of Theorem ~\ref{thm:main3a} is complete. 
	\end{proof}

\subsection{Proof of Theorem ~\ref{thm:main3b}}
\begin{proof} 
We first prove part (a) of Theorem ~\ref{thm:main3b}. 

(a) The convergence proof of the ITPS algorithm is similar to that of the SPCA algorithm in Theorem ~\ref{thm:main3a} if we replace the objective function $f$  with $\tilde{f}$ as in \eqref{eq:problem_ITPS} and replace $g$ with
\[
\tilde{g}(\bB)=-2\sum_{i=1}^n\bx_i^T\bX^T\bX\bB(\bB^T\bX^T\bX\bX^T\bX\bB)^{-\frac{1}2}\bB^T\bx_i+\|\bB\|_F^2+\lambda_1\|\bB\|_1.
\]

Next, we prove part (b) of Theorem ~\ref{thm:main3b}. 

(b)  Under the assumption of part (a), there exists $\epsilon_1>0$ such that $\bar{\bB}^T\bX^T\bX\bX^T\bX\bar{\bB}$ is invertible in the neighbor of $\bar{\bB}_{\mathrm{ITPS}}$.  Following the similar proof of  Theorem ~\ref{thm:main3a}, there exists $\epsilon_2>0$ such that when $\lambda_0$ is large enough, for any local minimizer $\bar{\bB}$ of function $g$ in  $N_{\epsilon_1}(\bar{\bB}_{\mathrm{ITPS}})$, $g(\bB)$ is strongly convex in $N_{\epsilon_1}(\bar{\bB}_{\mathrm{ITPS}})$ with convexity parameter at least 1. Similar as the similar proof of  Theorem ~\ref{thm:main3a}, there exists $\epsilon_3>0$ such that $N_{\epsilon_3}(\bar{\bB}_{\mathrm{ITPS}})$ is a subset of $\mathcal{U}$. Let $\epsilon_0=\min(\epsilon_1,\epsilon_2,\epsilon_3)$. As a result, for any $\bB\in \mathcal{U}$ and close to $\bar{\bB}$, we have $$g(\bB)-g(\bar{\bB})\geq \|\bB-\bar{\bB}\|_F^2/2.$$

Note that  \begin{align*}\bar{g}(\bB)=\lim_{\lambda_0\rightarrow\infty}\lambda_0 g(\bB/\lambda_0),\end{align*} As a result, when $\lambda_0$ is large enough, ${g}(\bB/\lambda_0)$ has a unique local minimizer in $N_{\epsilon_0}(\bar{B}_{\mathrm{ITPS}})$. Denote it by $\hat{\bB}/\lambda_0$, because $(\bB)-g(\bar{\bB})\geq \|\bB-\bar{\bB}\|_F^2/2$, we have  $\hat{\bB}$ converges to $\bar{\bB}_{\mathrm{ITPS}}$ as $\lambda_0\rightarrow\infty$. Because $\hat{\bB}/\lambda_0$ is the unique local minimizer of function $g$ in $N_{\epsilon_0}(\bar{\bB}_{\mathrm{ITPS}})$ and $g$ is locally convex in  $N_{\epsilon_0}(\bar{\bB}_{\mathrm{ITPS}})$ when $\lambda_0$ is large enough, by the definition of $g$ and $f$, it can be verified that $\hat{\bB}/\lambda_0$ corresponds to a stationary point of the constrained optimization problem \eqref{eq:problem_matrix} when $\lambda_0$ is large enough in the sense that for 
$$\hat{\bA}=\argmin_{\bA^T\bA=\bI}f(\bA,\hat{\bB}/\lambda_0)=\bX^T\bX\hat{\bB}(\hat{\bB}^T\bX^T\bX\bX^T\bX\hat{\bB})^{-\frac{1}2},$$
and
$$\hat{\bB}/\lambda_0=\argmin_{\bB}f(\hat{\bA},\bB).$$ Denote the stationary points by $(\bar{\bA}_{\lambda_0},\bar{\bB}_{\lambda_0})$, then we have $$\lim_{\lambda_0 \rightarrow\infty}(\bar{\bA}_{\lambda_0},{\lambda_0\bar{\bB}_{\lambda_0}})=(\bar{\bA}_{\mathrm{ITPS}},\bar{\bB}_{\mathrm{ITPS}}).$$ The statement in part (b) is then obtained by these two observations.

Therefore, the proof of Theorem ~\ref{thm:main3b} is complete. 


	
	\end{proof}
	
	\subsection{Proof of Theorem~\ref{thm:main}}\label{sec:mainproof}
    We just need to assume that $n\geq 5\log p$ to prove this theorem.	The proof is based on a list of properties of the iterative update formula~\eqref{eq:update}-\eqref{eq:updateB}. In this section, we present these properties in Lemmas~\ref{lemma:sparse}-~\ref{lemma:prob}, and the proof of Theorem~\ref{thm:main} based on these helpful
lemmas. The proofs of these lemmas are rather technical and are deferred to Section 5. 

	
Lemma~\ref{lemma:sparse} establishes a condition on $\bB$ such that $\bB^{new}_{S^c}=0$, i.e., $\bB^{new}$ has the correct sparsity, where $\bB^{new}$ is obtained from $\bB$ and one iteration of the update \eqref{eq:update}-\eqref{eq:updateB}: 
	\begin{align}\label{eq:update0}
	    \bA& =\bX^T\bX\bB\bW\\
	    \bB^{new}&=\argmin_{\bB\in\reals^{p\times r}}\|\bX(\bB-\bA)\|_F^2+\lambda_0\|\bB\|_F^2+\lambda_1\|\bB\|_1,\label{eq:update0B}
	\end{align}

	\begin{lemma}[Iterative sparsity]\label{lemma:sparse}
	Assume
	\begin{equation}\label{eq:sparsity_assumption_lambda0}
\lambda_0\geq {2{s}\sqrt{\log p}\|\bX_{\calS}\|}
	\end{equation}
	and
	\begin{equation}\label{eq:sparsity_assumption_lambda1}
	\lambda_1\geq 4( h_1(\bB)
	+M_0h_2(\bB)),
	\end{equation}
	for $M_0=\max_{i\in\calS^c}\|\bX_i\|$, 
	\[h_1(\bB)=\max_{i\in\calS^c}\|\bX_i^T\Big(\bX_{\calS}(\bX_{\calS}^T\bX_{\calS}+\lambda_0\bI)^{-1}\bX_{\calS}^T-\bI\Big)\bX_{i^c}\bU_{\bX_{i^c}^T\bX_{i^c}\bB}\|,\]
	and
	\[
	h_2(\bB)=\max_{i\in\calS^c}
	\frac{\|\bX_i\bX_i^T\Pi_{\bX_{\calS}\bB}\| }{\sigma_{r}(\bX_{i^c}^T\Pi_{\bX_{\calS}\bB})}+
	C(1+\log r)\|\bX\|\frac{\|\Pi_{\bX_{\calS}  \bB}\bX_{i}\bX_{i}^T\Pi_{\bX_{\calS}\bB}\|}{\sigma_r(\Pi_{\bX_{\calS}  \bB}\bX_{i^c}\bX_{i^c}^T\Pi_{\bX_{\calS}\bB})},
	\]
	then with probability at least $1-2/p$, for all $\bB$ with $\bB_{\calS^c}=0$, we have $\bB^{new}_{\calS^c}=0$.
	
	\end{lemma}
		Lemma~\ref{lemma:converge} establishes an upper bound on the estimation error, under the assumption that  $\bB^{(\iter)}$ have the correct sparsity for all $\iter\geq 1$, i.e., $\bB^{(\iter)}_{S^c}=0$ for all $\iter\geq 1$.
	\begin{lemma}[Iterative improvement]\label{lemma:converge}
	Let $\bD=(\bX_{\calS}^T\bX_{\calS}+\lambda_0\bI)^{-1}\bX_{\calS}^T\bX\bX^T\bX_{\calS}$, $\Pi_{\bD,r}\in\reals^{p\times p}$ represents the projectors to the span of the first $r$ eigenvector of $\bD$, and $\lambda_i(\bD)$ represents the $i$-th largest eigenvalue of $\bD$. Assume that 
	\begin{equation}\label{eq:iterative_assumption}
	\frac{\lambda_r(\bD)}{\lambda_{r+1}(\bD)}\geq 1+c_0{\beta_r^2},\end{equation}
for some nonzero $c_0$, then there exists $M>0$ depending on $c_0$ such that if	\begin{equation}\label{eq:lemmaconvergeassumption1}
	\frac{\sqrt{sr}\lambda_1}{\sqrt{\lambda_0}\sigma_r({\bD}^{1/2})}<M\min(\beta_r^2,1), \|{\Pi}_{\bB^{(0)}}-\Pi_{\bD,r}\|_F\leq \frac{1}2,
	\end{equation}
and \begin{equation}\label{eq:lemmaconvergeassumption2}\text{$\bB_{\calS^c}^{(\iter)}=0$ for all $\iter\geq 1$},\end{equation} then
	\begin{align}\label{eq:improve}
	\|\lim_{\iter\rightarrow\infty}{\Pi}_{\bB^{(\iter)}}-\Pi_{\bD,r}\|_F\leq \frac{C\sqrt{sr}\lambda_1}{\min(\beta_r^2,1)\sqrt{\lambda_0}\lambda_r({\bD}^{1/2})}.
	\end{align}
	\end{lemma}
	The last lemma establishes some useful probabilistic bounds under the spiked covariance model for the expressions in Lemma~\ref{lemma:sparse} and Lemma~\ref{lemma:converge}. 
	\begin{lemma}[Probabilistic estimations]
\label{lemma:prob}	Assuming the conditions \textbf{(C0), (C1)}, and \textbf{(C3)}b,  then under the same event $E_1$ of Theorem~\ref{thm:main}, the following inequalities hold: 
\begin{align*}
    M_0&\leq 2\sqrt{n\log p},\\
    \|\Pi_{\bD,r}-\Pi_{\bV}\|_F& \leq C\frac{\sqrt{(\beta_1^2+1)}}{\beta_r^2\sqrt{n}}(\sqrt{r(s-r)}+2\sqrt{t }),  \\ \frac{\lambda_r(\bD)}{\lambda_{r+1}(\bD)}&\ge 1+C{\beta_r^2},\\
    \lambda_r(\bD)&\geq C(\beta_r^2+1)n((\beta_r^2+1)n+p)/\lambda_0,\\
    \|\bX_{\calS}\|&\leq 2\sqrt{n(\beta_1^2+1)}.
\end{align*}
	 In addition, for $a\leq \sqrt{1/s}$, we have the upper bounds of $h_1(\bB)$ and $h_2(\bB)$ uniformly over $\{\bB: \bB_{\calS^c}=0, \|\Pi_{\bB}-\Pi_{\bV}\|\leq a\}$ as follows:
	\begin{align*}
	&\max_{\bB: \bB_{\calS^c}=0, \|\Pi_{\bB}-\Pi_{\bV}\|\leq a}h_1(\bB)\leq C\sqrt{\log p}\sqrt{n(\beta_1^2+1)+p},\\
	&\max_{\bB: \bB_{\calS^c}=0, \|\Pi_{\bB}-\Pi_{\bV}\|\leq a}h_2(\bB)\leq  	C\log p\sqrt{rn}\frac{{(1+\log r)\sqrt{r(\beta_1^2+1)}}+\sqrt{p}}{p}.
	\end{align*}
	\end{lemma}
	
Given the above lemmas, we are ready to prove our main results in Theorem~\ref{thm:main}. 
	
\begin{proof}[Proof of Theorem~\ref{thm:main}] Theorem~\ref{thm:main} is proved by combining Theorem~\ref{thm:main3a} and Lemmas~\ref{lemma:sparse}-~\ref{lemma:prob}. We start with the verification of the assumptions in Theorem~\ref{thm:main3a} and these lemmas. 

In  Lemma~\ref{lemma:sparse}, the assumptions \eqref{eq:sparsity_assumption_lambda0} and \eqref{eq:sparsity_assumption_lambda1}  are verified by combining \textbf{(C3)} and Lemma~\ref{lemma:prob}. 

For Lemma~\ref{lemma:converge}, the assumption \eqref{eq:iterative_assumption}  is verified by applying Lemma~\ref{lemma:prob}. The assumptions in \eqref{eq:lemmaconvergeassumption1} are verified by applying \textbf{(C3)} and Lemma~\ref{prop:init} (The assumption $\frac{\sqrt{sr}\lambda_1}{\sqrt{\lambda_0}\sigma_r(\tilde{\bD}^{1/2})}\leq c\beta_r$ can be reduced to $\sqrt{sr}\lambda_1\leq c\min(\beta_r^2,1)\sqrt{(\beta_r^2+1)n((\beta_r^2+1)n+p)}$ using Lemma~\ref{lemma:prob}). The assumption \eqref{eq:lemmaconvergeassumption2} is verified by applying Lemma~\ref{lemma:sparse}.

For Theorem~\ref{thm:main3a},  the assumption can be proved as follows. First, by the fact that the RHS of  \eqref{eq:improve} is not larger than one when we choose large $M_1$ in \textbf{(C3)}, any accumulation point $\bar{\bB}$ has rank $r$. Second, Lemma~\ref{lemma:sparse} implies that $\bB^{\iter}_{\calS^c}=0$ for all $\iter$ and as a result, $\bar{\bB}^T\bX^T\bX\bX^T\bX\bar{\bB}=\bar{\bB}^T(\bX_{\calS}^T\bX\bX^T\bX_{\calS})\bar{\bB}$. Third, $(\bX_{\calS}^T\bX\bX^T\bX_{\calS})$ has rank $s$ in the spiked covariance model with probability $1$. Combining them, then we have proved that $\bar{\bB}^T\bX^T\bX\bX^T\bX\bar{\bB}$ has rank $r$.

Due to the continuity of Frobenius norm function and the algorithm convergence, we have
\begin{equation*}
     \lim_{\iter\rightarrow\infty}\|{\Pi_{\bB^{(\iter)}}}-\Pi_{\bD,r}\|_F= \|\lim_{\iter\rightarrow\infty}{\Pi_{\bB^{(\iter)}}}-\Pi_{\bD,r}\|_F
\end{equation*}
Given this fact, combining  Lemmas~\ref{lemma:converge}    and~\ref{lemma:prob}, we have
\begin{align*}&\lim_{\iter\rightarrow\infty}\|{\Pi_{\bB^{(\iter)}}}-\Pi_{\bV}\|_F\\
\leq& \lim_{\iter\rightarrow\infty}\|{\Pi_{\bB^{(\iter)}}}-\Pi_{\bD,r}\|_F+\|\Pi_{\bD,r}-\Pi_{\bV}\|_F\\
=& \|\lim_{\iter\rightarrow\infty}{\Pi_{\bB^{(\iter)}}}-\Pi_{\bD,r}\|_F+\|\Pi_{\bD,r}-\Pi_{\bV}\|_F\\
\leq& \frac{C}{\min(\beta_r^2,1)\sqrt{\lambda_{r}(\bD)}}\frac{\sqrt{sr}\lambda_1}{\sqrt{\lambda_0}} + C\frac{\sqrt{(\beta_1^2+1)}}{\beta_r^2\sqrt{n}}(\sqrt{r(s-r)}+2\sqrt{t }),
\end{align*}
where we have used the triangle inequality in the first inequality and used Lemmas~\ref{lemma:converge} \& \ref{lemma:prob} in the second inequality. According to Lemma~\ref{lemma:prob}, we have $$\lambda_r(\bD)\geq C(\beta_r^2+1)n((\beta_r^2+1)n+p)/\lambda_0.$$ 
Noting that $
\frac{C}{\min(\beta_r^2,1)}\leq \frac{C}{\beta_r^2}+C=\frac{C(\beta_r^2+1)}{\beta_r^2}$,  then we have
\begin{align*}
      &\frac{C}{\min(\beta_r^2,1)\sqrt{\lambda_{r}(\bD)}}\frac{\sqrt{sr}\lambda_1}{\sqrt{\lambda_0}}\\
      	\leq & \ 
	 \frac{C}{\min(\beta_r^2,1)}\frac{\sqrt{sr}\lambda_1}{\sqrt{(\beta_r^2+1)n((\beta_r^2+1)n+p)}}\\
    	\leq & \ 
	 \frac{C(\beta_r^2+1)}{\beta_r^2}\frac{\sqrt{sr}\lambda_1}{\sqrt{(\beta_r^2+1)n((\beta_r^2+1)n+p)}}\\
	 \leq & \ 
	 \frac{C(\beta_r^2+1)\sqrt{sr}\lambda_1}{\beta_r^2}\min\Big(\frac{1}{(\beta_r^2+1)n}, 	 \frac{1}{\sqrt{np(\beta_r^2+1)}}\Big)\\
	 = & \  C\frac{\lambda_1\sqrt{sr}}{\beta_r^2\sqrt{n}}\min\Big(\frac{1}{\sqrt{n}},\sqrt{\frac{\beta_r^2+1}{p}}\Big)
\end{align*}
where we have used the fact that $$ \frac{1}{
\sqrt{A+B}}\geq\frac{1}{\sqrt{2\max(A,B)}}= \frac{1}{\sqrt{2}}\min(\frac{1}{\sqrt{A}},\frac{1}{\sqrt{B}})$$ again
with $A=(\beta_r^2+1)n>0$ and $B={np(\beta_r^2+1)}>0$ in the third inequality. 

Finally, we have
\begin{align*}&\lim_{\iter\rightarrow\infty}\|{\Pi_{\bB^{(\iter)}}}-\Pi_{\bV}\|_F\\
\leq& \frac{C}{\min(\beta_r^2,1)\sqrt{\lambda_{r}(\bD)}}\frac{\sqrt{sr}\lambda_1}{\sqrt{\lambda_0}} + C\frac{\sqrt{(\beta_1^2+1)}}{\beta_r^2\sqrt{n}}(\sqrt{r(s-r)}+2\sqrt{t })\\
 	\leq & \  C\left(\frac{\lambda_1\sqrt{sr}}{\beta_r^2\sqrt{n}}\min\Big(\frac{1}{\sqrt{n}},\sqrt{\frac{\beta_r^2+1}{p}}\Big)+\frac{\sqrt{sr(\beta_1^2+1)}}{\beta_r^2 \sqrt{n}}+\frac{\sqrt{t (\beta_1^2+1)}}{\beta_r^2 \sqrt{n}}\right)\\
 	\leq & \ C\kappa\sqrt{r}\left(\lambda_1\min\Big(\frac{1}{\sqrt{n(\beta_1^2+1)}},\frac{1}{\sqrt{p}}\Big)+1+\sqrt{\frac{t }{sr}}\right)\\
 	\leq&C\kappa\sqrt{r}\left(\frac{\lambda_1}{\sqrt{n(\beta^2+1)+p}}+1+\sqrt{\frac{t }{sr}}\right)
\end{align*}
where we have used the definition of $\kappa$ in the third inequality and used the fact that $ \frac{1}{
\sqrt{A+B}}\geq \frac{1}{\sqrt{2}}\min(\frac{1}{\sqrt{A}},\frac{1}{\sqrt{B}})$ with $A=n(\beta^2+1)>0$ and $B=p>0$ in the last inequality. Now, by the condition \textbf{(C3)}, we know that $$\frac{\lambda_1}{\sqrt{n(\beta_1^2+1)+p}}\geq O(1).$$ Then, the right-hand side of the above estimation error bound reduces to
$$
\lim_{\iter\rightarrow\infty}\|{\Pi_{\bB^{(\iter)}}}-\Pi_{\bV}\|_F
\le
C_1\kappa\sqrt{r}\left(\frac{\lambda_1}{\sqrt{n(\beta^2+1)+p}}+\sqrt{\frac{t }{sr}}\right),
$$
which immediately implies \eqref{eq:error_main}. Therefore, the proof of Theorem \ref{thm:main} is complete.  
\end{proof}

	\subsection{Proof of Lemma~\ref{prop:init}}
	\begin{proof}[Proof of Lemma~\ref{prop:init}]
	(a) We just need to assume that $n> 16\log p$ to prove this lemma. 
	Let $D_i=\sum_{j=1}^n X_{ij}^2$, then for any $i\in\calS^c$, $D_i\sim \chi_n^2$ and by Example 2.11 of \cite{wainwright_2019},
	\begin{equation}\label{eq:chin}
	\Pr\Big(\Big|\frac{\chi_n^2}{n}-1\Big|>t\Big)<2\exp(-nt^2/8), \text{for all $0<t<1$}.
	\end{equation}
	As a result, let $t=4\sqrt{\log p/n}<1$ and  $C_{thr}=n(1+t)=n+4\sqrt{n\log p}$, then $\Pr(i\in\hat{S}^c)\geq 1-2/p^2$. Applying a union bound over all $i\in\calS^c$,
	\[
	\Pr(\hat{\calS}\subseteq\calS)=	\Pr(i\in\hat{S}^c,\,\,\text{for all $i\in\calS^c$})\geq 1-p(2/p^2)=1-2/p.
	\]
That is, $\bB^{(0)}$ has the correct sparsity with probability at least $1-2/p$.
	
For any $i\in\calS$, $D_i\sim (1+\sum_{j=1}^r\beta_j^2\bV_{ij}^2)\chi_n^2$, and   
\[
\Pr(i\notin \hat{\calS})=\Pr\Big((1+\sum_{j=1}^r\beta_j^2\bV_{ij}^2)\chi_n^2\leq n+4\sqrt{n\log p}\Big)=
\Pr\Big(\frac{\chi_n^2}{n}\leq \frac{n+4\sqrt{n\log p}}{n(1+\sum_{j=1}^r\beta_j^2\bV_{ij}^2)}\Big).
\]

Recall Assumption (C0) that $n/\log p\geq 5$, there exists a constant $C_2$ such that 
\[
\frac{1+4\sqrt{\log p/n}}{1+C_2\sqrt{\log p/n}}= 1-2\sqrt{\log p/n},
\]
so \eqref{eq:chin} implies that if $\sum_{j=1}^r \beta_j^2\bV_{ij}^2>C_2\sqrt{\log p/n}$, then
\[
\Pr(i\notin \hat{\calS})\leq \Pr\Big(\frac{\chi_n^2}{n}\leq 1-2\sqrt{\log p/n}\Big)\leq 2/p^2.
\]
 Applying a union bound over all $i\in\calS\setminus\hat{\calS}$ (which contains at most $s$ indices), we have \[\Pr\left(\|\bV_{\calS\setminus\hat{\calS}}\|_F\leq \sqrt{C_2s}(\frac{\log p}{n})^{\frac{1}{4}}\right)>1-2s/p^2>1-2/p,\] and  Wedin's $\sin\theta$-theorem~\cite{Wedin1972} implies 
	\begin{equation}\label{eq:lemmainit1}
\Pr\left(	\|\Pi_{\bV_{\hat{\calS}}}-\Pi_{\bV_{\calS}}\|_F\leq \frac{\sqrt{C_2 s}}{\beta_r}(\frac{\log p}{n})^{\frac{1}{4}}\right)>1-2/p.
	\end{equation}
	In addition, since $\sum_{i\in\calS\setminus\hat{\calS}}\sum_{j=1}^r \beta_j^2\bV_{ij}^2\leq C_2 s\sqrt{\log p/n}$, and by Assumption \textbf{(C1a)}, we may choose $M_0'$ such that $C_2s\sqrt{\log p/n}\leq \frac{1}{4}\beta_r^2$, we have $$\sigma_r([\diag(\beta_1,\cdots,\beta_r)\bV^T]_{\hat{\calS}})\leq \sigma_r([\diag(\beta_1,\cdots,\beta_r)\bV^T]_{\calS})-\frac{1}{2}\beta_r\geq \frac{1}{2}\beta_r,$$ and similarly, $\|[\diag(\beta_1,\cdots,\beta_r)\bV^T]_{\hat{\calS}}\|\leq 2\beta_1$. Applying the same argument as in the upper bound of $\|\Pi_{\bV}-\Pi_{\bX_{\calS,r}}\|$ in the proof of Lemma~\ref{lemma:prob}, we have that under event $E_1$, then with probability at least $1-\exp(-t/2)$,
	\begin{equation}\label{eq:lemmainit2}
\|\Pi_{\bX_{\hat{\calS}},r}-\Pi_{\bV_{\hat{\calS}}}\|\leq C\frac{\sqrt{(\beta_1^2+1)}}{\beta_r^2\sqrt{n}}(\sqrt{s}+\sqrt{t}).
	\end{equation}
Since $\Pi_{\bB^{(0)}}=\Pi_{\bX_{\hat{\calS}},r}$, \eqref{eq:lemma11} is then proved by combining \eqref{eq:lemmainit1} and \eqref{eq:lemmainit2}, and  \eqref{eq:lemma12} follows from letting $t=\sqrt{s}/\kappa$. Therefore, the proof of Lemma~\ref{prop:init} is complete. 

{(b) Following Corollary 3.3 of \cite{NIPS2013_81e5f81d}, we have
\begin{equation*}
\|\widehat{\bZ}-\bSigma\|_F\leq C\frac{\beta_1^2+1}{\beta_r^2}s\sqrt{\log p/n},
\end{equation*}
and combining it with the Davis-Kahan Theorem, we have
\begin{equation}
\|\Pi_{\widehat{\bV}}-\Pi_{\bSigma,r}\|_F=\|\Pi_{\widehat{\bZ},r}-\Pi_{\bSigma,r}\|_F\leq C\frac{\beta_1^2+1}{\beta_r^4}s\sqrt{\log p/n}=c_{init}.
\end{equation}
Now let us investigate $[\bX^T\bX\widehat{\bV}]_i=\bX_i^T\bX\widehat{\bV}$. Let $\bv\in\reals^p$ be a vector representing the correlations between $\bX_i$ and $\bX_1,\cdots,
\bX_p$, that is, its $k$-th element chosen such that $v_k=\Expect(\bX_k^T\bX_i)/(\bX_i^T\bX_i)$. Then we have that for $\bB=\diag(\beta_1^2,\cdots,\beta_r^2)\in\reals^{r\times r}$
\begin{align*}
\|\bv^T\bV\|=&\|((\bV_i^T\bB\bV_i+1)^{-1}(\bV_i^T\bB\bV))^T\bV\|\\
=&\|(\bV_i^T\bB\bV_i+1)^{-1}\bV_i^T\bB\|\\
=&\frac{\sqrt{\sum_{j}\beta_j^4\bV_{ij}^2}}{1+\sum_{j}\beta_j^2\bV_{ij}^2}\\
\geq &\frac{\beta_r^2\|\bV_i\|}{1+\beta_1^2}
\end{align*}
for $i\in\calS$,
\begin{align*}
&\|\bX_i^T\bX\widehat{\bV}\|=\|\bX_i^T\bX_i\bv^T\widehat{\bV}+\bX_i^T(\bX-\bX_i\bv^T)\widehat{\bV}\|\\\geq& \|\bX_i\|^2\Big(\|\bv^T\bV\|-c_{init}\Big)-\|\bX_i^T\bX_{\calS\setminus\{i\}}\|-\|\bX_i^T\bX_{\calS^\perp}\widehat{\bV}\|\\
\geq & \|\bX_i\|^2\Big(\frac{\beta_r^2\|\bV_i\|}{1+\beta_1^2}-c_{init}\Big)-C\log p\|\bX_i^T\|\sqrt{(\beta_1^2+1)s}-C\log p\|\bX_i^T\|c_{init},
\end{align*}
where the last inequality follows from the fact that $\bX_{\calS\setminus\{i\}}$ is a Gaussian matrix of size $(s-1)\times n$ with elementwise variance bounded by $(\beta^2+1)$, and $\bX_{\calS^\perp}$ is a Gaussian matrix of size $(p-s)\times n$ with elementwise variance $1$. The additional factor $C\log p$ ensures that the event happens with probability $1-1/p$.

On the other hand, for $i\not\in\calS$, $\bv$ is a vector such that $v_i=1$ and $v_j=0$ for $j\neq i$, and
\begin{align*}
&\|\bX_i^T\bX\widehat{\bV}\|=\|\bX_i^T\bX_i\bv^T\widehat{\bV}+\bX_i^T(\bX-\bX_i\bv^T)\widehat{\bV}\|\leq \|\bX_i\|^2\|\widehat{\bV}_i\|+\|\bX_i^T\bX_{\{i\}^c}\|\\\leq& \|\bX_i\|^2c_{init}+\|\bX_i^T\bX_{\{i\}^c}\|\leq \|\bX_i\|^2c_{init}+\|\bX_i^T\bX_{\calS\setminus\{i\}}\|+\|\bX_i^T\bX_{\calS^\perp}\widehat{\bV}\|\\
\leq&\|\bX_i\|^2c_{init}+\log p \|\bX_i^T\|\sqrt{(\beta_1^2+1)s}+\log p\|\bX_i^T\|c_{init}.
\end{align*}
Since $\|\bX_i\|^2\sim n\chi_n^2$ for $i\not\in\calS$, we have the correct selection in the sense that $\widehat{\calS}\subseteq\calS$.

In addition, for all $i\in \calS\setminus\widehat{\calS}$, note that  $\|\bX_i\|^2\sim n(1+\sum_{j=1}^r\beta_j^2\bV_{ij}^2)\chi_n^2$ for $i\in\calS$, we have
\begin{align*}
&
\sqrt{n}\Big(\beta_r^2\sqrt{\sum_{j=1}^r\bV_{ij}^2}\Big/(\beta_1^2+1)-2C\frac{\beta_1^2+1}{\beta_r^4}s\sqrt{\log p/n}\Big) \\
\leq &\log ^2p\Big(\sqrt{(\beta_1^2+1)s}+C\frac{\beta_1^2+1}{\beta_r^4}s\sqrt{\log p/n}\Big),
\end{align*}
that is,
\[
\sqrt{\sum_{j=1}^r\bV_{ij}^2}\leq C\log^2 p\Big(\frac{\beta_1^2+1}{\beta_r^2}(\frac{\beta_1^2+1}{\beta_r^4}s\sqrt{\log p/n}+\sqrt{(\beta_1^2+1)s/n})\Big).
\]
Following the proof of \eqref{eq:lemmainit1} in part (a), we have that if
\[
C\sqrt{s}\beta_1\log^2 p\Big(\frac{\beta_1^2+1}{\beta_r^2}(\frac{\beta_1^2+1}{\beta_r^4}s\sqrt{\log p/n}+\sqrt{(\beta_1^2+1)s/n})\Big)\leq \beta_r,
\]
then
\begin{align*}
&
\Pr\left(	\|\Pi_{\bV_{\widehat{\calS}}}-\Pi_{\bV_{\calS}}\|_F\leq C\log^2 p\Big(\frac{\beta_1^2+1}{\beta_r^2}\sqrt{s}(\frac{\beta_1^2+1}{\beta_r^4}s\sqrt{\log p/n}+\sqrt{(\beta_1^2+1)s/n})\Big)\right)\\
>&1-2/p.
\end{align*}
Therefore, the proof of Lemma~\ref{prop:init} is complete.
}
\end{proof}
	
\subsection{Proof of Theorem~\ref{thm:main2}}
\begin{proof}[Proof of Theorem~\ref{thm:main2}]
The proof of Theorem~\ref{thm:main2} is obtained by combining Theorem~\ref{thm:main} and Lemma~\ref{prop:init}.\end{proof}

		\subsection{Proof of Lemma~\ref{lemma:sparse}}
	\begin{proof}[Proof of Lemma~\ref{lemma:sparse}]
	The proof is based on the two components. First, we will show that for $\bW=(\bB^{T}\bX^T\bX\bX^T\bX\bB)^{-\frac{1}2}$, if 
	\begin{align}
	\nonumber\frac{\lambda_1}{2}\geq& \Bigg\|\bX_{\calS^c}^T\Big(\bX_{\calS}(\bX_{\calS}^T\bX_{\calS}+\lambda_0\bI)^{-1}\bX_{\calS}^T-\bI\Big)\bX\bX^T\bX\bB\bW\\
	&-\frac{1}{2}\lambda_1\bX_{\calS^c}^T\bX_{\calS}(\bX_{\calS}^T\bX_{\calS}+\lambda_0\bI)^{-1}\sign({\bB}^{new}_{\calS})\Bigg\|_{\infty},\label{eq:lambda1}
	\end{align}then ${\bB}^{new}_{\calS^c}=0$.
	Second, we will show that the RHS of \eqref{eq:lambda1} can be estimated using
	\begin{align}
	&\Big\|\bX_{\calS^c}^T\Big(\bX_{\calS}(\bX_{\calS}^T\bX_{\calS}+\lambda_0\bI)^{-1}\bX_{\calS}^T-\bI\Big)\bX\bX^T\bX\bB\bW\Big\|_\infty< h_1(\bB)
	+M_0h_2(\bB)\label{eq:part01}
	\end{align}
and
	\begin{equation}
	\Pr\left(\max_{\bT\in\reals^{s\times r}}\|\bX_{\calS^c}^T\bX_{\calS}(\bX_{\calS}^T\bX_{\calS}+\lambda_0\bI)^{-1}\sign(\bT)\|_\infty<\frac{t_2 {s}\|\bX_{\calS}\|}{\lambda_0}\right)>1-2pe^{-\frac{t_2^2}{2}}.\label{eq:part02}
	\end{equation}
	
Combining \eqref{eq:lambda1}, \eqref{eq:part01}, and \eqref{eq:part02}, and let $t_2=2\sqrt{\log p}$,  Lemma~\ref{lemma:sparse} is proved. 

	\textbf{Proof of \eqref{eq:lambda1}}. 	To prove $\bB^{new}_{\calS^c}=0$, it is sufficient to investigate the ``oracle'' solution $\tilde{\bB}^{new}$, which is defined as the optimization problem \eqref{eq:update0B} with the additional constraint that $\bB_{\calS^c}=0$:
	\begin{equation}
	\tilde{\bB}^{new}=\argmin_{\bB\in\reals^{p\times r}, \bB_{\calS^c}=0}\|\bX(\bB-\bA)\|_F^2+\lambda_0\|\bB\|_F^2+\lambda_1\|\bB\|_1,\label{eq:update_tilde0}
	\end{equation}
	and show that $\bB^{new}=\tilde{\bB}^{new}$.
	
	The oracle solution $\tilde{\bB}^{new}$ defined in \eqref{eq:update_tilde0} can be equivalently described by: $\tilde{\bB}^{new}_{\calS^c}=0$, 
	\begin{align}
	\tilde{\bB}^{new}_{\calS}=\argmin_{\bB_{\calS}\in\reals^{s\times r}}&\Tr\Big((\bB_{\calS}-\bA_{\calS})^T\bX_{\calS}^T\bX_{\calS}(\bB_{\calS}-\bA_{\calS})+2\bB_{\calS}^T\bX_{\calS}^T\bX_{\calS^c}(-\bA_{\calS^c}) \Big)\nonumber \\ & +\lambda_0\|\bB_{\calS}\|_F^2+\lambda_1\|\bB_{\calS}\|_1.\label{eq:update_tilde}
	\end{align}
	Since the sub-derivative of the objective function in \eqref{eq:update_tilde} at $\tilde{\bB}^{new}$ contains zero, we have
	\begin{equation}\label{eq:tilde_derivative}
	0\in 2\bX_{\calS}^T\bX_{\calS}(\tilde{\bB}_{\calS}^{new}-\bA_{\calS})+2\bX_{\calS}^T\bX_{\calS^c}(-\bA_{\calS^c})+2\lambda_0\tilde{\bB}_{\calS}^{new}+\lambda_1\sign(\tilde{\bB}_{\calS}^{new}),
	\end{equation}
	where $\sign(\bB)$ represents the sign of the matrix $\bB$ and $[-1,1]$ if the corresponding entry is zero:
	\[
	[\sign(\bB)]_{ij}=\begin{cases}1,\,\,\text{if $\bB_{ij}>0$,}\\-1,\,\,\text{if $\bB_{ij}<0$,}\\ [-1,1]\,\,\text{if $\bB_{ij}=0$}.
	\end{cases}
	\]
	It then follows that
	\begin{align}\nonumber
	&\tilde{\bB}^{new}_{\calS}\in (\bX_{\calS}^T\bX_{\calS}+\lambda_0\bI)^{-1}\left(\bX_{\calS}^T\bX_{\calS}\bA^{(k+1)}_{\calS}+\bX_{\calS}^T\bX_{\calS^c}\bA^{(k+1)}_{\calS^c}-\frac{1}{2}\lambda_1\sign(\tilde{\bB}^{new}_{\calS})\right)\\
	=& \ (\bX_{\calS}^T\bX_{\calS}+\lambda_0\bI)^{-1}\left(\bX_{\calS}^T\bX\bA-\frac{1}{2}\lambda_1\sign(\tilde{\bB}^{new}_{\calS})\right).\label{eq:update_s}
	\end{align}

	To prove Lemma~\ref{lemma:sparse}, it is sufficient to show $\bB^{new}=\tilde{\bB}^{new}$, which holds if the sub-derivative of the objective function in \eqref{eq:update0B} at $\tilde{\bB}^{new}$ contains zero, that is, \begin{equation}\label{eq:tildenew1}
	0\in 2\bX^T\bX(\tilde{\bB}^{new}-\bA)+2\lambda_0\tilde{\bB}^{new}+\lambda_1\sign(\tilde{\bB}^{new}).\end{equation} Since \eqref{eq:tilde_derivative} suggests $[2\bX^T\bX(\tilde{\bB}^{new}-\bA)+2\lambda_0\tilde{\bB}^{new}+\lambda_1\sign(\tilde{\bB}^{new})]_{\calS}=0$, to prove \eqref{eq:tildenew1}, it is sufficient to show that \begin{equation}\label{eq:tildenew2}0\in [2\bX^T\bX(\tilde{\bB}^{new}-\bA)+2\lambda_0\tilde{\bB}^{new}+\lambda_1\sign(\tilde{\bB}^{new})]_{\calS^c}.\end{equation} Since $\tilde{\bB}^{new}_{\calS^c}=0$, $\sign(\tilde{\bB}^{new})_{\calS^c}$ is elementwisely $[-1,1]$ and \eqref{eq:tildenew2} holds when\[
	\|[2\bX^T\bX(\tilde{\bB}^{new}-\bA)+2\lambda_0\tilde{\bB}^{new}]_{\calS^c}\|_{\infty}\leq \lambda_1.
	\]
	Plug in the expression of $\tilde{\bB}^{new}_{\calS}$ in \eqref{eq:update_s} and $\tilde{\bB}^{new}_{\calS^c}=0$, this above equation is equivalent to
	\begin{align*}\nonumber
    \frac{\lambda_1}{2}\geq& \Bigg\|\bX_{\calS^c}^T\bX_{\calS}(\bX_{\calS}^T\bX_{\calS}+\lambda_0\bI)^{-1}\left(\bX_{\calS}^T\bX\bA-\frac{1}{2}\lambda_1\sign(\tilde{\bB}^{new}_{\calS})\right)-\bX_{\calS^c}^T\bX\bA\Bigg\|_{\infty}\\
	=& \ \Bigg\|\bX_{\calS^c}^T\Big(\bX_{\calS}(\bX_{\calS}^T\bX_{\calS}+\lambda_0\bI)^{-1}\bX_{\calS}^T-\bI\Big)\bX\bA\\
	&-\frac{1}{2}\lambda_1\bX_{\calS^c}^T\bX_{\calS}(\bX_{\calS}^T\bX_{\calS}+\lambda_0\bI)^{-1}\sign(\tilde{\bB}^{new}_{\calS})\Bigg\|_{\infty}
	\end{align*}
	which is equivalent to \eqref{eq:lambda1}. That is, when \eqref{eq:lambda1} holds, then \eqref{eq:tildenew1} holds and $\bB^{new}=\tilde{\bB}^{new}$. 
	\textbf{Proof of \eqref{eq:part01}}. 	By definition,
	\begin{align*}
	&\Big\|\bX_{\calS^c}^T\Big(\bX_{\calS}(\bX_{\calS}^T\bX_{\calS}+\lambda_0\bI)^{-1}\bX_{\calS}^T-\bI\Big)\bX\bX^T\bX\bB\bW\Big\|_\infty\\=
	&\max_{i\in\calS^c, 1\leq j\leq r}\Bigg|\Big[\bX_{\calS^c}^T\Big(\bX_{\calS}(\bX_{\calS}^T\bX_{\calS}+\lambda_0\bI)^{-1}\bX_{\calS}^T-\bI\Big)\bX\bX^T\bX\bB\bW\Big]_{ij}\Bigg|,
	\end{align*}
	and note that for any $i\in\calS^c$ and $1\leq j\leq r$, $$\Big[\bX_{\calS^c}^T\Big(\bX_{\calS}(\bX_{\calS}^T\bX_{\calS}+\lambda_0\bI)^{-1}\bX_{\calS}^T-\bI\Big)\bX\bX^T\bX\bB\bW\Big]_{ij}$$ is given by the inner product of $\bX_{i}$ and 
	\begin{equation}\label{eq:part1}
	\Big(\bX_{\calS}(\bX_{\calS}^T\bX_{\calS}+\lambda_0\bI)^{-1}\bX_{\calS}^T-\bI\Big)\Big[\bX\bX^T\bX\bB\bW\Big]_j.
	\end{equation}
	To decouple the dependence between $\bX_i$ and \eqref{eq:part1}, we write $\Big[\bX\bX^T\bX\bB\bW\Big]$ as the sum of the following three components, where $\bW_{i^c}=(\bB^T\bX_{i^c}^T\bX_{i^c}\bX_{i^c}^T\bX_{i^c}\bB)^{-\frac{1}2}$ is an approximation of $\bW$ that does not depend on $\bX_i$:
	\begin{equation}\label{eq:part11}
	\bX_{i^c}\bX_{i^c}^T\bX_{i^c}\bB\bW_{i^c}
	\end{equation}
	\begin{equation}\label{eq:part12}
	(\bX\bX^T\bX-\bX_{i^c}\bX_{i^c}^T\bX_{i^c})\bB\bW_{i^c}
	\end{equation}
	\begin{equation}\label{eq:part13}
	\bX\bX^T\bX\bB\big(\bW_{i^c}-\bW\big).
	\end{equation}
	Denote the upper bounds of the operator norms of \eqref{eq:part11}, \eqref{eq:part12}, and \eqref{eq:part13} by $M_1$, $M_2$, and $M_3$ separately, note that $\bX_i$ is \emph{i.i.d.} $N(0,1)$ and independent of $(\bX_{\calS}^T\bX_{\calS}+\lambda_0\bI)^{-1}$ and the expression in \eqref{eq:part1}, and $\|(\bX_{\calS}^T\bX_{\calS}+\lambda_0\bI)^{-1}\|\leq 1$, 
	\begin{align*}
&	\Pr\left(\Big|[\bX_{\calS^c}^T\Big(\bX_{\calS}(\bX_{\calS}^T\bX_{\calS}+\lambda_0\bI)^{-1}\bX_{\calS}^T-\bI\Big)\bX\bX^T\bX\bB\bW]_{ij}\Big|\leq h_1(\bB) + M_0(M_2+M_3)\right)\\\geq &1-\exp(-t^2/2).
	\end{align*}
	Applying a union bound over all $i\in\calS^c$ and $1\leq j\leq r$, we have
	\begin{align*}
&
	\Pr\left(\Big\|\bX_{\calS^c}^T\Big(\bX_{\calS}(\bX_{\calS}^T\bX_{\calS}+\lambda_0\bI)^{-1}\bX_{\calS}^T-\bI\Big)\bX\bX^T\bX\bB\bW\Big\|_{\infty}\leq h_1(\bB) + M_0(M_2+M_3)\right)\\\geq & 1-pr\exp(-t^2/2).
	\end{align*}
	
	To prove \eqref{eq:part01}, it remains to show that $M_2+M_3\leq h_2(\bB)$.
	
	\textbf{Upper bound of $M_2$}
	
Since $i\in S^c$ and $\bB_{S^c}=0$, we have $\bX_i\bB=0$ and  
	\[
	 (\bX\bX^T\bX-\bX_{i^c}\bX_{i^c}^T\bX_{i^c})\bB=(\bX_i\bX_i^T\bX_{i^c}+\bX_{i^c}\bX_{i^c}^T\bX_i)\bB=\bX_i\bX_i^T\bX_{i^c}\bB,
	\]
	and 
	\begin{align*}
	&\|\bX_i\bX_i^T\bX_{i^c}\bB(\bB^{T}\bX_{i^c}^T\bX_{i^c}\bX_{i^c}^T\bX_{i^c}\bB)^{-\frac{1}2}\|\\
	\leq&\|\bX_i\bX_i^T\bX_{i^c}\bB\|\|(\bB^{T}\bX_{i^c}^T\bX_{i^c}\bX_{i^c}^T\bX_{i^c}\bB)^{-\frac{1}2}\|\\=& \ \|\bX_i\bX_i^T\bX_{i^c}\bB\|\sigma_r^{-\frac{1}2}(\bB^{T}\bX_{i^c}^T\bX_{i^c}\bX_{i^c}^T\bX_{i^c}\bB)\\
	=& \ 
	\frac{\|\bX_i\bX_i^T\Pi_{\bX_{i^c}\bB}\|}{\sigma_{r}(\bX_{i^c}^T\Pi_{\bX_{i^c}\bB})},
		\end{align*}
	
	we have
	\[
	M_2\leq\frac{\|\bX_i\bX_i^T\Pi_{\bX_{i^c}\bB}\| }{\sigma_{r}(\bX_{i^c}^T\Pi_{\bX_{i^c}\bB})}=\frac{\|\bX_i\bX_i^T\Pi_{\bX_{\calS}\bB}\| }{\sigma_{r}(\bX_{i^c}^T\Pi_{\bX_{\calS}\bB})}.
	\]
	%
	%
	%
	
	
	\textbf{Upper bound of $M_3$}
	
	Lemma~\ref{lemma:pertubation1} implies that 
	\begin{align*}
	&(\bB^T\bX_{i^c}^T\bX_{i^c}\bX_{i^c}^T\bX_{i^c}\bB)^{1/2}-(\bB^T\bX^T\bX\bX^T\bX\bB)^{1/2}=\bZ(\bB^T\bX_{i^c}^T\bX_{i^c}\bX_{i^c}^T\bX_{i^c}\bB)^{1/2}
	\end{align*}
	and 
	\begin{align*}
	&\|\bZ\|\\
	\leq & \ C(1+\log r)\|\bW_{i^c}\bB^T\bX^T\bX\bX^T\bX\bB\bW_{i^c}-\bI\|
	\\=& \ C(1+\log r)\|\bW_{i^c}\bB^T(\bX^T\bX\bX^T\bX-\bX_{i^c}^T\bX_{i^c}\bX_{i^c}^T\bX_{i^c})\bB\bW_{i^c}\|
	\\=& \ C(1+\log r)\|\bW_{i^c}\bB^T\bX_{\calS}^T(\bX\bX^T-\bX_{i^c}\bX_{i^c}^T)\bX_{\calS}\bB(\bB^T\bX_{\calS}^T\bW_{i^c}\|\\
	=& \ C(1+\log r)\|(\bB^T\bX_{\calS}^T\bX_{i^c}\bX_{i^c}^T\bX_{\calS}\bB)^{-\frac{1}2}\bB^T\bX_{\calS}^T\Pi_{\bX_{\calS}  \bB}\bX_i\\
	& \ \bX_i^T\Pi_{\bX_{\calS}\bB}\bX_{\calS}\bB(\bB^T\bX_{\calS}^T\bX_{i^c}\bX_{i^c}^T\bX_{\calS}\bB)^{-\frac{1}2}\|
	\\\leq& 
	C(1+\log r){\|\Pi_{\bX_{\calS}  \bB}\bX_i\bX_i^T\Pi_{\bX_{\calS}\bB}\|}\|(\bB^T\bX_{\calS}^T\bX_{i^c}\bX_{i^c}^T\bX_{\calS}\bB)^{-\frac{1}2}\bB^T\bX_{\calS}^T\|^2\\
	\leq&C(1+\log r)\frac{\|\Pi_{\bX_{\calS}  \bB}\bX_i\bX_i^T\Pi_{\bX_{\calS}\bB}\|}{\sigma_r(\Pi_{\bX_{\calS}  \bB}\bX_{i^c}\bX_{i^c}^T\Pi_{\bX_{\calS}\bB})},
	\end{align*}
where the third equality uses $\bW_{i^c}=(\bB^T\bX_{\calS}^T\bX_{i^c}\bX_{i^c}^T\bX_{\calS}\bB)^{-\frac{1}2}$ and the last inequality is due to the fact that
\begin{align*}
&\|(\bB^T\bX_{\calS}^T\bX_{i^c}\bX_{i^c}^T\bX_{\calS}\bB)^{-\frac{1}2}\bB^T\bX_{\calS}^T\|^2=
\|\bX_{\calS}\bB(\bB^T\bX_{\calS}^T\bX_{i^c}\bX_{i^c}^T\bX_{\calS}\bB)^{-1}\bB^T\bX_{\calS}^T\|\\
=& \ \Big\|\bX_{\calS}\bB\Big(\bB^T\bX_{\calS}^T\bU_{\bX_{\calS}  \bB}(\bU_{\bX_{\calS}  \bB}^T\bX_{i^c}\bX_{i^c}^T\bU_{\bX_{\calS}  \bB})\bU_{\bX_{\calS}  \bB}^T\bX_{\calS}\bB\Big)^{-1}\bB^T\bX_{\calS}^T\Big\|
\\
\leq & \ \Big\|\bX_{\calS}\bB\Big(\bB^T\bX_{\calS}^T\bU_{\bX_{\calS}  \bB}\big({\sigma_r(\bU_{\bX_{\calS}  \bB}\bX_{i^c}\bX_{i^c}^T\bU_{\bX_{\calS}\bB})}\bI\big)\bU_{\bX_{\calS}\bB}^T\bX_{\calS}\bB\Big)^{-1}\bB^T\bX_{\calS}^T\Big\|\\
=& \ \frac{1}{\sigma_r(\bU_{\bX_{\calS}  \bB}\bX_{i^c}\bX_{i^c}^T\bU_{\bX_{\calS}\bB})}\Big\|\bX_{\calS}\bB\Big(\bB^T\bX_{\calS}^T\bU_{\bX_{\calS}  \bB}\bU_{\bX_{\calS}\bB}^T\bX_{\calS}\bB\Big)^{-1}\bB^T\bX_{\calS}^T\Big\|\\
&=\frac{1}{\sigma_r(\Pi_{\bX_{\calS}  \bB}\bX_{i^c}\bX_{i^c}^T\Pi_{\bX_{\calS}\bB})}.
\end{align*}
	Combining the estimation of $\bZ$ with
	\begin{align*}
	&\bX\bX^T\bX\bB\big(\bW_{i^c}-\bW\big)
	\\=& \ \bX \bX^T\bX\bB\bW\big((\bB^T\bX^T\bX\bX^T\bX\bB)^{1/2}-(\bB^T\bX_{i^c}^T\bX_{i^c}\bX_{i^c}^T\bX_{i^c}\bB)^{1/2}\big)\bW_{i^c}\\
	=& \ \bX \Pi_o(\bX^T\bX\bB)\bZ.
	\end{align*}
	we have
	\[
	M_3\leq \|\bX \Pi_o(\bX^T\bX\bB)\bZ\|\leq C(1+\log r)\|\bX\|\frac{\|\Pi_{\bX_{\calS}  \bB}\bX_{i}\bX_{i}^T\Pi_{\bX_{\calS}\bB}\|}{\sigma_r(\Pi_{\bX_{\calS}  \bB}\bX_{i^c}\bX_{i^c}^T\Pi_{\bX_{\calS}\bB})}.
	\]
	
	\textbf{Proof of \eqref{eq:part02}}.
	We will prove the following instead, which implies \eqref{eq:part02}:
	\begin{equation}
	\Pr\left(\max_{\bt\in\reals^{s}, \|\bt\|_{\infty}\leq 1}\|\bX_{\calS^c}^T\bX_{\calS}(\bX_{\calS}^T\bX_{\calS}+\lambda_0\bI)^{-1}\bt\|_\infty<\frac{t_2\sqrt{s}\|\bX_{\calS}\|}{\lambda_0}\right)>1-2pe^{Cs}\exp(-t_2^2/2).\label{eq:part03}
	\end{equation}
	
	First,
	\begin{align*}
	&\|\bX_{\calS^c}^T\bX_{\calS}(\bX_{\calS}^T\bX_{\calS}+\lambda_0\bI)^{-1}\bt\|_\infty=
	\max_{i\in\calS^c}|\bX_{i}^T\bX_{\calS}(\bX_{\calS}^T\bX_{\calS}+\lambda_0\bI)^{-1}\bt|\\
	&\leq \max_{i\in\calS^c}\|\bX_{i}^T\bX_{\calS}\|\|(\bX_{\calS}^T\bX_{\calS}+\lambda_0\bI)^{-1}\bt\|\leq  \frac{\sqrt{s}}{\lambda_0}\max_{i\in\calS^c}\|\bX_{i}^T\bX_{\calS}\|.\end{align*}
	In addition, since $\bX_i$ is independent of $\bX_{\calS}$, $\bX_i^T\bX_{\calS}$ is bounded by a Gaussian distribution $N(0,\|\bX_{\calS}\|\bI_{s\times s})$, and the tail bound of Gaussian distribution  in  Proposition~\ref{prop:probability}(b) implies that
	\[
	\Pr(\|\bX_i^T\bX_{\calS}\|>t_2\sqrt{s}\|\bX_{\calS}\|)\leq 2\exp\Big(-\frac{t_2^2}{2}\Big),
	\]
	and applying a union bound over $i\in\calS^c$, \eqref{eq:part03} is proved and so does \eqref{eq:part02}. Therefore, the proof of Lemma~\ref{lemma:sparse} is complete. 
	\end{proof}
	\subsection{Proof of Lemma~\ref{lemma:converge}} 
	
	Before proceeding to the proof of Lemma~\ref{lemma:converge}, we first present a useful lemma.
	
	\begin{lemma}\label{lemma:orth}
 For any orthogonal matrix $\bV'\in\reals^{p\times r}$, $\|\Pi_{\bV}-\Pi_{\bV'}\|_F\leq C \|\Pi_{\bV',\perp}\bV\|_F$. 
\end{lemma}

\begin{proof}[Proof of Lemma~\ref{lemma:orth}]
It follows from the theory of principal angles such as Section 3.2.1 of \cite{bbd7088358424c3391ca59a6b25d401b} and Section 6.4.3 of \cite{golub2013matrix}. Let $\theta_1,\cdots,\theta_r\in[0,\pi/2]$ be the principal angles between $\bV$ and $\bV'$, then $\|\Pi_{X}-\Pi_{\bV'}\|_F^2=2\sum_{i=1}^d\sin^2\theta_i+(1-\cos\theta_i)^2=4\sum_{i=1}^r\sin^2\frac{\theta_i}{2}$ and $\|\Pi_{\bV',\perp}\bV\|_F^2=\sum_{i=1}^d\sin^2\theta_i$. Therefore, the proof of Lemma~\ref{lemma:orth} is complete. 
\end{proof}

Now, we are ready to prove Lemma~\ref{lemma:converge}.

	\begin{proof}[Proof of Lemma~\ref{lemma:converge}]
	The update formula for $\bB^{(k)}_{\calS}$ is given by \eqref{eq:update_s} as follows:
	\begin{align}\nonumber
	&{\bB}^{new}_{\calS}=(\bX_{\calS}^T\bX_{\calS}+\lambda_0\bI)^{-1}\left(\bX_{\calS}^T\bX\bA-\frac{1}{2}\lambda_1\sign({\bB}^{new}_{\calS})\right)\\\nonumber
	&=(\bX_{\calS}^T\bX_{\calS}+\lambda_0\bI)^{-1}\Big(\bX_{\calS}^T\bX\bX^T\bX_{\calS}\bB_{\calS}\Big)(\bB^{T}_{\calS}\bX_{\calS}^T\bX\bX^T\bX_{\calS}\bB_{\calS})^{-\frac{1}2}\\&+\frac{1}{2}\lambda_1(\bX_{\calS}^T\bX_{\calS}+\lambda_0\bI)^{-1}\sign({\bB}^{new}_{\calS})\label{eq:update_s1}.
	\end{align}
	Let $\bC^{(\iter)}=(\bX_{\calS}^T\bX_{\calS}+\lambda_0\bI)^{1/2}\bB^{(\iter)}$ and $$\tilde{\bD}=(\bX_{\calS}^T\bX_{\calS}+\lambda_0\bI)^{-\frac{1}2}\bX_{\calS}^T\bX\bX^T\bX_{\calS}(\bX_{\calS}^T\bX_{\calS}+\lambda_0\bI)^{-\frac{1}2},$$ then \eqref{eq:update_s1} implies
	\begin{equation}\label{eq:lemma4_0}
	{\bC}^{new}_{\calS}=\tilde{\bD}\bC_{\calS}(\bC^{T}_{\calS}\tilde{\bD}\bC_{\calS})^{-\frac{1}2}+\bP,
	\end{equation}
	where $\bP=\frac{\lambda_1}{2}(\bX_{\calS}^T\bX_{\calS}+\lambda_0\bI)^{-\frac{1}2}\sign({\bB}^{new}_{\calS})$.

Let $f(\bC)=\min_{\bv\in\reals^{r}}\frac{\|\Pi_{{\bD},r}\bC\bv\|}{\|\Pi_{{\bD},r,\perp}\bC\bv\|}$. Then we have $f(\bC^{(1)})\geq 1$ by the assumption on initialization and 
\[
f\Big(\tilde{\bD}\bC_{\calS}(\bC^{T}_{\calS}\tilde{\bD}\bC_{\calS})^{-\frac{1}2}\Big)\geq \frac{\sigma_r(\tilde{\bD})}{\sigma_{r+1}(\tilde{\bD})}f(\bC_{\calS}).
\]
We will prove that $f(\bC^{(\iter)})\geq 1$ holds for all $\iter\geq 1$.
For any $\bC$ such that $f(\bC)\geq 1$, we have
\[
\min_{\|\bv\|=1}\|\Pi_{{\bD},r}\tilde{\bD}^{1/2}\bC_{\calS}(\bC^{T}_{\calS}\tilde{\bD}\bC_{\calS})^{-\frac{1}2}\bv\|\geq \frac{1}{2}
\]
and
\[
\min_{\|\bv\|=1}\|\Pi_{{\bD},r}\tilde{\bD}\bC_{\calS}(\bC^{T}_{\calS}\tilde{\bD}\bC_{\calS})^{-\frac{1}2}\bv\|\geq \frac{1}{2}\sigma_r(\tilde{\bD})^{1/2}.
\]
Combining it with \eqref{eq:lemma4_0} and $\|\bP\|\leq \|\bP\|_F\leq \frac{\sqrt{sr}\lambda_1}{\sqrt{\lambda_0}}$,
\begin{align*}
f\Big({\bC}^{new}_{\calS}\Big)\geq& \frac{\frac{\sigma_r(\tilde{\bD})f(\bC_{\calS})}{\sqrt{f(\bC_{\calS})^2\sigma_{r}^2(\tilde{\bD})+\sigma_{r+1}^2(\tilde{\bD})}}\frac{1}{2}\sigma_r(\tilde{\bD})^{1/2}-\frac{\sqrt{sr}\lambda_1}{\sqrt{\lambda_0}}}{\frac{\sigma_{r+1}(\tilde{\bD})}{\sqrt{f(\bC_{\calS})^2\sigma_{r}^2(\tilde{\bD})+\sigma_{r+1}^2(\tilde{\bD})}}\frac{1}{2}\sigma_r(\tilde{\bD})^{1/2}+\frac{\sqrt{sr}\lambda_1}{\sqrt{\lambda_0}}}\\
\geq &\frac{\sigma_r(\tilde{\bD})f(\bC_{\calS})-2\frac{\sqrt{sr}\lambda_1}{\sqrt{\lambda_0}}\sigma_r(\tilde{\bD})^{1/2}f(\bC_{\calS})}{\sigma_{r+1}(\tilde{\bD})+2\frac{\sqrt{sr}\lambda_1}{\sqrt{\lambda_0}}\sigma_r(\tilde{\bD})^{1/2}f(\bC_{\calS})}\\=& \ \frac{\sigma_r(\tilde{\bD})-2\frac{\sqrt{sr}\lambda_1}{\sqrt{\lambda_0}}\sigma_r(\tilde{\bD})^{1/2}}{\sigma_{r+1}(\tilde{\bD})+2\frac{\sqrt{sr}\lambda_1}{\sqrt{\lambda_0}}\sigma_r(\tilde{\bD})^{1/2}f(\bC_{\calS})}f(\bC_{\calS}),
\end{align*}
where the second inequality follows from $\sigma_r(\tilde{\bD})\geq \sigma_{r+1}(\tilde{\bD})$ and $f(\bC_{\calS})\geq 1$. It follows that
\begin{equation}\label{eq:fbC}
\lim_{\iter\rightarrow\infty} f(\bC_{\calS}^{(\iter)})\geq \frac{1-2\frac{\sqrt{sr}\lambda_1}{\sqrt{\lambda_0}\sigma_r(\tilde{\bD})^{1/2}}-\frac{1}{1+C\beta_r^2}}{2\frac{\sqrt{sr}\lambda_1}{\sqrt{\lambda_0}\sigma_r(\tilde{\bD})^{1/2}}}\geq C \frac{\sqrt{\lambda_0}\sigma_r(\tilde{\bD})^{1/2}\beta_r^2}{\sqrt{sr}\lambda_1(1+\beta_r^2)}.
\end{equation}
Let $\bU^{(\iter)}=\tilde{\bD}^{1/2}\bC^{(\iter)}_{\calS}(\bC^{(\iter)\,T}_{\calS}\tilde{\bD}\bC^{(\iter)}_{\calS})^{-\frac{1}2}$, then
\[
\bU^{(\iter+1)}=P_{orth}(\tilde{\bD}\bU^{(\iter)}+\tilde{\bD}^{1/2}\bP),
\]
where $P_{orth}(\bX)=\bX(\bX^T\bX)^{-\frac{1}2}$ is the projection to the nearest orthogonal matrix, and $\bU^{(\iter)}$ and $\bA^{(\iter)}$ have the same singular values. In addition, \eqref{eq:fbC} implies that if we let $\gamma$ represent the RHS of \eqref{eq:fbC} and $\gamma\leq \frac{1}2$ (which holds for large $M$), then
\[
\lim_{\iter\rightarrow\infty} \|\Pi_{\tilde{\bD},r,\perp}\bU^{(\iter)}\|\leq C/\gamma,\,\,\, \lim_{\iter\rightarrow\infty} \sigma_r(\Pi_{\tilde{\bD},r}\bU^{(\iter)}\|\geq \sqrt{1-C^2/\gamma^2}\geq c.
\]
Repeat the same argument for $f'(\bC)=\min_{\bv\in\reals^{r}}\frac{\|\Pi_{{\bD},r}\bC\bv\|}{\|\Pi_{{\bD},r,\perp}\bC\|_F}$ and using
\begin{align*}
\|\Pi_{\tilde{\bD},r,\perp}\bU^{(\iter+1)}\|_F\leq & \ \|\Pi_{\tilde{\bD},r,\perp}P_{orth}(\tilde{\bD}^{1/2}\bU^{(\iter)}+\bP)\|_F\\
\leq & \  \frac{\sigma_{r+1}(\tilde{\bD}^{1/2})\|\Pi_{\tilde{\bD},r,\perp}\bU^{(\iter)}\|_F+\|\bP\|_F}{\sigma_r(\tilde{\bD}^{1/2})\sigma_r(\Pi_{\tilde{\bD},r}\bU^{(\iter)})-\|\bP\|_F},
\end{align*}
we have
\begin{equation}\label{eq:lemma44}
\lim_{\iter\rightarrow\infty} \|\Pi_{\tilde{\bD},r,\perp}\bU^{(\iter)}\|_F\leq C/\gamma.
\end{equation}

Equation \eqref{eq:improve} is then proved by applying Lemma~\ref{lemma:orth} to \eqref{eq:lemma44}. Therefore, the proof of Lemma~\ref{lemma:sparse} is complete. 
\end{proof}

	\subsection{Proof of Lemma~\ref{lemma:prob}}
	
	The proof of Lemma~\ref{lemma:prob} will repeatedly apply some results from probability and linear algebra, and we first summarize these results as follows:
\begin{proposition}\label{prop:probability}
(a) [Restatement of Theorem 6.1 in \cite{wainwright_2019}]
Let $\bX\in\reals^{n\times d}$ be a matrix such that each row is \emph{i.i.d.} sampled from $N(0,\Sigma)$, where $\Sigma\in \reals^{d\times d}$. Then
\begin{align*}
\Pr\left(\sigma_1(\bX)\geq \sqrt{\sigma_1(\Sigma)}(\sqrt{n}+\sqrt{d}+\delta)\right)\leq e^{-\delta^2/2}
\end{align*}
and when $n\geq d$,
\begin{align*}
\Pr\left(\sigma_d(\bX)\leq \sqrt{\sigma_d(\Sigma)}(\sqrt{n}-\sqrt{d}-\delta)\right)\leq e^{-\delta^2/2}
\end{align*}
(b) [Restatement of Lemma 1 in \cite{10.1214/aos/1015957395}]
Let $(y_1,\cdots,y_n)$ be \emph{i.i.d.} from $N(0,1)$ and $Z=\sum_{i=1}^na_i(y_i^2-1)$, then
\[
\Pr(Z\geq 2\|\ba\|\sqrt{t}+2\|\ba\|_{\infty}t)\leq \exp(-t).
\]
(c) [Courant-Fischer min-max theorem, see Theorem 1.3.2 in \cite{tao2012topics}] Let $\bA\in\reals^{p\times p}$ be a symmetric matrix. Then we have
\[
\lambda_{r+1}(\bA)=\inf_{\mathrm{dim}(V)=p-r+1}\sup_{\bv\in V: \|\bv\|=1}\bv^T\bA\bv.
\] 
\end{proposition}

{In addition, we introduce the following lemma, which can be considered as a perturbation result, when the matrix $\bZ$ is decomposed into a $2\times 2$ block as follows and we assume that $\bZ_1^{(1)}=\Pi_{\bU}\bZ\Pi_{\bU}$ is the dominant component: 
\[
\bZ= \begin{pmatrix}
\bZ_1^{(1)}=\Pi_{\bU}\bZ\Pi_{\bU} &  \Pi_{\bU}\bZ\Pi_{\bU,\perp}\\
\Pi_{\bU,\perp}\bZ\Pi_{\bU} & \bZ_1^{(2)}=\Pi_{\bU,\perp}\bZ\Pi_{\bU,\perp}
\end{pmatrix}.
\]
\begin{lemma}\label{lemma:decompose}
Given a symmetric matrix $\bZ\in\reals^{p\times p}$ and $\bU\in\reals^{p\times r}$ of rank $r<p$, and let $\bZ_1^{(1)}=\Pi_{\bU}\bZ\Pi_{\bU}$,  $\bZ_1^{(2)}=\Pi_{\bU,\perp}\bZ\Pi_{\bU,\perp}$, $\bZ_2=\Pi_{\bU}\bZ\Pi_{\bU,\perp}+\Pi_{\bU,\perp}\bZ\Pi_{\bU}$. If $\sigma_{d}(\bZ_1^{(1)})-\|\bZ_1^{(2)}\|-2\|\bZ\|>0$, then \\
1. $\sigma_r(\bZ)-\sigma_{r+1}(\bZ)\geq \sigma_r(\bZ_1^{(1)})-\|\bZ_1^{(2)}\|-2\|\bZ_2\|$, $|\sigma_r(\bZ)- \sigma_r(\bZ_1^{(1)})|\leq \|\bZ_2\|$ .\\
2. $\|\Pi_{\bZ,r}-\Pi_{\bU}\|\leq \frac{\|\bZ_2\|}{\sigma_r(\bZ_1^{(1)})-\|\bZ_1^{(2)}\|}$, $\|\Pi_{\bZ,r}-\Pi_{\bU}\|_F\leq \frac{\|\bZ_2\|_F}{\sigma_r(\bZ_1^{(1)})-\|\bZ_1^{(2)}\|}$\\
3. $\sigma_{r+1}(\bZ)\leq \|\bZ_1^{(2)}\|$.
\end{lemma}

\begin{proof}[Proof of Lemma~\ref{lemma:decompose}]
Since the column spaces of $\bZ_1^{(1)}$ and $\bZ_1^{(2)}$ are orthogonal, when $\sigma_r(\bZ_1^{(1)})>\|\bZ_1^{(2)}\|$,  $\sigma_r(\bZ_1)=\sigma_r(\bZ_1^{(1)})$ and $\sigma_{r+1}(\bZ_1)=\|\bZ_1^{(2)}\|$. Then the proof of property 1 follows from Weyl's inequality for singular values: $|\sigma_{i}(\bZ)- \sigma_{i}(\bZ_1)|<\|\bZ_2\|$ for all $i$.

The proof of property 2 follows from the Davis-Kahan Theorem.

The proof of property 3 follows from Proposition~\ref{prop:probability}(c), which implies $$\sigma_{r+1}(\bZ)\leq\|\Pi_{\bX_{\calS},r,\perp}\bZ\Pi_{\bX_{\calS},r,\perp}\|=\|\bZ_1^{(2)}\|.$$ Now, the proof of Lemma~\ref{lemma:decompose} is complete. 
\end{proof}
}

Now, we are ready to prove Lemma~\ref{lemma:prob}. 

	\begin{proof}[Proof of Lemma~\ref{lemma:prob}]
	The proof is based on the following events \textbf{(P1)--(P7)}, and $E_1$ is the event that \textbf{(P1)--(P7)} hold. 

Properties of $\bX_{i}$:

	
	\textbf{(P1)}. For any projector $\Pi_{\bU}\in\reals^{p\times p}$ with rank $d$ and independent of $\bX_{\calS^c}$,  $$\max_{i\in\calS^c}\|\bX_i^T\Pi_{\bU}\|\leq \sqrt{d}+2\sqrt{\log p}.$$ As a  particular case, if $\Pi_{\bU}=\bI$, then we have $$M_0\leq \sqrt{n}+2\sqrt{\log p}.$$


	Properties of $\bX_{\calS}$:
	
\textbf{(P2)}. For some $c_4\leq \frac{\sqrt{n}}{2}-\sqrt{s}$, we have
$$ \|\bX_{\calS}(\bI-\Pi_{\bV})\|\leq \sqrt{n}+\sqrt{s}+c_4,$$ $$\sigma_r(\bX_{\calS}\bV)\geq \sqrt{\beta_r^2+1}(\sqrt{n}-\sqrt{s}-c_4),$$ $$\|\Pi_{\bV}\bX_{\calS}^T\bX_{\calS}\Pi_{\bV}^\perp\|\leq 2\sqrt{(\beta_1^2+1)n}(\sqrt{s}+c_4),$$ 
and 
$$\|\Pi_{\bV}\bX_{\calS}^T\bX_{\calS}\Pi_{\bV}^\perp\|_F\leq \sqrt{(\beta_1^2+1)n}(\sqrt{r(s-r)}+2\sqrt{t }).$$

	 
	 	\textbf{(P3)}. $\|\bX_{\calS}\|\leq 2\sqrt{(\beta_1^2+1)n}$.

	Properties of $\bX_{\calS^c}$:

	\textbf{(P4)}. 
For some $c_6\leq \frac{1}{2}\sqrt{p-s}-\sqrt{s}$,
\begin{align*}
\sqrt{p-s}-\sqrt{s}-c_6\leq & \  \sigma_r\Big(\Pi_{\bX_{\calS,r}}\bX_{\calS^c}\bX_{\calS^c}^T\Pi_{\bX_{\calS,r}}\Big)\\
\leq & \ \Big\|\Pi_{\bX_{\calS,r}}\bX_{\calS^c}\bX_{\calS^c}^T\Pi_{\bX_{\calS,r}}\Big\|\\
\leq & \sqrt{p-s}+\sqrt{s}+c_6,
\end{align*}
$$
\|\Pi_{\bX_{\calS,r}}\bX_{\calS^c}\bX_{\calS^c}^T\Pi_{\bX_{\calS,r,\perp}}\|\leq 2\sqrt{p-s}(\sqrt{s}+c_6),$$ and 
$$
\|\Pi_{\bX_{\calS,r}}\bX_{\calS^c}\bX_{\calS^c}^T\Pi_{\bX_{\calS,r,\perp}}\|_F\leq \sqrt{p-s}(\sqrt{r(s-r)}+2\sqrt{t }).$$
	
Properties of $\bX$:

\textbf{(P5)}. $\max_{i\in\calS}\sigma_s(\Pi_{\bX_{\calS}}^T\bX_{i^c}\bX_{i^c}^T\Pi_{\bX_{\calS}})\geq C p$ for some constant $C$.

\textbf{(P6)}. $\max_{i\in\calS^c}\frac{	\|\bX_{i^{c}}^T\bX_{\calS}(\bI-\Pi_{\bV})\|}{	\sigma_r(\bX_{i^{c}}^T\bX_{\calS}\Pi_{\bV})}\leq C$ for some constant $C$.
	
\textbf{(P7)}. $\|\bX\|\leq 2\sqrt{n(\beta_1^2+1)+p}$.

In what follows, we will first prove that these events hold with high probability. Recall that $\bX\in\reals^{n\times p}$ can be considered as a matrix with rows \emph{i.i.d.} sampled from $N(0,\Sigma)$, where $\Sigma=\diag(\beta_1^2+1,\cdots, \beta_r^1+1,1,\cdots,1)\in\reals^{p\times p}$. 
	
\textit{Probability of \textbf{(P1)}.} Since $\|\bX_i^T\Pi_{\bX_{\calS}\bV}\|=\|\bX_i^T\bU_{\bX_{\calS}\bV}\|$, and $\bX_i^T\bU_{\bX_{\calS}\bV}$ is a vector of length $r$ with elements \emph{i.i.d.} sampled from $N(0,1)$, the tail bound for $\chi_r^2$ distributed variable in Proposition~\ref{prop:probability}(b) implies that with probability at least $1-\exp(-2t)$, $\|\bX_i^T\bU_{\bX_{\calS}\bV}\|^2\leq d+4\sqrt{dt}+4t\leq (\sqrt{n}+2\sqrt{t})^2$. Then a union bound with $t={\log p}$ proves that the first event in \textbf{(P1)} holds with probability at least $1-1/p$. 
	
\textit{Probability of \textbf{(P2)}.}	Since $\bX_{\calS}(\bI-\Pi_{\bV})$ is equivalent to a matrix of size $n\times (s-r)$ with entries \emph{i.i.d.} from $N(0,1)$, Proposition~\ref{prop:probability}(a) implies that $ \|\bX_{\calS}(\bI-\Pi_{\bV})\|\leq \sqrt{n}+\sqrt{s}+c_4$ holds with probability at least  $1-2\exp(-c_4^2/2)$. 

By noticing that $\bX_{\calS}\bV\in\reals^{n\times r}$ has rows \emph{i.i.d.} sampled from $N(0,\Sigma')$, where $\Sigma'$ has eigenvalues $\beta_1^2+1,\cdots, \beta_r^2+1$. Then  Proposition~\ref{prop:probability}(a) implies that $\sigma_r(\bX_{\calS}\bV)\geq \sqrt{\beta_r^2+1}(\sqrt{n}-\sqrt{s}-c_4)$ holds with probability at least $1-2\exp(-c_4^2/2)$.

The proof of the third inequality again follows from Proposition~\ref{prop:probability}(a). Since $\Pi_{\bV}^\perp\bX_{\calS}^T\bX_{\calS}\Pi_{\bV}$ is equivalent to a matrix of size $(s-r)\times r$ with each element \emph{i.i.d.} sampled from a Gaussian distribution with variance not larger than  $(\beta_1^2+1)n$, with probability at least $1-2\exp(-c_4^2/2)$, $\|\Pi_{\bV}^\perp\bX_{\calS}^T\bX_{\calS}\Pi_{\bV}\| \leq 2\sqrt{(\beta_1^2+1)n}(\sqrt{s}+c_4)$.

The proof of the last inequality again follows from the same observation on $\Pi_{\bV}^\perp\bX_{\calS}^T\bX_{\calS}\Pi_{\bV}$, and  Proposition~\ref{prop:probability}(b), which shows that it holds with probability $1-\exp(-t )$

\textit{Probability of \textbf{(P5)}.}	The proof of \textbf{(P5)} is based on the fact that $$\sigma_s(\Pi_{\bX_{\calS}}^T\bX_{i^c}\bX_{i^c}^T\Pi_{\bX_{\calS}})\geq \sigma_s(\Pi_{\bX_{\calS}}^T\bX_{\calS^c}\bX_{\calS^c}^T\Pi_{\bX_{\calS}})=\sigma_s(\bU_{\bX_{\calS}}^T\bX_{\calS^c}\bX_{\calS^c}^T\bU_{\bX_{\calS}}),$$ and $\bU_{\bX_{\calS}}^T\bX_{\calS^c}$ is a matrix of size $s\times (p-s)$ with entries \emph{i.i.d.} from $N(0,1)$. Then Proposition~\ref{prop:probability}(a) and \textbf{(C0)} imply that \textbf{(P5)} holds with probability $1-2\exp(-Cp)$.

\textit{Probability of \textbf{(P6)}.} Let $\bX_{\calS}^{(1)}=\bX_{\calS}\Pi_{\bV}$ and  $\bX_{\calS}^{(2)}=\bX_{\calS}(\bI-\Pi_{\bV})$, then $\bX_{i^{c}}^T\bX_{\calS}\Pi_{\bV}$ can be split into two parts: $\bX_{\calS}^{(1)\,T}\bX_{\calS}^{(1)}$ and $\bG\bX_{\calS}^{(1)}$, where $\bG\in\reals^{(p-1-r) \times n}$ is \emph{i.i.d.} $N(0,1)$. As a result, Proposition~\ref{prop:probability}(a) implies that with probability at least $1-2p\exp(-C\min(n,p))$, for all $i\in\calS^c$, 
	\begin{align*}
&	\sigma_r(\bX_{i^{c}}^T\bX_{\calS}\Pi_{\bV})^2\geq \sigma_r(\bX_{\calS}^{(1)})^4+(\beta_r^2+1)n(\sqrt{p-1-r}-\sqrt{r}-t)^2	\\
	\geq & \ (\beta_r^2+1)^2(\sqrt{n}-\sqrt{r}-t)^4+(\beta_r^2+1)n(\sqrt{p-1-r}-\sqrt{r}-t)^2\\
	\geq & \ C\Big((\beta_r^2+1)^2n^2+(\beta_r^2+1)np\Big),\end{align*} where we apply \textbf{(C0)}, $r\leq s$, and let $t=C\sqrt{\min(n,p)}$ in the last step.
	
Since $\bX_{\calS}(\bI-\Pi_{\bV})=\bX_{\calS}^{(2)}(\bI-\Pi_{\bV})$ and $\bX_{i^c}=\bX_{\{\calS,i\}^c}+\bX_{\calS}^{(1)}+\bX_{\calS}^{(2)}$, $\bX_{i^{c}}^T\bX_{\calS}(\bI-\Pi_{\bV})$ can be split into three parts: $\bX_{\calS}^{(1)\,T}\bX_{\calS}^{(2)}$, $\bX_{\calS}^{(2)\,T}\bX_{\calS}^{(2)}$, and $\bX_{\{\calS,i\}^c}^T\bX_{\calS}^{(2)}$ with their rows orthogonal to each other. As a result, Proposition~\ref{prop:probability}(a) with $t=\sqrt{s}$ implies that with probability at least $1-\exp(-n/2)-\exp(-t_2/2)-\exp(-p/2)$, 
		\begin{align*}
	&\|\bX_{i^{c}}^T\bX_{\calS}(\bI-\Pi_{\bV})\|^2\\
	\leq& \ \|\bX^T\bX_{\calS}(\bI-\Pi_{\bV})\|^2\\
	\leq & \ \|\bX_{\calS}^{(2)\,T}\bX_{\calS}^{(2)}\|^2+ \|\bX_{\calS}^{(1)\,T}\bX_{\calS}^{(2)}\|^2+\|\bX_{\calS^c}^T\bX_{\calS}^{(2)}\|^2
	\\
	\leq& \ (\sqrt{n}+\sqrt{s-r}+\sqrt{n})^4+(\beta_1^2+1)n(\sqrt{s-r}+\sqrt{r}+\sqrt{t_2})^2 \\
	& \ + n(\sqrt{p-s}+\sqrt{s-r}+\sqrt{p})^2
	\\\leq & \ C\left(n^2+(\beta_1^2+1)n(s+t_2)+np\right).
	\end{align*}
	Let $t_2=\frac{(\beta_r^2+1)^2n}{(\beta_1^2+1)}$,  then $t_2\geq \frac{\beta_r^4n}{(\beta_1^2+1)}=\frac{s}{\kappa^2}$, and the previous estimations of $\|\bX_{i^{c}}^T\bX_{\calS}(\bI-\Pi_{\bV})\|$ and $\sigma_r(\bX_{i^{c}}^T\bX_{\calS}\Pi_{\bV})$ imply that \textbf{(P6)} holds with probability at least $1-\exp(-n/2)-\exp(-p/2)-\exp(-\frac{s}{2\kappa^2})-2p\exp(-C\min(n,p))$.

\textit{Probability of  \textbf{(P7)}.}	The proof of \textbf{(P7)} follows from the combination of the upper bounds of $\|\bX_{\calS}\|$ and $\|\bX_{\calS^c}\|$ in \textbf{(P4)} and \textbf{(P3)}.
	
\textit{Probability of \textbf{(P4)}.} The proof of \textbf{(P4)} also follows from Proposition~\ref{prop:probability}(a), as $\bU_{\bX_{\calS}}^T\bX_{\calS^c}$ is a matrix of size $s\times (p-s)$ with elements \emph{i.i.d.} from $N(0,1)$. As a result, the first inequality holds with probability at least $1-2\exp(-c_6^2/2)$. The second inequality is proved by noticing that $\bU_{\bX_{\calS,r}}^T\bX_{\calS^c}\bX_{\calS^c}^T\bU_{\bX_{\calS,r,\perp}}$ is a matrix of size $r\times (s-r)$ with each element \emph{i.i.d.} sampled from $N(0,p-s)$, and Proposition~\ref{prop:probability}(a) shows that is holds with probability at least $1-2\exp(-c_6^2/2)$, and Proposition~\ref{prop:probability}(b) shows that the last inequality holds with probability $1-\exp(-t )$.

\textit{Probability of \textbf{(P3)}.}	 The upper bound of $\|\bX_{\calS}\|$  follows from Proposition~\ref{prop:probability}(a) and the assumption that $s\leq n$, which holds with probability at least  $1-C\exp(-n/C)$. As for $\sigma_{r+1}(\bX_{\calS})$, the Courant-Fischer min-max theorem~\cite[Theorem 1.3.2]{tao2012topics} implies that $\sigma_{r+1}(\bX_{\calS})\leq \|(\bI-\Pi_{\bV})\bX_{\calS}\|$, and noticing that $(\bI-\Pi_{\bV})\bX_{\calS}$ is equivalent to a  matrix of size $n\times r$ with elements \emph{i.i.d.} from $N(0,1)$, \cite[Example 6.1]{wainwright_2019} implies that $\sigma_{r+1}(\bX_{\calS})\leq 2\sqrt{n}$ with probability at least $1-C\exp(-n/C)$.

{In summary, $E_1$ holds with probability at least $1-\frac{C}{p}-Cp\exp(-C\min(n,p))-2\exp(-\frac{(\tau-1)^2s}{8})-\exp(-\frac{s}{2\kappa^2})-6\exp(-\frac{c_4^2}{2})-4\exp(-\frac{c_6^2}{2})-2\exp(-t ).$

In the proof, we will use $c_5=1/6$ and $\kappa_{*}= 1/48$ \begin{equation}\label{eq:c57}c_4=c_6=\min \Big(\frac{c_5\sqrt{s}}{128\kappa},\frac{c_5\sqrt{n}}{80}, \frac{\sqrt{p-s}}{2}\Big)-\sqrt{s},\end{equation}} which leads to the desired probability bound in \eqref{eq:prob_E}.
	
Next, we will prove the estimation error bound in Lemma~\ref{lemma:prob} under  event $E_1$.
	
	%
	%
	%
	%
	%
	
	\textbf{Uniform upper bound of $h_2(\bB)$} 
	Recall that \[
	h_2(\bB)=\max_{i\in\calS^c}
	\frac{\|\bX_i\bX_i^T\Pi_{\bX_{\calS}\bB}\| }{\sigma_{r}(\bX_{i^c}^T\Pi_{\bX_{\calS}\bB})}+
	C(1+\log r)\|\bX\|\frac{\|\Pi_{\bX_{\calS}  \bB}\bX_{i}\bX_{i}^T\Pi_{\bX_{\calS}\bB}\|}{\sigma_r(\Pi_{\bX_{\calS}  \bB}\bX_{i^c}\bX_{i^c}^T\Pi_{\bX_{\calS}\bB})},
	\]
we will estimate the upper bounds of $\frac{\|\Pi_{\bX_{\calS}  \bB}\bX_{i}\bX_{i}^T\Pi_{\bX_{\calS}\bB}\|}{\sigma_r(\Pi_{\bX_{\calS}  \bB}\bX_{i^c}\bX_{i^c}^T\Pi_{\bX_{\calS}\bB})}$ and $\max_{i\in\calS^c}
	\frac{\|\bX_i\bX_i^T\Pi_{\bX_{\calS}\bB}\| }{\sigma_{r}(\bX_{i^c}^T\Pi_{\bX_{\calS}\bB})}$ separately, while $\|\bX\|$ is bounded in \textbf{(P3)}.
	
First, we bound $\frac{\|\Pi_{\bX_{\calS}  \bB}\bX_{i}\bX_{i}^T\Pi_{\bX_{\calS}\bB}\|}{\sigma_r(\Pi_{\bX_{\calS}  \bB}\bX_{i^c}\bX_{i^c}^T\Pi_{\bX_{\calS}\bB})}$. 	The numerator is $$\|\bX_i^T\Pi_{\bX_{\calS}\bB}\|^2\leq 2\|\bX_i^T\Pi_{\bX_{\calS}\bV}\|^2+2\|\bX_i^T(\Pi_{\bX_{\calS}\bB}-\Pi_{\bX_{\calS}\bV})\|^2,$$ and \textbf{(P1)} implies that $\max_{i\in\calS^c}\|\bX_i^T\bU_{\bX_{\calS}\bV}\|\leq 2\sqrt{r\log p}$. Lemma~\ref{lemma:pertubation1} and \textbf{(P3)} imply that
	\begin{align}\nonumber
&	\|\bX_i^T(\Pi_{\bX_{\calS}\bB}-\Pi_{\bX_{\calS}\bV})\|\leq \|\bX_i^T\Pi_{\bX_{\calS}}\|\frac{\|\bX_{\calS}(\bI-\Pi_{\bV})\|}{\sigma_r(\bX_{\calS}\bV)}\|\Pi_{\bV}-\Pi_{\bB}\|\\\leq& 2\sqrt{s\log p}\frac{\frac{3}{2}\sqrt{n}}{\frac{1}{2}\sqrt{(\beta_r^2+1)n}}\|\Pi_{\bV}-\Pi_{\bB}\|= 6\sqrt{\frac{s\log p}{1+\beta_r^2}}\|\Pi_{\bV}-\Pi_{\bB}\|.\label{eq:temp1}
	\end{align}

	As for the denominator, since $\bU_{\bX_{\calS}  \bB}\bX_{i^c}\bX_{i^c}^T\bU_{\bX_{\calS}\bB}\in\reals^{r\times r}$ can be viewed as $$\bU_{\bX_{\calS}  }\bX_{i^c}\bX_{i^c}^T\bU_{\bX_{\calS}}\in\reals^{s\times s}$$ multiplied from left and right by an orthogonal matrix $\bV\in\reals^{s\times r}$,  \textbf{(P5)} implies
	\begin{equation}\label{eq:h2_est}
	\sigma_r(\Pi_{\bX_{\calS}  \bB}\bX_{i^c}\bX_{i^c}^T\Pi_{\bX_{\calS}\bB})\geq\sigma_s(\bU_{\bX_{\calS}}^T\bX_{i^c}\bX_{i^c}^T\bU_{\bX_{\calS}})\geq \frac{p}{2}.
	\end{equation}
Second, we estimate $\frac{\|\bX_i\bX_i^T\Pi_{\bX_{i^c}\bB}\|}{\sigma_{r}(\bX_{i^c}^T\Pi_{\bX_{i^c}\bB})}$. Its denominator can be estimated using \eqref{eq:h2_est} and 
$$\sigma_{r}(\bX_{i^c}^T\Pi_{\bX_{i^c}\bB})=\sqrt{\sigma_{r}(\Pi_{\bX_{\calS}  \bB}\bX_{i^c}\bX_{i^c}^T\Pi_{\bX_{\calS}\bB})}\geq \sqrt{\sigma_{s}(\Pi_{\bX_{\calS}  }\bX_{i^c}\bX_{i^c}^T\Pi_{\bX_{\calS}})}\geq \sqrt{p/2},$$
where the last step follows from \textbf{(P5)}. The numerator $\|\bX_i\bX_i^T\Pi_{\bX_{i^c}\bB}\|$ can be estimated using 
	\[
	\|\bX_i\bX_i^T\Pi_{\bX_{i^c}\bB}\|\leq \|\bX_i\bX_i^T\Pi_{\bX_{i^c}\bV}\|+\|\bX_i\bX_i^T\Pi_{\bX_{i^c}\bB}-\bX_i\bX_i^T\Pi_{\bX_{i^c}\bV}\|,
	\]
where assumption \textbf{(P1)} implies that for all $i\in\calS^c$,
	\[
	\|\bX_i\bX_i^T\Pi_{\bX_{i^c}\bV}\|\leq\|\bX_i\|\|\bX_i^T\Pi_{\bX_{i^c}\bV}\| \leq 4\log p\sqrt{nr},
	\]
	and assumptions \textbf{(P1)}-\textbf{(P3)}, and Lemma~\ref{lemma:pertubation1} implies that
	\begin{align*}
	\|\bX_i\bX_i^T\Pi_{\bX_{i^c}\bV}-\bX_i\bX_i^T\Pi_{\bX_{i^c}\bB}\|&\leq \|\bX_i\|\|\bX_i^T\Pi_{\bX_{\calS}}\|\frac{\|\bX_{\calS}(\bI-\Pi_{\bV})\|}{\sigma_r(\bX_{\calS}\bV)}\|\Pi_{\bV}-\Pi_{\bB}\|\\
	&\leq 12	\log p\sqrt{ns}\|\Pi_{\bV}-\Pi_{\bB}\|,
	\end{align*}
where the last step uses the same argument as in \eqref{eq:temp1}. 
Combining all the estimations, we have
	\begin{align*}
	&\max_{\bB: \bB_{\calS^c}=0, \|\Pi_{\bB}-\Pi_{\bV}\| \leq a} h_2(\bB) \\ 	\leq& \max_{i\in\calS^c}
	\frac{\|\bX_i\bX_i^T\Pi_{\bX_{\calS}\bV}\|+\|\bX_i\bX_i^T(\Pi_{\bX_{\calS}\bV}-\Pi_{\bX_{\calS}\bB})\| }{\sigma_{r}(\bX_{i^c}^T\Pi_{\bX_{\calS}\bB})}\\
	& \ +
	C(1+\log r)\|\bX\|\frac{\|\Pi_{\bX_{\calS}  \bB}\bX_{i}\bX_{i}^T\Pi_{\bX_{\calS}\bB}\|}{\sigma_r(\Pi_{\bX_{\calS}  \bB}\bX_{i^c}\bX_{i^c}^T\Pi_{\bX_{\calS}\bB})}\\
\leq & \  \frac{4\log p\sqrt{nr}+12a\log p\sqrt{ns} }{\sqrt{p/2}}+C(1+\log r) 2\sqrt{n(\beta_1^2+1)+p} \frac{8r\log p+72a^2 \frac{s\log p}{1+\beta_r^2}}{\frac{p}{2}} 
\\
\leq & \  C(1+\log r)\log p{\sqrt{n(\beta_1^2+1)}}\frac{r+\frac{sa^2}{\beta_r^2+1}}{p}\\
& \ +C\log p\frac{\sqrt{n}\Big(\sqrt{r}+a\sqrt{s}\Big)+(1+\log r)\Big(r+\frac{sa^2}{\beta_r^2+1}\Big) }{\sqrt{p}} \\
\leq & \ 
C(1+\log r)\log p{\sqrt{n(\beta_1^2+1)}}\frac{r}{p}+C\log p\frac{\sqrt{nr}+(1+\log r)r }{\sqrt{p}}\\\leq & \ C(1+\log r)\log p{\sqrt{n(\beta_1^2+1)}}\frac{r}{p}+C\log p\frac{\sqrt{nr} }{\sqrt{p}}\\
=& \ C\log p\sqrt{rn}\frac{{(1+\log r)\sqrt{r(\beta_1^2+1)}}+\sqrt{p}}{p},
\end{align*}
where the third inequality decompose $\sqrt{n(\beta_1^2+1)+p}$ into $\sqrt{n(\beta_1^2+1)}$ and $\sqrt{p}$, the fourth inequality and $sa^2\leq r$, and the last one uses $\log r<r<n$.
	\\
\textbf{Uniform upper bound of $h_1(\bB)$} 
Since $\bV$ is independent of $\bX$, so $\bX_i$ is independent of $\Big(\bX_{\calS}(\bX_{\calS}^T\bX_{\calS}+\lambda_0\bI)^{-1}\bX_{\calS}^T-\bI\Big)\bX_{i^c}\bU_{\bX_{i^c}^T\bX_{i^c}\bV}$, so with probability $1-1/p$, we have
    \begin{align*}
	& h_1(\bV) \\ 
	\leq & \  \ 2\sqrt{\log p}\max_{i\in\calS^c}\|\Big(\bX_{\calS}(\bX_{\calS}^T\bX_{\calS}+\lambda_0\bI)^{-1}\bX_{\calS}^T-\bI\Big)\bX_{i^c}\bU_{\bX_{i^c}^T\bX_{i^c}\bV}\|\\ 
	\leq & \  \ 2\sqrt{\log p}\max_{i\in\calS^c}\|\bX_{i^c}\|\\ 
	\leq & \  \ 2\sqrt{\log p}\|\bX\|.
    \end{align*}
	In addition, Lemma~\ref{lemma:pertubation1} and assumptions \textbf{(P1)} and \textbf{(P6)} imply that 
	\begin{align*}
	&\|h_1(\bV)-h_1(\bB)\| \\ 
	\leq & \  \ \max_{i\in\calS^c}\|\bX_i\Big(\bX_{\calS}(\bX_{\calS}^T\bX_{\calS}+\lambda_0\bI)^{-1}\bX_{\calS}^T-\bI\Big)\bX_{i^c}\|\|\bU_{\bX_{i^c}^T\bX_{i^c}\bV}-\bU_{\bX_{i^c}^T\bX_{i^c}\bB}\|\\
	\leq & \ \ \max_{i\in\calS^c}\|\bX_i\Pi_{\bX_{\calS}}\|\|\bX\|\max_{i\in\calS^c}\frac{\|\bX_{i^c}^T\bX_{\calS}(\bI-\Pi_{\bV})\|}{\sigma_r(\bX_{i^c}^T\bX_{\calS}\bV)}\|\Pi_{\bV}-\Pi_{\bB}\| \\ 
	\leq & \  \ 2\|\bX\|\sqrt{s\log p}\|\Pi_{\bV}-\Pi_{\bB}\|,
	\end{align*}
	Applying assumption \textbf{(P7)}, 
	%
	\begin{align*}
	&\max_{\bB: \bB_{\calS^c}=0, \|\Pi_{\bB}-\Pi_{\bV}\|\leq a} h_1(\bB)\leq  4\sqrt{n(\beta_1^2+1)+p}\sqrt{\log p}(1+a\sqrt{s}).
	\end{align*}
	
{
\textbf{Analysis of $\bX_{\calS}$} In this part, we note that $c_4$ defined in \eqref{eq:c57} satisfies 
\[
c_4+\sqrt{s}\leq \sqrt{n}/2
\]


Apply Lemma~\ref{lemma:decompose} and decompose $\bZ=\bX_{\calS}\bX_{\calS}^T$ into $\bZ_1+\bZ_2=\bZ_1^{(1)}+\bZ_1^{(2)}+\bZ_2$ with $\bU=\bV$, then $\bZ_1=\Pi_{\bV}\bX_{\calS}^T\bX_{\calS}\Pi_{\bV}+\Pi_{\bV}^\perp\bX_{\calS}^T\bX_{\calS}\Pi_{\bV}^\perp$ and $\bZ_2=\Pi_{\bV}^\perp\bX_{\calS}^T\bX_{\calS}\Pi_{\bV}+\Pi_{\bV}\bX_{\calS}^T\bX_{\calS}\Pi_{\bV}^\perp$.  Since $\Pi_{\bV}\bX_{\calS}^T\bX_{\calS}\Pi_{\bV}$ has rank $r$ and the same column space as $\bV$, and the column space of $\Pi_{\bV}^\perp\bX_{\calS}^T\bX_{\calS}\Pi_{\bV}^\perp$ is perpendicular to that of $\bV$, if $
	\sigma_r(\Pi_{\bV}\bX_{\calS}^T\bX_{\calS}\Pi_{\bV})\geq \|\Pi_{\bV}^\perp\bX_{\calS}^T\bX_{\calS}\Pi_{\bV}^\perp\|$, then \begin{equation}\label{eq:upper_bound_1}\Pi_{\bZ_1,r}=\Pi_{\bV},\,\,\, \sigma_r(\bZ_1)=\sigma_r(\Pi_{\bV}\bX_{\calS}^T\bX_{\calS}\Pi_{\bV}),\,\,\,\sigma_{r+1}(\bZ_1)=\|\Pi_{\bV}^\perp\bX_{\calS}^T\bX_{\calS}\Pi_{\bV}^\perp\|.	\end{equation} 
In addition, the singular values of $\bZ_1^{(1)}$ are the $\sigma_i^2(\bX_{\calS}\bV)$ for $1\leq i\leq r$, the singular values of $\bZ_1^{(2)}$ are $\sigma_i^2(\bX_{\calS}\Pi_{\bV,\perp})$, and by \textbf{(P2)},  $\|\bZ_2\|\leq 2\sqrt{(\beta_1^2+1)n}(\sqrt{s}+c_4)$. 

\textit{Lower bound of $\sigma_r^2(\bX_{\calS})-\sigma_{r+1}^2(\bX_{\calS})$.} Applying property 2 of Lemma~\ref{lemma:decompose} and \textbf{(P2)}, we have
\begin{align}&\sigma_r^2(\bX_{\calS})-\sigma_{r+1}^2(\bX_{\calS})\geq \sigma_r^2(\bX_{\calS}\bV)-\|\bX_{\calS}\Pi_{\bV,\perp}\|^2-2\|\bZ_2\|\\
\geq & (\beta_r^2+1)(\sqrt{n}-\sqrt{s}-c_4)^2-(\sqrt{n}+\sqrt{s}+c_4)^2-4\sqrt{(\beta_1^2+1)n}(\sqrt{s}+c_4)\\\geq &\beta_r^2(\sqrt{n}-\sqrt{s}-c_4)^2-8\sqrt{(\beta_1^2+1)n}(\sqrt{s}+c_4)\geq  c_5\beta_r^2n,\label{eq:sigma_r1_X}
\end{align}
where the last step uses the fact that \[
c_5\leq (1-c_7)^2-8\frac{\kappa}{\sqrt{s}}c_7\,\,\,\text{for  $c_7=\frac{\sqrt{s}+c_4}{\sqrt{n}}$.}
\]
It holds because $c_7\leq \frac{1}2$ and the condition in \textbf{(C1)} implies that $\kappa\leq \kappa_*$, and $\kappa_*$ can be chosen such that $(1-c_7)^2-8\frac{\kappa}{\sqrt{s}}c_7\geq 1/6$, say, $\kappa_*=1/48$.

\textit{Upper and lower bounds of $\sigma_r(\bX_{\calS})$.} Note that
\[
\sigma_r^2(\bX_{\calS})\geq \sigma_r^2(\bX_{\calS}\bV)-\|\bZ_2\|\geq (\beta_r^2+1)(\sqrt{n}-\sqrt{s}-c_4)^2-3\sqrt{\beta_1^2+1}\sqrt{n}(\sqrt{s}+c_4),
\]
Now, we use  the same argument as in \eqref{eq:sigma_r1_X} to obtain that
$$
\sigma_r^2(\bX_{\calS})\geq c_5\beta_r^2n+n.
$$

Since $\sqrt{s}+c_4\leq \sqrt{n}$, then
\begin{align}
&\sigma_r^2(\bX_{\calS})\leq \sigma_r^2(\bX_{\calS}\bV)+\|\bZ_2\|\leq  (\beta_r^2+1)(\sqrt{n}+\sqrt{s}+c_4)^2\\&+ 4\sqrt{\beta_1^2+1}\sqrt{n}(\sqrt{s}+c_4)\leq  (5-c_5)\beta_r^2n+4n\leq 5\beta_r^2 n+4n,
\end{align}
where the second inequality step uses  the same argument as in \eqref{eq:sigma_r1_X}.

\textit{Upper bound of $\sigma_{r+1}(\bX_{\calS})$.} Property 3 of Lemma~\ref{lemma:decompose} and  $\sqrt{s}+c_4\leq \sqrt{n}$ imply that
\begin{equation}\label{eq:upper_r+1}
\sigma_{r+1}^2(\bX_{\calS})\leq \|\bX_{\calS}\Pi_{\bV,\perp}\|^2\leq (\sqrt{n}+\sqrt{s}+c_4)^2\leq 4n.
\end{equation}

\textit{Upper bound of $\|\Pi_{\bV}-\Pi_{\bX_{\calS},r}\|_F$.} Property 2 of Lemma~\ref{lemma:decompose} and the same argument as in \eqref{eq:sigma_r1_X} imply that
\begin{align}\nonumber
&\|\Pi_{\bV}-\Pi_{\bX_{\calS,r}}\|_F\leq \frac{2\sqrt{(\beta_1^2+1)n}(\sqrt{r(s-r)}+2\sqrt{t })}{(\beta_r^2+1)(\sqrt{n}-\sqrt{s}-c_4)^2-(\sqrt{n}+\sqrt{s}+c_4)^2}\\
\leq & \  \frac{2\sqrt{(\beta_1^2+1)n}(\sqrt{r(s-r)}+2\sqrt{t })}{c_5\beta_r^2n}=\frac{2\sqrt{\beta_1^2+1}(\sqrt{r(s-r)}+2\sqrt{t })}{c_5\beta_r^2\sqrt{n}}.\label{eq:eigenspacebound1}
\end{align}

\textbf{Analysis of $\bD$.} In this part, we note that $c_6$ is defined such that 
\[
c_6+\sqrt{s} \leq  \min \Big(\frac{c_5\sqrt{s}}{128\kappa},\frac{c_5\sqrt{n}}{80}, \frac{\sqrt{p-s}}{2}\Big).
\]
Applying the definition of $\kappa$ and the property that $\frac{p+n}{\sqrt{p}}\geq 2\sqrt{n}$, we have 
\[
c_6+\sqrt{s}\leq \frac{c_5\sqrt{s}}{128\kappa}=\frac{c_5\sqrt{n}\beta_r^2}{128\sqrt{(\beta_1^2+1)}}\leq \frac{c_5(p+n)\beta_r^2}{256\sqrt{(\beta_1^2+1)p}}\leq \frac{c_5(p+n)\beta_r^2}{256\sqrt{p}},
\]
\[
c_6+\sqrt{s}\leq \frac{c_5\sqrt{n}}{80}\leq   \frac{(p+n)c_5\beta_r^2}{160\beta_r^2\sqrt{p}},
\]
and in summary, 
\begin{equation}\label{eq:c7property}
c_6+\sqrt{s}\leq \min \Big(\frac{(p+n)c_5\beta_r^2}{128\sqrt{(\beta_1^2+1)p}},\frac{(p+n)c_5\beta_r^2}{16(5\beta_r^2+8)\sqrt{p}},\frac{\sqrt{p-s}}{2}\Big).
\end{equation}

Note that $\kappa_*$ can be chosen such that $c_5=1/6$, 

Similar to the previous argument, to establish an upper bound of $\|\Pi_{\tilde{\bD},r}-\Pi_{\bX_{\calS},r}\|$, we apply Lemma~\ref{lemma:decompose} with $\bU=\bX_{\calS,r}$ to decompose $\tilde{\bD}$  into three components $\tilde{\bD}_1^{(1)}$, $\tilde{\bD}_1^{(2)}$, and $\tilde{\bD}_2$. We first decompose $\bY=\bX_{\calS}^T\bX\bX^T\bX_{\calS}=\bX_{\calS}^T\bX_{\calS}\bX_{\calS}^T\bX_{\calS}+\bX_{\calS}^T\bX_{\calS^c}\bX_{\calS^c}^T\bX_{\calS}$ into $(\bY_1^{(1)},\bY_1^{(2)},\bY_2)$ by letting $\bU$ in Lemma~\ref{lemma:decompose} be $\bX_{\calS,r}$:
\begin{align*}
\bY_{1}^{(1)}&=\bX_{\calS,r}^T\bX_{\calS,r}\bX_{\calS,r}^T\bX_{\calS,r}+\bX_{\calS,r}^T\Pi_{\bX_{\calS,r}}\bX_{\calS^c}\bX_{\calS^c}^T\Pi_{\bX_{\calS,r}}\bX_{\calS,r},\\
\bY_{1}^{(2)}&=\bX_{\calS,r,\perp}^T\bX_{\calS,r,\perp}\bX_{\calS,r,\perp}^T\bX_{\calS,r,\perp}+\bX_{\calS,r,\perp}^T\Pi_{\bX_{\calS,r,\perp}}\bX_{\calS^c}\bX_{\calS^c}^T\Pi_{\bX_{\calS,r,\perp}}\bX_{\calS,r,\perp},\\
\bY_2&=\bX_{\calS}^T\bU_{\bX_{\calS}}(\bU_{\bX_{\calS,r}}^T\bX_{\calS^c}\bX_{\calS^c}^T\bU_{\bX_{\calS,r,\perp}}+\bU_{\bX_{\calS,r}}^T\bX_{\calS^c}\bX_{\calS^c}^T\bU_{\bX_{\calS,r,\perp}})\bU_{\bX_{\calS}}^T\bX_{\calS}.
\end{align*}

Recall the notations that $\bX_{\calS,r}$ is the rank-$r$ approximation of $\bX_{\calS}$,  $\bX_{\calS,r,\perp}=\bX_{\calS}-\bX_{\calS,r}$ has rank $s-r$.
 In addition, we define $\tilde{\bD}_k$ and $\tilde{\bD}_1^{(i)}$ such that 
\begin{align*}
	\tilde{\bD}_k=(\bX_{\calS}^T\bX_{\calS}+\lambda_0\bI)^{-\frac{1}2}\bY_k(\bX_{\calS}^T\bX_{\calS}+\lambda_0\bI)^{-\frac{1}2},\,\,\,\text{$k=1,2$,}\\
	\tilde{\bD}_1^{(i)}=(\bX_{\calS}^T\bX_{\calS}+\lambda_0\bI)^{-\frac{1}2}\bY_1^{(i)}(\bX_{\calS}^T\bX_{\calS}+\lambda_0\bI)^{-\frac{1}2},\,\,\,\text{$i=1,2$.}
	\end{align*}
Recall that $\tilde{\bD}=(\bX_{\calS}^T\bX_{\calS}+\lambda_0\bI)^{-\frac{1}2}\bX_{\calS}^T\bX\bX^T\bX_{\calS}(\bX_{\calS}^T\bX_{\calS}+\lambda_0\bI)^{-\frac{1}2}$,  $(\tilde{\bD}_1^{(1)}, \tilde{\bD}_1^{(2)}, \tilde{\bD}_2)$ is also a decomposition of $\tilde{\bD}$ with $\bU=\bX_{\calS,r}$.

If $\sigma_r(\tilde{\bD}_1^{(1)})-\|\tilde{\bD}_1^{(2)}\|-2\|\tilde{\bD}_2\|$ is nonnegative, then Lemma~\ref{lemma:decompose} implies the following properties:\\
1. $\sigma_r(\tilde{\bD})-\sigma_{r+1}(\tilde{\bD})\geq \sigma_r(\tilde{\bD}_1^{(1)})-\|\tilde{\bD}_1^{(2)}\|-2\|\tilde{\bD}_2\|$, $|\sigma_r(\tilde{\bD})- \sigma_r(\tilde{\bD}_1^{(1)})|\leq \|\tilde{\bD}_2\|$ .\\
2. $\|\Pi_{\tilde{\bD},r}-\Pi_{\bX_{\calS},r}\|\leq \frac{\|\tilde{\bD}_2\|}{\sigma_r(\tilde{\bD}_1^{(1)})-\|\tilde{\bD}_1^{(2)}\|}$\\
3. $\sigma_{r+1}(\tilde{\bD})\leq \|\tilde{\bD}_1^{(2)}\|$

\textit{Upper and lower bounds of $\sigma_r(\tilde{\bD}_1^{(1)})$.} Note that we have $\bY_{1,\mathrm{lower}}^{(1)}\psdleq \bY_1^{(1)}\psdleq \bY_{1,\mathrm{upper}}^{(1)}$ for \begin{align*}\bY_{1,\mathrm{lower}}^{(1)}=& \ \bX_{\calS,r}^T\bX_{\calS,r}\bX_{\calS,r}^T\bX_{\calS,r}+\bX_{\calS,r}^T\sigma_r\Big(\Pi_{\bX_{\calS,r}}\bX_{\calS^c}\bX_{\calS^c}^T\Pi_{\bX_{\calS,r}}\Big)\bX_{\calS,r}\\
\bY_{1,\mathrm{upper}}^{(1)}=& \ \bX_{\calS,r}^T\bX_{\calS,r}\bX_{\calS,r}^T\bX_{\calS,r}+\bX_{\calS,r}^T\Big\|\Pi_{\bX_{\calS,r}}\bX_{\calS^c}\bX_{\calS^c}^T\Pi_{\bX_{\calS,r}}\Big\|\bX_{\calS,r},
\end{align*}
we have $\tilde{\bD}_{1,\mathrm{lower}}^{(1)}\psdleq \tilde{\bD}_1^{(1)}\psdleq \tilde{\bD}_{1,\mathrm{upper}}^{(1)}$, where 
\[
\tilde{\bD}_{1,*}^{(1)}=(\bX_{\calS}^T\bX_{\calS}+\lambda_0\bI)^{-\frac{1}2}\bY_{1,*}^{(1)}(\bX_{\calS}^T\bX_{\calS}+\lambda_0\bI)^{-\frac{1}2}.
\]
By property \textbf{(P4)}, $\sigma_r(\tilde{\bD}_1^{(1)})$ is bounded from below by
\begin{align}\nonumber
\sigma_r(\tilde{\bD}_1^{(1)})\geq & \  \sigma_r(\tilde{\bD}_{1,\mathrm{lower}}^{(1)})
\\
\geq & \ \frac{\sigma_r^2(\bX_{\calS})}{\sigma_r^2(\bX_{\calS})+\lambda_0} (\sigma_r^2(\bX_{\calS})+(\sqrt{p-s}-\sqrt{s}-c_6)^2)\nonumber\\
\geq & \frac{\sigma_r^2(\bX_{\calS})}{2\lambda_0} (\sigma_r^2(\bX_{\calS})+(\sqrt{p-s}-\sqrt{s}-c_6)^2)\nonumber\\
\geq & \ \frac{c_5(\beta_r^2+1)n}{2\lambda_0}\Big(c_5(\beta_r^2+1)n+\frac{p-s}{4}\Big),\nonumber
\end{align}
where the last step uses  \eqref{eq:c7property}, and  $\sigma_r(\tilde{\bD}_1^{(1)})$ is bounded above by
\begin{align}\nonumber
\sigma_r(\tilde{\bD}_1^{(1)})\leq & \ \sigma_r(\tilde{\bD}_{1,\mathrm{upper}}^{(1)})
\\
\leq & \ \frac{\sigma_r^2(\bX_{\calS})}{\sigma_r^2(\bX_{\calS})+\lambda_0} (\sigma_r^2(\bX_{\calS})+(\sqrt{p-s}+\sqrt{s}+c_6)^2)\nonumber\\
\leq & \  \frac{\sigma_r^2(\bX_{\calS})}{\lambda_0} (\sigma_r^2(\bX_{\calS})+(\sqrt{p-s}-\sqrt{s}-c_6)^2) \nonumber\\
\leq & \ \frac{(4+c_5)(\beta_r^2+1)n}{\lambda_0}((4+c_5)(\beta_r^2+1)n+3p). \nonumber
\end{align}

\textit{Upper bound of $\|\tilde{\bD}_1^{(2)}\|$.} The proof is similar to the upper bound of $\sigma_r(\tilde{\bD}_1^{(1)})$ and we have
\[
\|\tilde{\bD}_1^{(2)}\|\leq \frac{\sigma_{r+1}^2(\bX_{\calS})}{\sigma_{r+1}^2(\bX_{\calS})+\lambda_0} (\sigma_{r+1}^2(\bX_{\calS})+(\sqrt{p-s}+\sqrt{s}+c_6)^2)\leq \frac{4n}{\lambda_0}(4n+3p).
\]

\textit{Lower bound of $\sigma_r(\tilde{\bD}_1^{(1)})-\|\tilde{\bD}_1^{(2)}\|$}. Applying the intermediate steps in the estimations of the lower bound of $\sigma_r(\tilde{\bD}_1^{(1)})$ and the upper bound of $\|\tilde{\bD}_1^{(2)}\|$, we have that	
\begin{align*}
\sigma_r(\tilde{\bD}_1^{(1)})-\|\tilde{\bD}_1^{(2)}\|
\geq & \frac{\sigma_r^4(\bX_{\calS})+(p-s+(\sqrt{s}+c_6)^2)\sigma_r^2(\bX_{\calS})}{\sigma_r^2(\bX_{\calS})+\lambda_0}\\
& \ - \frac{\sigma_{r+1}^4(\bX_{\calS})+(p-s+(\sqrt{s}+c_6)^2)\sigma_{r+1}^2(\bX_{\calS})}{\sigma_{r+1}^2(\bX_{\calS})+\lambda_0}\\& \ - 2\sqrt{p-s}(\sqrt{s}+c_6)\frac{\sigma_{r}^2(\bX_{\calS})+\sigma_{r+1}^2(\bX_{\calS})}{\lambda_0}, 
\end{align*}
and the first two terms are bound by
\begin{align}
&\frac{\sigma_r^4(\bX_{\calS})+(p-s+(\sqrt{s}+c_6)^2)\sigma_r^2(\bX_{\calS})}{\sigma_r^2(\bX_{\calS})+\lambda_0}\nonumber \\ & \ -\frac{\sigma_{r+1}^4(\bX_{\calS})+(p-s+(\sqrt{s}+c_6)^2)\sigma_{r+1}^2(\bX_{\calS})}{\sigma_{r+1}^2(\bX_{\calS})+\lambda_0}\nonumber\\
=& \ \frac{\lambda_0(p-s+(\sqrt{s}+c_6)^2+\sigma_{r}^2(\bX_{\calS})+\sigma_{r+1}^2(\bX_{\calS}))+\sigma_{r}^2(\bX_{\calS})\sigma_{r+1}^2(\bX_{\calS})}{(\sigma_r^2(\bX_{\calS})+\lambda_0)(\sigma_{r+1}^2(\bX_{\calS})+\lambda_0)}\nonumber\\
& (\sigma_{r}^2(\bX_{\calS})-\sigma_{r+1}^2(\bX_{\calS}))\nonumber\\
\geq& \ \frac{\lambda_0(p-s+(\sqrt{s}+c_6)^2+\sigma_{r}^2(\bX_{\calS}))}{(\sigma_r^2(\bX_{\calS})+\lambda_0)(\sigma_{r+1}^2(\bX_{\calS})+\lambda_0)}(\sigma_{r}^2(\bX_{\calS})-\sigma_{r+1}^2(\bX_{\calS}))\nonumber\\
\geq& \ \frac{p-s+(\sqrt{s}+c_6)^2+\sigma_{r}^2(\bX_{\calS})}{4\lambda_0}(\sigma_{r}^2(\bX_{\calS})-\sigma_{r+1}^2(\bX_{\calS}))\nonumber \\\geq& \ \frac{p-s+(\sqrt{s}+c_6)^2+c_5\beta_r^2n+n}{4\lambda_0}c_5\beta_r^2n\geq \frac{p+n}{4\lambda_0}c_5\beta_r^2n, \label{eq:upper_bound_6}
\end{align}    
and the third term is bounded by 
\begin{align*}
2\sqrt{p-s}(\sqrt{s}+c_6)\frac{\sigma_{r}^2(\bX_{\calS})+\sigma_{r+1}^2(\bX_{\calS})}{\lambda_0}\leq 
2(c_6+\sqrt{s})\sqrt{p}n\frac{5\beta_r^2+8}{\lambda_0}.
\end{align*}
As a result, applying the second component in  \eqref{eq:c7property}, we have
\[
\sigma_r(\tilde{\bD}_1^{(1)})-\|\tilde{\bD}_1^{(2)}\|\geq \frac{p+n}{8\lambda_0}c_5\beta_r^2n.
\]

\textit{Upper bound of $\|\tilde{\bD}_2\|$.} Note that $\bY_2$ can be viewed as a $2$ by $2$ block matrix with two blocks being zero and two blocks being $\bX_{\calS}^T\Pi_{\bX_{\calS,r}}\bX_{\calS^c}\bX_{\calS^c}^T\Pi_{\bX_{\calS,r,\perp}}\bX_{\calS}$ and its transpose, we have
\begin{align}\nonumber
\|\tilde{\bD}_2\|&\leq \frac{\|\bX_{\calS}^T\Pi_{\bX_{\calS,r}}\bX_{\calS^c}\bX_{\calS^c}^T\Pi_{\bX_{\calS,r,\perp}}\bX_{\calS}\|}{\lambda_0}\\
\nonumber
&\leq \frac{\|\bX_{\calS}\|\|\Pi_{\bX_{\calS,r}}\bX_{\calS^c}\bX_{\calS^c}^T\Pi_{\bX_{\calS,r,\perp}}\|\|\Pi_{\bX_{\calS,r,\perp}}\bX_{\calS}\|}{{\lambda_0}}\\
\nonumber&\leq  \   \frac{\|\bX_{\calS}\|\sigma_{r+1}(\bX_{\calS})}{\lambda_0}2\sqrt{p-s}(\sqrt{s}+c_6)\\
&\leq \frac{8n\sqrt{\beta_1^2+1}}{\lambda_0}(c_6+\sqrt{s})\sqrt{p}, \label{eq:bD2} 
	\end{align}
	where the last two inequalities follow from \eqref{eq:upper_r+1}, \textbf{(P3)} and \textbf{(P4)}.
	
\textit{Upper bound of $\|\Pi_{{\bD},r}-\Pi_{\bX_{\calS, r}}\|_F$} Applying Property 2 in Lemma~\ref{lemma:decompose}, a similar argument as in \eqref{eq:bD2} implies that
\begin{align}\nonumber
\|\tilde{\bD}_2\|_F&\leq \frac{\|\bX_{\calS}^T\Pi_{\bX_{\calS,r}}\bX_{\calS^c}\bX_{\calS^c}^T\Pi_{\bX_{\calS,r,\perp}}\bX_{\calS}\|_F}{\lambda_0}\\
\nonumber
&\leq \frac{\|\bX_{\calS}\|\|\Pi_{\bX_{\calS,r}}\bX_{\calS^c}\bX_{\calS^c}^T\Pi_{\bX_{\calS,r,\perp}}\|_F\|\Pi_{\bX_{\calS,r,\perp}}\bX_{\calS}\|}{{\lambda_0}}\\ 
\nonumber& \leq\   \frac{\|\bX_{\calS}\|\sigma_{r+1}(\bX_{\calS})}{\lambda_0}\sqrt{p-s}(\sqrt{r(s-r)}+2\sqrt{t })\\
&\leq \frac{4n\sqrt{\beta_1^2+1}}{\lambda_0}(\sqrt{r(s-r)}+2\sqrt{t })\sqrt{p}, \label{eq:bD2F} 
\end{align}

Combining \eqref{eq:upper_bound_6} and \eqref{eq:bD2F},	\begin{align}
\nonumber\|\Pi_{\tilde{\bD},r}-\Pi_{\bX_{\calS},r}\|_F&\leq \frac{\|\tilde{\bD}_2\|_F}{\sigma_r(\tilde{\bD}_1^{(1)})-\|\tilde{\bD}_1^{(2)}\|}\\
\nonumber &\leq \frac{16(\sqrt{r(s-r)}+2\sqrt{t }) \sqrt{\beta_1^2+1}\sqrt{p}}{c_5\beta_r^2(p+n)}\\&\leq \frac{8(\sqrt{r(s-r)}+2\sqrt{t }) \sqrt{\beta_1^2+1}}{c_5\beta_r^2\sqrt{n}}.\label{eq:eigenspacebound2}\end{align}

\textit{Lower bound of $\frac{\lambda_r(\bD)}{\lambda_{r+1}(\bD)}$.}
\begin{align*}
&\frac{\lambda_r(\bD)}{\lambda_{r+1}(\bD)}=\frac{\sigma_r(\tilde{\bD})}{\sigma_{r+1}(\tilde{\bD})}
\geq 1+ \frac{\sigma_r(\tilde{\bD}_1^{(1)})-\|\tilde{\bD}_1^{(2)}\|-2\|\tilde{\bD}_2\|}{\|\tilde{\bD}_1^{(2)}\|}\\
\geq &1+\frac{\frac{p+n}{4\lambda_0}c_5\beta_r^2n-2 \frac{8n\sqrt{\beta_1^2+1}}{\lambda_0}(c_6+\sqrt{s})\sqrt{p}}{\frac{4n}{\lambda_0}(4n+3p)}=1+\frac{(p+n)c_5\beta_r^2-64(c_6+\sqrt{s})\sqrt{(\beta_1^2+1)p}}{16(4n+3p)}\\
\geq & 1+\frac{(p+n)c_5\beta_r^2}{32(4n+3p)}\geq 1+c_0\beta_r^2,
\end{align*}
where the third inequality holds because of \eqref{eq:c7property} (see the first component in the RHS).\\
\textit{Lower bound of $\lambda_r(\bD)$.}
\begin{align}
&\lambda_r(\bD)=\sigma_r(\tilde{\bD})\geq \sigma_r(\tilde{\bD}^{(1)}_1)-\|\tilde{\bD}_2\|\nonumber\\\geq&
\frac{c_5(\beta_r^2+1)n}{2\lambda_0}\Big(c_5(\beta_r^2+1)n+\frac{p-s}{4}\Big)-\frac{8n\sqrt{\beta_1^2+1}}{\lambda_0}(c_6+\sqrt{s})\sqrt{p}\nonumber\\\geq&
\frac{c_5(\beta_r^2+1)n}{2\lambda_0}\Big(c_5(\beta_r^2+1)n+\frac{p-s}{4}\Big)-\frac{c_5\beta_r^2n (p+n)}{16\lambda_0}\nonumber\\
\geq &\frac{c_5(\beta_r^2+1)n}{2\lambda_0}\Big(c_5(\beta_r^2+1)n+\frac{p-2s}{8}-\frac{n}{8}\Big)\nonumber\\
\geq &\frac{c}{\lambda_0}(\beta_r^2+1)n\Big((\beta_r^2+1)n+p\Big),
\end{align}
where the last inequality holds since $c_5\geq 1/6$ and $p\geq 2s$.

\textit{Upper bound of $\|\Pi_{{\bD},r}-\Pi_{\bV}\|_F$}Combining \eqref{eq:eigenspacebound1} and \eqref{eq:eigenspacebound2},
\[
\|\Pi_{{\bD},r}-\Pi_{\bV}\|_F\leq  \frac{10(\sqrt{r(s-r)}+2\sqrt{t }) \sqrt{\beta_1^2+1}}{c_5\beta_r^2\sqrt{n}}.
\]
Recall $c_5\geq 1/6$, the bound of $\|\Pi_{{\bD},r}-\Pi_{\bV}\|_F$ in the lemma is obtained. Therefore, the proof of Lemma~\ref{lemma:prob} is complete. 
}

\end{proof}
\subsection{Auxiliary Lemmas}
	
\begin{lemma}\label{lemma:pertubation1}
	For any matrix $\bX\in\reals^{m\times p}$, $\bB\in\reals^{p\times r}$, and orthogonal matrix $\bV\in\reals^{p\times r}$ with $m>r$ and $p>r$,  
	\[
	\|\Pi_{\bX\bV}-\Pi_{\bX\bB}\|\leq\frac{\|\bX(\bI-\Pi_{\bV})\|}{\sigma_{r}(\bX\bV)}\|\Pi_{\bV}-\Pi_{\bB}\|.
	\]
	\end{lemma}
	\begin{proof}[Proof of Lemma~\ref{lemma:pertubation1}]
	Since $\|\Pi_{\bX\bV}-\Pi_{\bX\bB}\|$ is the sine of the principal angle between the column space of $\bX\bB$ and $\bX\bV$. As a result, it can be written as 
	$\|\Pi_{\bX\bV}-\Pi_{\bX\bB}\|=\max_{\bb\in\Sp(\bB)}\frac{\|\Pi_{\bX\bV}\bX\bb\|}{\|\bX\bb\|}= \max_{\bb\in\Sp(\bB)} \min_{\bv\in\reals^n}\sin \tan^{-1} \frac{\|\bX\bb-\bX\bv\|}{\|\bX\bv\|}$. Let $\bv=\Pi_{\bV}\bb$, then $$\frac{\|\bv-\bb\|}{\|\bb\|}\leq \tan\sin^{-1}\|\Pi_{\bV}-\Pi_{\bB}\|.$$   As a result,
	\begin{align*}&
	\|\Pi_{\bX\bV}-\Pi_{\bX\bB}\|\leq \max_{\bb\in\Sp(\bB)} \sin \tan^{-1} \frac{\|\bX\bb-\bX\bv\|}{\|\bX\bv\|}\leq  \sin \tan^{-1} \frac{\|\bX(\bI-\Pi_{\bV})(\bb-\bv)\|}{\|\bX\bv\|}
	\\\leq&  \frac{\|\bX(\bI-\Pi_{\bV})\|}{\sigma_r(\bX\Pi_{\bV})}\|\Pi_{\bV}-\Pi_{\bB}\|.
	\end{align*}
	Thus, the proof of Lemma~\ref{lemma:pertubation1} is complete. 
	\end{proof}
	
	\begin{lemma}\label{lemma:pertubation2}
	Assuming two subspaces $L_1, L_2\in\reals^D$,s a linear operator $\calA: \reals^D\rightarrow\reals^D$ represented by a matrix $\bA\in\reals^{N\times N}$, and $L_1'=\calA(L_1), L_2'\in\calA(L_2)$ be the subspaces after the operator. Then 
	\[
	\|\Pi_{L_1'}-\Pi_{L_2'}\|\leq \frac{\sigma_1(\bA)}{\sigma_D(\bA)}\|\Pi_{L_1}-\Pi_{L_2}\|
	\]
	\end{lemma}
	\begin{proof}[Proof of Lemma~\ref{lemma:pertubation2}]
	We follow  Section 3.2.1 of \cite{bbd7088358424c3391ca59a6b25d401b} and Section 6.4.3 of \cite{golub2013matrix} to use the theory of principal angles and obtain that
	$\|\Pi_{L_1}-\Pi_{L_2}\|=\sin(\theta(L_1,L_2))$, where $\theta(L_1,L_2)$ is the principal angle between $L_1$ and $L_2$ defined by $$\theta(L_1,L_2)=\max_{\bu\in L_1}\sin^{-1}(\frac{\mathrm{dist}(\bu, L_2)}{\|\bu\|}).$$
	As a result, 
	\[
	\|\Pi_{L_1}-\Pi_{L_2}\|=\max_{\bu\in L_1}\frac{\mathrm{dist}(\bu, L_2)}{\|\bu\|}.
	\]
	
Then for any $\bu\in L_1$, $$\frac{\mathrm{dist}(\bu, L_2)}{\|\bu\|}\leq \|\Pi_{L_1}-\Pi_{L_2}\|.$$ As a result, after the transformation $\calA$, we have that for any $\bu'\in L_1'$, $$\frac{\mathrm{dist}(\bu', L_2')}{\|\bu'\|}\leq \frac{\sigma_1(\bA)}{\sigma_D(\bA)}\|\Pi_{L_1}-\Pi_{L_2}\|.$$ That is,
\[
	\|\Pi_{L_1'}-\Pi_{L_2'}\|=\max_{\bu'\in L_1'}\frac{\mathrm{dist}(\bu', L_2')}{\|\bu'\|}\leq \frac{\sigma_1(\bA)}{\sigma_D(\bA)}\|\Pi_{L_1}-\Pi_{L_2}\|.
\]
Thus, the proof of Lemma~\ref{lemma:pertubation2} is complete. 
	\end{proof}
	\begin{lemma}\label{lemma:pertubation3}
	For any positive definite matrix $\bB\in\reals^{n\times n}$ and symmetric matrix $\bDelta\in\reals^{n\times n}$, the unitary factor of the polar decomposition in $\bB\exp(\bDelta)$, denoted by $\bU$, satisfies that
	\[
	\|\bU-\bI\|\leq C\log n\|\bDelta\|.
	\]
	This suggests that for any matrix $\bB$
	\[
	\bB-(\bB\exp(2\bDelta)\bB^T)^{1/2}=\bB\bZ,
	\]
	for some $\bZ$ such that $\|\bZ\|=\|\bI-\exp(\bDelta)\bU\|\leq \|\bI-\exp(\bDelta)\|+\|\exp(\bDelta)\|\|\bU-\bI\|\leq C(\log n+1)\|\bDelta\|$.
	\end{lemma}
	\begin{proof}[Proof of Lemma~\ref{lemma:pertubation3}]
	First, we will prove that it holds for small $\bDelta$. Since $\bB(\bI+\bDelta)\bU^T$ is symmetric, we have 
	\[
	\bB(\bI+\bDelta)\bU^T=\bU(\bI+\bDelta)\bB.
	\]
	Let $\bY$ be the skew-symmetric matrix such that $\bU=\bI+\bY+O(\|\bY\|_F^2)$, then we have
	\[
	\bB(\bI+\bDelta)(\bI-\bY)=(\bI+\bY)(\bI+\bDelta)\bB
	\]
	and
	\[
	\bB\bDelta-\bB\bY=\bY\bB+\bDelta\bB.
	\]
	WLOG assume that $\bB=\diag(b_1,\cdots, b_n)$, then we have
	\[
	\bY_{ij}=\frac{b_i-b_j}{b_i+b_j}\bDelta_{ij}.
	\]
	Applying Corollary 3.3 of \cite{Roy1993}, we have
	\[
	\|\bY\|\leq 2\gamma_n\|\Delta\|,
	\]
	where $\gamma_n=\frac{1}{n}\sum_{i=1}^n|\cot(2i-1)\pi/2n|$.
	
	Second, let $\bU(t)$ be the unitary factor of the polar decomposition in $\bB\exp(\bDelta)$, then the previous analysis implies that
	\[
	\|\bU(t+\epsilon)-\bU(t)\|\leq \epsilon\|\bDelta\|+O(\epsilon^2).
	\]
	Combining it for all $t=0,\epsilon, 2\epsilon,\cdots, 1-\epsilon$, let $\epsilon\rightarrow 0$, and note that $\gamma_n<C\log n$, so the desired result is obtained. Thus, the proof of Lemma~\ref{lemma:pertubation3} is complete. 
	\end{proof}
	
\section{Conclusion}\label{sec:conclusion}

This paper was motivated to solve the important
theoretical gap for providing theoretical guarantees for the popular SPCA algorithm proposed by \cite{doi:10.1198/106186006X113430}. On the one hand, we proved the guarantees of algorithmic convergence to a stationary point for the SPCA algorithm of \cite{doi:10.1198/106186006X113430}. On the other hand, we also proved that, under a sparse spiked covariance model, the SPCA algorithm of \cite{doi:10.1198/106186006X113430} can recover the principal subspace consistently under mild regularity conditions. We also showed that the estimation error bound can match the best available bounds of existing works or the minimax rates up to some logarithmic factors. Moreover, we {studied} {ITPS, a computationally more efficient variant of the SPA algorithm}, and provided its theoretical guarantees. The numerical performance of both algorithms  was demonstrated in simulation studies.

{
\bibliographystyle{agsm}
\bibliography{main-arxiv}

@article{10.1214/12-AOS1014,
author = {Aharon Birnbaum and Iain M. Johnstone and Boaz Nadler and Debashis Paul},
title = {{Minimax bounds for sparse PCA with noisy high-dimensional data}},
volume = {41},
journal = {The Annals of Statistics},
number = {3},
publisher = {Institute of Mathematical Statistics},
pages = {1055 -- 1084},
year = {2013}
}

@article{LU2016681,
	author = {Meng Lu and Jianhua Z. Huang and Xiaoning Qian},
	issn = {0031-3203},
	journal = {Pattern Recognition},
	pages = {681-691},
	title = {Sparse exponential family Principal Component Analysis},
	volume = {60},
	year = {2016}}

@article{10.1214/009117905000000233,
author = {Jinho Baik and Gérard Ben Arous and Sandrine Péché},
title = {{Phase transition of the largest eigenvalue for nonnull complex sample covariance matrices}},
volume = {33},
journal = {The Annals of Probability},
number = {5},
publisher = {Institute of Mathematical Statistics},
pages = {1643 -- 1697},
year = {2005}
}

@article{10.2307/24307692,
 ISSN = {10170405, 19968507},
 author = {Debashis Paul},
 journal = {Statistica Sinica},
 number = {4},
 pages = {1617--1642},
 publisher = {Institute of Statistical Science, Academia Sinica},
 title = {Asymptotics of sample eigenstructure for a large dimensional spiked covariance model},
 volume = {17},
 year = {2007}
}

@article{10.1214/08-AOS618,
author = {Boaz Nadler},
title = {{Finite sample approximation results for principal component analysis: A matrix perturbation approach}},
volume = {36},
journal = {The Annals of Statistics},
number = {6},
publisher = {Institute of Mathematical Statistics},
pages = {2791 -- 2817},
year = {2008}
}

@article{10.2307/25662226,
 author = {Sungkyu Jung and J. S. Marron},
 journal = {The Annals of Statistics},
 number = {6B},
 pages = {4104--4130},
 publisher = {Institute of Mathematical Statistics},
 title = {PCA CONSISTENCY IN HIGH DIMENSION, LOW SAMPLE SIZE CONTEXT},
 volume = {37},
 year = {2009}
}

@article{BAIK20061382,
title = {Eigenvalues of large sample covariance matrices of spiked population models},
journal = {Journal of Multivariate Analysis},
volume = {97},
number = {6},
pages = {1382-1408},
year = {2006},
issn = {0047-259X},
author = {Jinho Baik and Jack W. Silverstein}
}

@book{tao2012topics,
  title={Topics in Random Matrix Theory},
  author={Tao, T.},
  isbn={9780821874301},
  lccn={2011045194},
  series={Graduate Studies in Mathematics},
  year={2012},
  publisher={American Mathematical Society}
}

@article{10.1214/aos/1009210544,
author = {Iain M. Johnstone},
title = {{On the distribution of the largest eigenvalue in principal components analysis}},
volume = {29},
journal = {The Annals of Statistics},
number = {2},
publisher = {Institute of Mathematical Statistics},
pages = {295 -- 327},
year = {2001}
}

@book{golub2013matrix,
  title={Matrix Computations},
  author={Golub, G.H. and Van Loan, C.F.},
  isbn={9781421407944},
  lccn={2012943449},
  series={Johns Hopkins Studies in the Mathematical Sciences},
  year={2013},
  publisher={Johns Hopkins University Press}
}

@article{bbd7088358424c3391ca59a6b25d401b,
title = "lp-Recovery of the Most Significant Subspace Among Multiple Subspaces with Outliers",
author = "Gilad Lerman and Teng Zhang",
year = "2014",
month = dec,
language = "English (US)",
volume = "40",
pages = "329--385",
journal = "Constructive Approximation",
issn = "0176-4276",
publisher = "Springer New York",
number = "3"
}

@article{10.1214/aos/1015957395,
author = {B. Laurent and P. Massart},
title = {{Adaptive estimation of a quadratic functional by model selection}},
volume = {28},
journal = {The Annals of Statistics},
number = {5},
publisher = {Institute of Mathematical Statistics},
pages = {1302 -- 1338},
year = {2000}
}

@article{10.1214/13-AOS1151,
author = {Vincent Q. Vu and Jing Lei},
title = {{Minimax sparse principal subspace estimation in high dimensions}},
volume = {41},
journal = {The Annals of Statistics},
number = {6},
publisher = {Institute of Mathematical Statistics},
pages = {2905 -- 2947},
year = {2013}
}

@book{wainwright_2019, place={Cambridge}, series={Cambridge Series in Statistical and Probabilistic Mathematics}, title={High-Dimensional Statistics: A Non-Asymptotic Viewpoint}, publisher={Cambridge University Press}, author={Wainwright, Martin J.}, year={2019}, collection={Cambridge Series in Statistical and Probabilistic Mathematics}}

@article{chen2020alternating,
  title={An alternating manifold proximal gradient method for sparse principal component analysis and sparse canonical correlation analysis},
  author={Chen, Shixiang and Ma, Shiqian and Xue, Lingzhou and Zou, Hui},
  journal={INFORMS Journal on Optimization},
  volume={2},
  number={3},
  pages={192--208},
  year={2020},
  publisher={INFORMS}
}

@Article{Wedin1972,
author={Wedin, Per-{\AA}ke},
title={Perturbation bounds in connection with singular value decomposition},
journal={BIT Numerical Mathematics},
year={1972},
month={Mar},
day={01},
volume={12},
number={1},
pages={99-111}
}

@article{cai2013,
  title={Sparse PCA: Optimal rates and adaptive estimation},
  author={Cai, T Tony and Ma, Zongming and Wu, Yihong},
  journal={The Annals of Statistics},
  volume={41},
  number={6},
  pages={3074--3110},
  year={2013},
  publisher={Institute of Mathematical Statistics}
}

@INPROCEEDINGS{6875223,
  author={Y. {Deshpande} and A. {Montanari}},
  booktitle={2014 IEEE International Symposium on Information Theory}, 
  title={Information-theoretically optimal sparse PCA}, 
  year={2014},
  volume={},
  number={},
  pages={2197-2201}}

@article{krauthgamer2015,
author = "Krauthgamer, Robert and Nadler, Boaz and Vilenchik, Dan",
fjournal = "The Annals of Statistics",
journal = "The Annals of Statistics",
month = "06",
number = "3",
title ="Do semidefinite relaxations solve sparse PCA up to the information limit?",
pages = "1300--1322",
volume = "43",
year = "2015"
}

@inproceedings{NIPS2013_81e5f81d,
 author = {Vu, Vincent Q and Cho, Juhee and Lei, Jing and Rohe, Karl},
 booktitle = {Advances in Neural Information Processing Systems},
 pages = {},
 title = {Fantope Projection and Selection: A near-optimal convex relaxation of sparse PCA},
 volume = {26},
 year = {2013}
}

@incollection{NIPS2014_5406,
title = {Sparse PCA via Covariance Thresholding},
author = {Deshpande, Yash and Montanari, Andrea},
booktitle = {Advances in Neural Information Processing Systems 27},
editor = {Z. Ghahramani and M. Welling and C. Cortes and N. D. Lawrence and K. Q. Weinberger},
pages = {334--342},
year = {2014},
publisher = {Curran Associates, Inc.}
}

@article{jankova2018debiased,
  author={J. {Janková} and S. {van de Geer}},
  journal={IEEE Transactions on Information Theory}, 
  title={De-biased sparse PCA: Inference for eigenstructure of large covariance matrices}, 
  volume={67},
  number={4},
  pages={2507--2527},
  year={2021}}

@article{ma2013,
author = "Ma, Zongming",
fjournal = "The Annals of Statistics",
journal = "The Annals of Statistics",
month = "04",
number = "2",
pages = "772--801",
publisher = "The Institute of Mathematical Statistics",
title = "Sparse principal component analysis and iterative thresholding",
volume = "41",
year = "2013"
}

@article{lei2015,
author = "Lei, Jing and Vu, Vincent Q.",
fjournal = "The Annals of Statistics",
journal = "The Annals of Statistics",
month = "02",
number = "1",
pages = "299--322",
publisher = "The Institute of Mathematical Statistics",
title = "Sparsistency and agnostic inference in sparse PCA",
volume = "43",
year = "2015"
}

@article{amini2009,
author = "Amini, Arash A. and Wainwright, Martin J.",
fjournal = "The Annals of Statistics",
journal = "The Annals of Statistics",
month = "10",
number = "5B",
pages = "2877--2921",
publisher = "The Institute of Mathematical Statistics",
title = "High-dimensional analysis of semidefinite relaxations for sparse principal components",
volume = "37",
year = "2009"
}

@article{SHEN2013317,
title = "Consistency of sparse PCA in High Dimension, Low Sample Size contexts",
journal = "Journal of Multivariate Analysis",
volume = "115",
pages = "317 - 333",
year = "2013",
issn = "0047-259X",
author = "Dan Shen and Haipeng Shen and J.S. Marron"
}

@ARTICLE{8412518,
  author={H. {Zou} and L. {Xue}},
  journal={Proceedings of the IEEE}, 
  title={A Selective Overview of Sparse Principal Component Analysis}, 
  year={2018},
  volume={106},
  number={8},
  pages={1311-1320}}

@article{10.5555/1756006.1756021,
author = {Journ\'{e}e, Michel and Nesterov, Yurii and Richt\'{a}rik, Peter and Sepulchre, Rodolphe},
title = {Generalized Power Method for Sparse Principal Component Analysis},
year = {2010},
issue_date = {3/1/2010},
publisher = {JMLR.org},
volume = {11},
issn = {1532-4435},
journal = {Journal of Machine Learning Research},
pages = {517–553},
numpages = {37}
}

@article{10.5555/2567709.2502610,
author = {Yuan, Xiao-Tong and Zhang, Tong},
title = {Truncated Power Method for Sparse Eigenvalue Problems},
year = {2013},
issue_date = {January 2013},
publisher = {JMLR.org},
volume = {14},
number = {1},
issn = {1532-4435},
journal = {Journal of Machine Learning Research},
pages = {899–925},
numpages = {27}
}

@inproceedings{10.5555/2976248.2976363,
author = {Moghaddam, Baback and Weiss, Yair and Avidan, Shai},
title = {Spectral Bounds for Sparse PCA: Exact and Greedy Algorithms},
year = {2005},
publisher = {MIT Press},
address = {Cambridge, MA, USA},
booktitle = {Proceedings of the 18th International Conference on Neural Information Processing Systems},
pages = {915–922},
numpages = {8},
location = {Vancouver, British Columbia, Canada},
series = {NIPS'05}
}

@article{10.5555/1390681.1442775,
author = {d'Aspremont, Alexandre and Bach, Francis and Ghaoui, Laurent El},
title = {Optimal Solutions for Sparse Principal Component Analysis},
year = {2008},
issue_date = {6/1/2008},
publisher = {JMLR.org},
volume = {9},
issn = {1532-4435},
journal = {Journal of Machine Learning Research},
pages = {1269–1294},
numpages = {26}
}

@article{Witten+Hastie+Tibishirani,
  title={A penalized matrix decomposition, with applications to sparse principal components and canonical correlation analysis},
  author={Witten, Daniela M and Tibshirani, Robert and Hastie, Trevor},
  journal={Biostatistics},
  volume={10},
  number={3},
  pages={515--534},
  year={2009},
  publisher={Oxford University Press}
}

@Manual{enet,
    title = {elasticnet: Elastic-Net for Sparse Estimation and Sparse PCA.},
    author = {Hui Zou and Trevor Hastie},
    year = {2020},
    note = {R package version 1.3},
    url = {https://CRAN.R-project.org/package=elasticnet},
  }

@article{doi:10.1137/050645506,
author = {d'Aspremont, Alexandre and El Ghaoui, Laurent and Jordan, Michael I. and Lanckriet, Gert R. G.},
title = {A Direct Formulation for Sparse PCA Using Semidefinite Programming},
journal = {SIAM Review},
volume = {49},
number = {3},
pages = {434-448},
year = {2007}
}

@article{attouch2010proximal,
  title={Proximal alternating minimization and projection methods for nonconvex problems: An approach based on the Kurdyka-{\L}ojasiewicz inequality},
  author={Attouch, H{\'e}dy and Bolte, J{\'e}r{\^o}me and Redont, Patrick and Soubeyran, Antoine},
  journal={Mathematics of Operations Research},
  volume={35},
  number={2},
  pages={438--457},
  year={2010},
  publisher={INFORMS}
}

@article{zou2005regularization,
  title={Regularization and variable selection via the elastic net},
  author={Zou, Hui and Hastie, Trevor},
  journal={Journal of the Royal Statistical Society: Series B (Statistical Methodology)},
  volume={67},
  number={2},
  pages={301--320},
  year={2005},
  publisher={Wiley Online Library}
}

@article{doi:10.1198/jasa.2009.0121,
author = {Iain M. Johnstone and Arthur Yu Lu},
title = {On Consistency and Sparsity for Principal Components Analysis in High Dimensions},
journal = {Journal of the American Statistical Association},
volume = {104},
number = {486},
pages = {682-693},
year  = {2009},
publisher = {Taylor & Francis}
}

@article{10.2307/1391037,
 author = {Ian T. Jolliffe and Nickolay T. Trendafilov and Mudassir Uddin},
 journal = {Journal of Computational and Graphical Statistics},
 number = {3},
 pages = {531--547},
 publisher = {[American Statistical Association, Taylor & Francis, Ltd., Institute of Mathematical Statistics, Interface Foundation of America]},
 title = {A Modified Principal Component Technique Based on the LASSO},
 volume = {12},
 year = {2003}
}

@article{Roy1993,
author = {Mathias, Roy},
title = {The Hadamard Operator Norm of a Circulant and Applications},
journal = {SIAM Journal on Matrix Analysis and Applications},
volume = {14},
number = {4},
pages = {1152-1167},
year = {1993}
}

@article{doi:10.1198/106186006X113430,
author = {Hui Zou and Trevor Hastie and Robert Tibshirani},
title = {Sparse Principal Component Analysis},
journal = {Journal of Computational and Graphical Statistics},
volume = {15},
number = {2},
pages = {265-286},
year  = {2006},
publisher = {Taylor & Francis}
}

@article{lee2010sparse,
  title={Sparse logistic principal components analysis for binary data},
  author={Lee, Seokho and Huang, Jianhua Z and Hu, Jianhua},
  journal={The Annals of Applied Statistics},
  volume={4},
  number={3},
  pages={1579--1601},
  year={2010},
  publisher={JSTOR}
}

@ARTICLE{2012arXiv1202.1242P,
       author = {{Paul}, Debashis and {Johnstone}, Iain M.},
        title = "{Augmented sparse principal component analysis for high dimensional data}",
      journal = {arXiv e-prints},
         year = 2012,
          eid = {arXiv:1202.1242},
        pages = {arXiv:1202.1242},
archivePrefix = {arXiv},
       eprint = {1202.1242},
 primaryClass = {math.ST},
       adsurl = {https://ui.adsabs.harvard.edu/abs/2012arXiv1202.1242P}
}

@inproceedings{NIPS2014_74563ba2,
 author = {Wang, Zhaoran and Lu, Huanran and Liu, Han},
 booktitle = {Advances in Neural Information Processing Systems},
 editor = {Z. Ghahramani and M. Welling and C. Cortes and N. Lawrence and K.Q. Weinberger},
 title = {Tighten after Relax: Minimax-Optimal Sparse PCA in Polynomial Time},
 volume = {27},
 pages ={3383-3391},
 year = {2014}
}
}

 \end{document}